\documentclass[11pt]{article}%
\usepackage{graphicx}
\usepackage{amsmath}
\usepackage{amsfonts}
\usepackage{amssymb}%
\providecommand{\U}[1]{\protect\rule{.1in}{.1in}}
\newtheorem{theorem}{Theorem}[section]

\newtheorem{corollary}[theorem]{Corollary}

\newtheorem{definition}[theorem]{Definition}

\newtheorem{lemma}[theorem]{Lemma}

\newtheorem{proposition}[theorem]{Proposition}
\newtheorem{remark}[theorem]{Remark}

\newenvironment{proof}[1][Proof]{\textbf{#1.} }{\hfill\rule{0.5em}{0.5em}}
{\catcode`\@=11\global\let\AddToReset=\@addtoreset
\AddToReset{equation}{section}

\AddToReset{theorem}{section}

\begin{document}

\title{On the connection between two quasilinear elliptic problems with source terms
of order 0 or 1}
\author{Haydar ABDEL HAMID\thanks{Laboratoire de Math\'{e}matiques et Physique
Th\'{e}orique, CNRS UMR 6083, Facult\'{e} des Sciences, 37200 Tours France.
E-mail address:veronmf@univ-tours.fr}
\and Marie Fran\c{c}oise BIDAUT-VERON\thanks{Laboratoire de Math\'{e}matiques et
Physique Th\'{e}orique, CNRS UMR 6083, Facult\'{e} des Sciences, 37200 Tours
France. E-mail address:abdelham@lmpt.univ-tours.fr}}
\maketitle

\begin{abstract}
We establish a precise connection between two elliptic quasilinear problems
with Dirichlet data in a bounded domain of $\mathbb{R}^{N}.$ The first one, of
the form
\[
-\Delta_{p}u=\beta(u)\left\vert \nabla u\right\vert ^{p}+\lambda f(x)+\alpha,
\]
involves a source gradient term with natural growth, where $\beta$ is
nonnegative, $\lambda>0,f(x)\geqq0$, and $\alpha$ is a nonnegative measure.
The second one, of the form
\[
-\Delta_{p}v=\lambda f(x)(1+g(v))^{p-1}+\mu,
\]
presents a source term of order $0,~$where $g$ is nondecreasing, and $\mu$ is
a nonnegative measure. Here $\beta$ and $g$ can present an asymptote. The
correlation gives new results of existence, nonexistence, regularity and
multiplicity of the solutions for the two problems, without or with measures.
New informations on the extremal solutions are given when $g$ is superlinear.

\end{abstract}
\tableofcontents

.\pagebreak\medskip

\section{Introduction}

Let $\Omega$ be a smooth bounded domain in $\mathbb{R}^{N}(N\geqq2)$ and
$1<p\leqq N$. In this paper we compare two quasilinear Dirichlet problems.
\medskip

The first one presents a source gradient term with a natural growth:%
\begin{equation}
-\Delta_{p}u=\beta(u)\left\vert \nabla u\right\vert ^{p}+\lambda
f(x)\hspace{0.5cm}\text{in }\Omega,\qquad u=0\quad\text{on }\partial\Omega,
\tag{PU$\lambda$}%
\end{equation}
where%
\begin{equation}
\beta\text{ }\in C^{0}(\left[  0,L\right)  ),\text{ }L\leqq\infty,\text{ and
}\beta\text{ is nonnegative, }\beta\not \equiv 0. \label{hypb}%
\end{equation}
and $\lambda>0$ is a given real, and
\[
f\in L^{1}(\Omega),\quad f\geqq0\text{ a.e. in }\Omega.
\]
The function $\beta$ can have an asymptote at point $L,$ and is not supposed
to be increasing. For some results we suppose that $f$ belongs to suitable
spaces $L^{r}(\Omega),r>1.$\medskip

The second problem involves a source term of order $0,$ with the same
$\lambda$ and $f:$%
\begin{equation}
-\Delta_{p}v=\lambda f(x)(1+g(v))^{p-1}\hspace{0.5cm}\text{in }\Omega,\qquad
v=0\quad\text{on }\partial\Omega. \tag{PV$\lambda$}%
\end{equation}
where%
\begin{equation}
g\in C^{1}(\left[  0,\Lambda\right)  ),\text{ }\Lambda\leqq\infty,\text{
}g(0)=0\text{ and }g\text{ is nondecreasing, }g\not \equiv 0. \label{hypo}%
\end{equation}
Here also $g$ can have an asymptote. In some cases where $\Lambda=\infty,$we
make a growth condition on $g$ of the form%

\begin{equation}
M_{Q}=\lim\sup_{\tau\longrightarrow\infty}\frac{g(\tau)^{p-1}}{\tau^{Q}%
}<\infty\label{hmq}%
\end{equation}
for some $Q>0,$ and setting $p^{\ast}=Np/(N-p),$ discuss according to the
position of $Q$ with respect to $p-1$ and
\[
Q_{1}=\frac{N(p-1)}{N-p},\qquad Q^{\ast}=p^{\ast}-1=\frac{N(p-1)+p}%
{N-p},\qquad(Q_{1}=Q^{\ast}=\infty\text{ if }p=N).
\]

Problem (PU$\lambda)$ has been studied by many authors. Among them, let us
mention the results of \cite{BoMuPu}, \cite{BoSeTr} for the case $p=2$,
\cite{FeMu}, \cite{FeMu2} for general quasilinear operators, when $\beta$ is
defined on $\mathbb{R},$ not necessarily positive, but bounded. Problem
(PU$\lambda$) has been studied in \cite{AAP} for $p=2$ and more general
$\beta$, defined on $\left[  0,\infty\right)  $, such that $\underline{\lim
}_{t\rightarrow\infty}\beta(t)>0,$ see also many references therein. For
general $p>1,$ the problem has been investigated in \cite{Por} in the
absorption case where $\beta(t)\leqq0$ with measure data, and in \cite{PorSe}
with a signed $\beta,$ with strong growth assumptions on $\left\vert
\beta\right\vert $.\medskip

Problem (PV$\lambda)$ is also the object of a very rich litterature for
$\Lambda=\infty,$ especially when $g$ is superlinear, and convex, $p=2,$ and
$f\in L^{\infty}(\Omega)$. Here a main question is to give the range of
$\lambda$ for which there exists at least one variational solution $v\in
W_{0}^{1,p}(\Omega),$ or for which there exists a minimal bounded solution,
and to get regularity properties of the limit of these solutions, called
extremal solutions. For $p=2,$ the case of the exponential $g(v)=e^{v}-1$ or
of a power $g(v)=v^{q}$ has been studied first, see \cite{CrRa}, \cite{MiPu},
and the general case was investigated in \cite{BrCMR}, \cite{BrVa}. The
regularity $L^{\infty}(\Omega)$ of the extremal solutions is also intensively
discussed in many works, see \cite{Ca} and references therein. Extensions to
general $p,$ are given in \cite{AzPe}, \cite{AzPePu}, \cite{Fe}, \cite{Ca} and
\cite{CaSa}, \cite{CaCapSa}. A second question is the existence of a second
solution when $g$ is subcritical with respect to the Sobolev exponent. It has
been obtained for power-type nonlinearities of type concave-convex, see
\cite{AmBrCe}, \cite{AzPeMa}, and \cite{AAP} for general convex function $g$
and $p=2,$ and some results are given in \cite{Fe} for a power and
$p>1.$\medskip

It is well known that a suitable change of variables problem (PU$\lambda$)
leads formally to problem (PV$\lambda$), at least when $L=\infty$. Suppose for
example that $\beta$ is a constant, that we can fix to $p-1:$
\begin{equation}
-\Delta_{p}u=(p-1)\left\vert \nabla u\right\vert ^{p}+\lambda f(x)\hspace
{0.5cm}\text{in }\Omega,\qquad u=0\quad\text{on }\partial\Omega. \label{M}%
\end{equation}
Setting $v=e^{u}-1$ leads formally to the problem
\begin{equation}
-\Delta_{p}v=\lambda f(x)(1+v)^{p-1}\hspace{0.5cm}\text{in }\Omega,\qquad
v=0\quad\text{on }\partial\Omega.\hspace{0.5cm} \label{E5}%
\end{equation}
and we can return from $v$ to $u$ by $u=\ln(1+v).$ However an example, due to
\cite{FeMu2}, shows that the correspondence is more complex: assuming $f=0$
and $\Omega=B(0,1),$ $p<N,$ equation \ref{M} admits the solution $u_{0}%
\equiv0,$ corresponding to $v_{0}\equiv0$; but it has also an infinity of
solutions:
\begin{equation}
u_{m}(x)={\ln}\left(  (1-m)^{-1}(\left\vert x\right\vert ^{-(N-p)/(p-1)}%
-m\right)  ), \label{um}%
\end{equation}
defined for any $m\in\left(  0,1\right)  $, and $v_{m}=e^{u_{m}}-1$ satisfies
\[
-\Delta_{p}v_{m}=K_{m,N}\delta_{0}\hspace{0.5cm}\text{in }\mathcal{D}^{\prime
}\left(  \Omega\right)  ,
\]
where $\delta_{0}$ is the Dirac mass concentrated at 0, and $K_{m,N}>0$, thus
$v_{m}\not \in W_{0}^{1,p}(\Omega).$ Observe that $u_{m}\in W_{0}^{1,p}%
(\Omega)$ and it solves problem (\ref{M}) in $\mathcal{D}^{\prime}\left(
\Omega\right)  $. Indeed the logarithmic singularity at $0$ is not seen in
$\mathcal{D}^{\prime}\left(  \Omega\right)  .\medskip$

In the case of a general $\beta,$ the change of unknown in (PU$\lambda$)
\begin{equation}
v(x)=\Psi(u(x))=\int_{0}^{u(x)}e^{\gamma(\theta)/(p-1)}d\theta,\text{ \quad
where }\gamma(t)=\int_{0}^{t}\beta(\theta)d\theta, \label{psiu}%
\end{equation}
leads formally to problem (PV$\lambda$), where $\Lambda=\Psi(L)$ and $g$ is
given by
\begin{equation}
g(v)=e^{\gamma(\Psi^{-1}(v))/(p-1)}-1=\frac{1}{p-1}\int_{0}^{v}\beta
(\Psi(s))ds. \label{gbt}%
\end{equation}
It is apparently less used the converse correspondence, even in the case $p=2$
: \textit{for any function} $g$ satisfying (\ref{hypo}), the change of
unknown
\begin{equation}
u(x)=H(v(x))=\int_{0}^{v(x)}\frac{ds}{1+g(s)} \label{hdv}%
\end{equation}
leads formally to problem (PU$\lambda$), where $\beta$ satisfies (\ref{hypb})
with $L=H(\Lambda)$; indeed $H=\Psi^{-1}.$ And $\beta$ is linked to $g$ by
relation (\ref{gbt}), in other words
\begin{equation}
\beta(u)=(p-1)g^{\prime}(v)=(p-1)g^{\prime}(\Psi(u)). \label{bdu}%
\end{equation}
As a consequence, $\beta$ is \textit{nondecreasing} if and only if $g$ is
\textit{convex}. Also the interval $\left[  0,L\right)  $ of definition of
$\beta$ is finite if and only if $1/(1+g)\in L^{1}\left(  0,\Lambda\right)  .$
Some particular $\beta$ correspond to well known equations in $v,$ where the
main interesting ones are
\[
-\Delta_{p}v=\lambda fe^{v},\qquad-\Delta_{p}v=\lambda f(1+v)^{Q},\;Q>p-1,
\]
where $\beta$ has an \textit{asymptote}, or
\[
-\Delta_{p}v=\lambda f(1+v)^{Q},\;Q<p-1,\qquad-\Delta_{p}v=\lambda
f(1+v)(1+\ln(1+v))^{p-1},
\]
where $\beta$ is defined on $\left[  0,\infty\right)  .\medskip$

Our aim is to precise the connection between problems (PU$\lambda$) and
problem (PV$\lambda$), with possible measure data. As we see below, it allows
to obtain new existence or nonexistence or multiplicity results, not only for
problem (PU$\lambda$) but \textit{also for problem }(PV$\lambda$%
).$\medskip\medskip$

In Section \ref{Reg}, we recall the notions of renormalized or reachable
solutions, of problem
\[
-\Delta_{p}U=\mu\hspace{0.5cm}\text{in }\Omega,\qquad U=0\quad\text{on
}\partial\Omega,
\]
when $\mu$ is a measure in $\Omega.$ We give new regularity results when
$\mu=F\in L^{m}(\Omega)$ for some $m>1,$ see Lemma \ref{boot}, or local
estimates when $F\in L_{loc}^{1}(\Omega)$, see Lemma \ref{secm}, or when $F$
depends on $U,$ see Proposition \ref{cig}.\medskip

In Section \ref{Cor} we prove the following correlation theorem between $u$
and $v$. We denote by $\mathcal{M}_{b}(\Omega)$ the set of bounded Radon
measures, $\mathcal{M}_{s}(\Omega)$ the subset of measures concentrated on a
set of $p$-capacity 0, called singular; and $\mathcal{M}_{b}^{+}(\Omega)$ and
$\mathcal{M}_{s}^{+}(\Omega)$ are the subsets on nonnegative ones.\medskip

\begin{theorem}
\label{TP}(i) Let $g$ be any function satisfying (\ref{hypo}). Let $v$ be any
\textbf{renormalized} solution of problem
\begin{equation}
-\Delta_{p}v=\lambda f(x)(1+g(v))^{p-1}+\mu_{s}\hspace{0.5cm}\text{in }%
\Omega,\qquad v=0\quad\text{on }\partial\Omega, \label{PS}%
\end{equation}
such that $0\leqq v(x)<\Lambda$ a.e. in $\Omega,$ where $\mu_{s}\in
\mathcal{M}_{s}^{+}(\Omega)$. Then there exists $\alpha_{s}\in\mathcal{M}%
_{s}^{+}(\Omega)$ , such that $u=H(v)$ is a \textbf{renormalized} solution of
problem
\begin{equation}
-\Delta_{p}u=\beta(u)\left\vert \nabla u\right\vert ^{p}+\lambda
f(x)+\alpha_{s}\hspace{0.5cm}\text{in }\Omega,\qquad u=0\quad\text{on
}\partial\Omega, \label{PA}%
\end{equation}

Moreover if $\mu_{s}=0,$ then $\alpha_{s}=0.$ If $\Lambda<\infty,$ then
$\mu_{s}=\alpha_{s}=0$ and $u,v\in W_{0}^{1,p}(\Omega)\cap L^{\infty}\left(
\Omega\right)  .$ If $L<\infty=\Lambda,$ then $\alpha_{s}=0$ and $u\in
W_{0}^{1,p}(\Omega)\cap L^{\infty}\left(  \Omega\right)  .$ If $L=\infty
=\Lambda$ and g is unbounded$,$ then $\alpha_{s}=0;$ if g is bounded, then
$\alpha_{s}=(1+g(\infty))^{1-p}$ $\mu_{s}.$\medskip

(ii) Let $\beta$ be any function satisfying (\ref{hypb}). Let $u$ be any
\textbf{renormalized} solution of problem (\ref{PA}), such that $0\leqq
u(x)<L$ a.e. in $\Omega$, where $\alpha_{s}$ $\in\mathcal{M}_{s}^{+}(\Omega)$.
Then there exists $\mu\in\mathcal{M}^{+}(\Omega),$ such that $v=\Psi(u)$ is a
\textbf{reachable} solution of problem
\begin{equation}
-\Delta_{p}v=\lambda f(x)(1+g(v))^{p-1}+\mu\hspace{0.5cm}\text{in }%
\Omega,\qquad v=0\quad\text{on }\partial\Omega; \label{PSG}%
\end{equation}
hence the equation holds in $\mathcal{D}^{\prime}\left(  \Omega\right)  )$ and
more precisely, for any $h\in W^{1,\infty}(\mathbb{R})$ such that $h^{\prime}$
has a compact support, and any $\varphi\in\mathcal{D}(\Omega),$
\begin{equation}
\int_{\Omega}\left\vert \nabla v\right\vert ^{p-2}\nabla v.\nabla
(h(v)\varphi)dx=\int_{\Omega}h(v)\varphi\lambda f(x)(1+g(v))^{p-1}%
dx+h(\infty)\int_{\Omega}\varphi d\mu. \label{hol}%
\end{equation}
Moreover if $L<\infty,$ then $\alpha_{s}=0$ and $u\in W_{0}^{1,p}(\Omega)\cap
L^{\infty}\left(  \Omega\right)  .$ If $\Lambda<\infty,$ then $\alpha_{s}%
=\mu=0$ and $u,v\in W_{0}^{1,p}(\Omega)\cap L^{\infty}\left(  \Omega\right)
.$ If $L=\infty$ and $\beta\not \in L^{1}((0,\infty)),$ then $\alpha_{s}=0;$
if $\beta\in L^{1}((0,\infty)),$ then $\mu=e^{\gamma(\infty)}\alpha_{s}$ is
singular, and $v$ is a renormalized solution. If $p=2,$ or $p=N,$ then in any
case $\mu$ is singular.
\end{theorem}

This theorem precises and extends the results of \cite[Theorems 4.2 and
4.3]{AAP} where $p=2$ and $\beta$ is defined on $\left[  0,\infty\right)  $
and bounded from below near $\infty.$ The proofs are different, based on
\textbf{the equations satisfied by the truncations} of $u$ and $v$. The fact
that $\alpha_{s}=0$ whenever $\beta\not \in L^{1}((0,\infty))$ also improves
some results of \cite{Por}. In all the sequel we assume $f\not \equiv
0.$\medskip

In Section \ref{lin} we study the case $\beta$ constant, which means $g$
linear. The existence is linked to an eigenvalue problem with the weight $f,$
\begin{equation}
-\Delta_{p}w=\lambda f(x)\left\vert w\right\vert ^{p-2}w\hspace{0.5cm}\text{in
}\Omega,\qquad w=0\quad\text{on }\partial\Omega, \label{E15}%
\end{equation}
hence to the first eigenvalue
\begin{equation}
\lambda_{1}(f)=\inf\left\{  (\int_{\Omega}\left\vert \nabla w\right\vert
^{p}dx)/(\int_{\Omega}f\left\vert w\right\vert ^{p}dx):w\in W_{0}^{1,p}%
(\Omega)\backslash\left\{  0\right\}  \right\}  . \label{VP}%
\end{equation}

\begin{theorem}
\label{T2}Assume that $\beta(u)\equiv p-1,$ or equivalently $g(v)=v.$\medskip

\noindent(i) If $0<\lambda<\lambda_{1}(f)$ there exists a unique solution
$v_{0}\in W_{0}^{1,p}(\Omega)$ to (\ref{E5}), and then a \textbf{unique}
solution $u_{0}\in W_{0}^{1,p}(\Omega)$ to (\ref{M}) such that $e^{u_{0}}-1\in
W_{0}^{1,p}(\Omega)$. If $f\in L^{N/p}(\Omega),$ then $u_{0},v_{0}\in
L^{k}(\Omega)$ for any $k>1.$ If $f\in L^{r}(\Omega),r>N/p,$ then $u_{0}$ and
$v_{0}\in L^{\infty}(\Omega).$\medskip

Moreover, if $f\in L^{r}(\Omega),r>N/p,$ then for any measure $\mu_{s}%
\in\mathcal{M}_{s}^{+}(\Omega)$, there exists a renormalized solution $v_{s}$
of%
\begin{equation}
-\Delta_{p}v_{s}=\lambda f(x)(1+v_{s})^{p-1}+\mu_{s}\hspace{0.5cm}\text{in
}\Omega,\qquad v_{s}=0\quad\text{on }\partial\Omega; \label{pps}%
\end{equation}
thus there exists an \textbf{infinity} of solutions $u_{s}=\ln(1+v_{s})\in
W_{0}^{1,p}(\Omega)$ of (\ref{M}), less regular than $u_{0}.$\medskip

\noindent(ii) If $\lambda>\lambda_{1}(f)\geqq0,$ or $\lambda=\lambda_{1}(f)>0$
and $f\in L^{N/p}(\Omega),p<N,$ then (\ref{M}), (\ref{E5}) and (\ref{pps})
admit no renormalized solution.\medskip
\end{theorem}

In Section \ref{PVW} we study the existence of solutions of the problem
(PV$\lambda)$ for general $g$ without measures. It is easy to show that the
set of $\lambda$ for which there exists a solution in $W_{0}^{1,p}(\Omega)$ is
an interval $\left[  0,\lambda^{\ast}\right)  $ and the set of $\lambda$ for
which there exists a minimal solution $\underline{v}_{\lambda}\in W_{0}%
^{1,p}(\Omega)\cap L^{\infty}(\Omega)$ such that $\left\Vert \underline
{v}_{\lambda}\right\Vert _{L^{\infty}(\Omega)}<\Lambda$ is an interval
$\left[  0,\lambda_{b}\right)  .$\medskip

The first important question is to know if\textit{ }$\lambda_{b}=\lambda
^{\ast}.$ One of the main results of this article is the extension of the
well-known result of \cite{BrCMR} relative to the case $p=2,$ improving also a
result of \cite{CaSa} for $p>1.$

\begin{theorem}
\label{truc}Assume that $g$ satisfies (\ref{hypo}) and $g$ is convex near
$\Lambda$, and $f\in L^{r}\left(  \Omega\right)  ,r>N/p$. There exists a real
$\lambda^{\ast}>0$ such that\medskip

if $\lambda\in\left(  0,\lambda^{\ast}\right)  $ there exists a minimal
\textbf{bounded} solution $\underline{v}_{\lambda}$ such $\left\Vert
\underline{v}_{\lambda}\right\Vert _{L^{\infty}\left(  \Omega\right)
}<\Lambda.$\medskip

if $\lambda>\lambda^{\ast}$ there exists no \textbf{renormalized} solution. In
particular it holds $\lambda_{b}=\lambda^{\ast}.$
\end{theorem}

\noindent Thus for $\lambda>\lambda^{\ast},$ not only there cannot exist
variational solutions but also there cannot exist \textit{renormalized}
solutions, which is new for $p\neq2.$ It is noteworthy that the proof
\textit{uses problem} (PU$\lambda)$ and \textit{is based on Theorem} \ref{TP}.
A more general result is given at Theorem \ref{impo}.\medskip

When $\Lambda=\infty$ and $\lambda_{b}<\infty,$ a second question is the
regularity of the extremal function defined by $v^{\ast}=\lim_{\lambda
\nearrow\lambda_{b}}\underline{v}_{\lambda}$. Is it a solution of the limit
problem, and in what sense? Is it variational, is it bounded? Under convexity
assumptions we extend some results of \cite{Ne} , \cite{Sa2} and \cite{AAP}:

\begin{theorem}
\label{trema}Assume that $g$ satisfies (\ref{hypo}) with $\Lambda=\infty$ and
$\lim_{t\longrightarrow\infty}$ $g(t)/t=\infty,$ and $g$ is convex near
$\infty$; and $f\in L^{r}\left(  \Omega\right)  ,r>N/p.$ Then the extremal
function $v^{\ast}=\lim_{\lambda\nearrow\lambda^{\ast}}\underline{v}_{\lambda
}$ is a \textbf{renormalized} solution of (PV$\lambda^{\ast}$).
Moreover\medskip

(i) If $N<p(1+p^{\prime})/(1+p^{\prime}/r)$, then $v^{\ast}$ $\in W_{0}%
^{1,p}(\Omega).$ If $N<pp^{\prime}/(1+1/(p-1)r),$ then $v^{\ast}$ $\in
W_{0}^{1,p}(\Omega)\cap L^{\infty}\left(  \Omega\right)  .$\medskip

(ii) If (\ref{hmq}) holds with $Q<Q_{1},$ and $f\in L^{r}(\Omega)$ with
$Qr^{\prime}<Q_{1},$ or if (\ref{hmq}) holds with $Q<Q^{\ast},$ and $f\in
L^{r}(\Omega)$ with $(Q+1)r^{\prime}<p^{\ast.},$ then $v^{\ast}$ $\in
W_{0}^{1,p}(\Omega)\cap L^{\infty}\left(  \Omega\right)  .$
\end{theorem}

\noindent The proof follows from Theorem \ref{much}, Propositions
\ref{Regned}, \ref{vlan} and \ref{flu}. Without assumption of convexity on
$g,$ we obtain local results, see Theorem \ref{ploc}, based on regularity
results of \cite{B-VPo} and Harnack inequality.\medskip

When $\Lambda=\infty$ another question is the multiplicity of the variational
solutions when $g$ is subcritical with respect\ to the Sobolev exponent. We
prove the existence of at least two variational solutions in the following cases:

\begin{theorem}
\label{main}Suppose that $g$ is defined on $\left[  0,\infty\right)  ,$ and
$\lim_{t\longrightarrow\infty}$ $g(t)/t=\infty,$ and that growth condition
(\ref{hmq}) holds with $Q<Q^{\ast}$, and $f\in L^{r}(\Omega)$ with
$(Q+1)r^{\prime}<p^{\ast}.$ Then\medskip

\noindent(i) if $g$ is convex near $\infty,$ there exists $\lambda_{0}>0$ such
that for any $\lambda<\lambda_{0},$ there exists at least \textbf{two}
\textbf{solutions} $v\in W_{0}^{1,p}(\Omega)\cap L^{\infty}(\Omega)$ of
(PV$\lambda$).\medskip

\noindent(ii) If $p=2$ and $g$ is convex, or if $g$ satisfies the
Ambrosetti-Rabinowitz condition (\ref{AR}) and $f\in L^{\infty}(\Omega)$, then
\textbf{for any} $\lambda\in\left[  0,\lambda^{\ast}\right)  $ there exists at
least two solutions $v\in W_{0}^{1,p}(\Omega)\cap L^{\infty}(\Omega)$ of
(PV$\lambda$).
\end{theorem}

This result is new even for $p=2,$ improving results of \cite{AAP} where the
constraints on $g$ are stronger, and simplifying the proofs$.$ In case $p>1$
and $g$ is of power-type, it solves the conjecture of \cite{Fe} that
$\lambda_{0}=\lambda^{\ast}$.$\medskip$

In Section \ref{vmeas} we study the existence for problem (PV$\lambda)$ with
measures, which requires a stronger growth assumption: (\ref{hmq}) with
$Q<Q_{1}:$

\begin{theorem}
\label{meas}Suppose that $g$ is defined on $\left[  0,\infty\right)  ,$ and
$f\in L^{r}(\Omega)$ with $r>N/p.$ Let $\mu\in\mathcal{M}_{b}^{+}(\Omega)$ be
arbitrary.\medskip

\noindent(i) Assume (\ref{hmq}) with $Q=p-1$ and$\quad M_{p-1}\lambda
<\lambda_{1}(f),$ or with $Q<p-1$ and $Qr^{\prime}<Q_{1}.$ Then problem
\[
-\Delta_{p}v=\lambda f(x)(1+g(v))^{p-1}+\mu\hspace{0.5cm}\text{in }%
\Omega,\qquad v=0\quad\text{on }\partial\Omega,
\]
admits a renormalized solution.\medskip

\noindent(ii) Assume (\ref{hmq}) with $Q\in\left(  p-1,Q_{1}\right)  $ and
$Qr^{\prime}<Q_{1}.$ The same result is true if $\lambda$ and $\left\vert
\mu\right\vert (\Omega)$ are small enough.
\end{theorem}

More generally we give existence results for problems where the unknown $U$
may be signed, of the form
\[
-\Delta_{p}U=\lambda h(x,U)+\mu\hspace{0.5cm}\text{in }\Omega,\qquad
U=0\hspace{0.5cm}\text{on }\partial\Omega,
\]
where $\mu\in\mathcal{M}_{b}(\Omega),$ and $\left\vert h(x,U)\right\vert \leqq
f(x)(1+\left\vert U\right\vert ^{Q}),$ precising and improving the results
announced in \cite{Gre2}, see Theorem \ref{exa}. $\medskip$

In Section \ref{ret}, we return to problem (PU$\lambda)$ for general $\beta,$
and give existence, regularity, uniqueness or multiplicity results using
Theorem \ref{TP} and the results of Sections \ref{PVW} and \ref{vmeas}.
\medskip

We also analyse the meaning of the growth assumptions (\ref{hmq}) for the
function $g$ in terms of $\beta.$ It was conjectured that if $\beta$
satisfying (\ref{hypb}) with $L=\infty$, and is nondecreasing with
$\lim_{t\longrightarrow\infty}\beta\left(  t\right)  =\infty,$ the function
$g$ satisfies the growth condition (\ref{hmq}) for any $Q>p-1.$ We show that
\textbf{the conjecture is} \textbf{not true,} and give sufficient conditions
implying (\ref{hmq}).\medskip\ 

Finally we give some extensions where the function $f$ can also depend on $u,$
or for problems with different powers of the gradient term.

\section{Notions of solutions \label{Reg}}

\subsection{Renormalized solutions \medskip}

We refer to \cite{DMOP} for the main definitions, properties of regularity and
existence of renormalized solutions. For any measure $\mu\in\mathcal{M}%
_{b}(\Omega)$ the positive part and the negative part of $\mu$ are denoted by
$\mu^{+}$ and $\mu^{-}$. The measure $\mu$ admits a unique decomposition%
\begin{equation}
\mu=\mu_{0}+\mu_{s},\text{ with }\mu_{0}\in\mathcal{M}_{0}(\Omega)\text{ and
}\mu_{s}=\mu_{s}^{+}-\mu_{s}^{-}\in\mathcal{M}_{s}(\Omega), \label{dec}%
\end{equation}
where $\mathcal{M}_{0}(\Omega)$ is the subset of measures such that $\mu(B)=0$
for every Borel set $B\subseteq\Omega$ with cap$_{p}(B,\Omega)=0$. If
$\mu\geqq0,$ then $\mu_{0}\geqq0$ and $\mu_{s}\geqq0$. And any measure $\mu
\in\mathcal{M}_{b}(\Omega)$ belongs to $\mathcal{M}_{0}(\Omega)$ if and only
if it belongs to $L^{1}(\Omega)+W^{-1,p^{\prime}}(\Omega)$. \medskip

For any $k>0$ and $s\in\mathbb{R},$ we define the truncation
\[
T_{k}(s)=\max(-k,\min(k,s)).
\]
If $U$ is measurable and finite a.e. in $\Omega$, and $T_{k}(U)$ belongs to
$W_{0}^{1,p}(\Omega)$ for every $k>0$; we can define the gradient $\nabla U$
a.e. in $\Omega$ by
\[
\nabla T_{k}(U)=\nabla U.\chi_{\left\{  \left\vert U\right\vert \leqq
k\right\}  }\text{ for any }k>0.
\]
Then $U$ has a unique cap$_{p}$-quasi continuous representative; in the sequel
$U$ will be identified to this representant. Next we recall two definitions of
renormalized solutions among four equivalent ones given in \cite{DMOP}. The
second one is mainly interesting, because it makes explicit the equation
solved by the truncations $T_{k}(U)$ in the sense of distributions.

\begin{definition}
\label{nor}Let $\mu=\mu_{0}+\mu_{s}^{+}-\mu_{s}^{-}\in\mathcal{M}_{b}(\Omega
)$. A function $U$ is a renormalized solution of problem
\begin{equation}
-\Delta_{p}U=\mu\hspace{0.5cm}\text{in }\Omega,\qquad U=0\quad\text{on
}\partial\Omega. \label{mu}%
\end{equation}
if $U$ is measurable and finite a.e. in $\Omega$, such that $T_{k}(U)$ belongs
to $W_{0}^{1,p}(\Omega)$ for any $k>0,$ and $\left\vert \nabla U\right\vert
^{p-1}{\in}L^{\tau}(\Omega),$ {for any }$\tau\in\left[  1,N/(N-1)\right)  ,$
and one of the two (equivalent) conditions holds:

\noindent(i) For any $h\in W^{1,\infty}(\mathbb{R})$ such that $h^{\prime}$
has a compact support, and any $\varphi\in W^{1,s}(\Omega)$ for some $s>N,$
such that $h(U)\varphi\in W_{0}^{1,p}(\Omega),$%
\begin{equation}
\int_{\Omega}\left\vert \nabla U\right\vert ^{p-2}\nabla U.\nabla
(h(U)\varphi)dx=\int_{\Omega}h(U)\varphi d\mu_{0}+h(\infty)\int_{\Omega
}\varphi d\mu_{s}^{+}-h(-\infty)\int_{\Omega}\varphi d\mu_{s}^{-}.
\label{norr}%
\end{equation}
(ii) For any $k>0,$ there exist $\alpha_{k},\beta_{k}\in\mathcal{M}_{0}%
(\Omega)\cap\mathcal{M}_{b}^{+}(\Omega),$ concentrated on the sets $\left\{
U=k\right\}  $ and $\left\{  U=-k\right\}  $ respectively, converging in the
narrow topology to $\mu_{s}^{+},\mu_{s}^{-}$ such that for any $\psi\in
W_{0}^{1,p}(\Omega)\cap L^{\infty}(\Omega),$
\begin{equation}
\int_{\Omega}\left\vert \nabla T_{k}(U)\right\vert ^{p-2}\nabla T_{k}%
(U).\nabla\psi dx=\int_{\left\{  \left\vert U\right\vert <k\right\}  }\psi
d\mu_{0}+\int_{\Omega}\psi d\alpha_{k}-\int_{\Omega}\psi d\beta_{k}.
\label{fli}%
\end{equation}
that means, equivalently%
\begin{equation}
-\Delta_{p}(T_{k}(U))=\mu_{0,k}+\alpha_{k}-\beta_{k}\qquad\text{in
}\mathcal{D}^{\prime}(\Omega) \label{flij}%
\end{equation}
where $\mu_{0,k}=\mu_{0}\llcorner_{\left\{  \left\vert U\right\vert
<k\right\}  }$ is the restriction of $\mu_{0}$ to the set $\left\{  \left\vert
U\right\vert <k\right\}  $.\medskip
\end{definition}

Corresponding notions of local renormalized solutions are studied in
\cite{B-V}. The following properties are well-known in case $p<N,$ see
\cite{BBGGPV}, \cite{DMOP} and more delicate in case $p=N,$ see \cite{DHM} and
\cite{KilSZ}, where they require more regularity on the domain, namely,
$\mathbb{R}^{N}\backslash\Omega$ is geometrically dense: $K_{N}(\Omega
)=\inf\left\{  r^{-N}\left\vert B(x,r)\backslash\Omega\right\vert
:x\in\mathbb{R}^{N}\backslash\Omega,r>0\right\}  >0.$

\begin{proposition}
\label{Benilan}Let $1<p\leqq N,$ and $\mu\in\mathcal{M}_{b}(\Omega)$. Let $U$
be a renormalized solution of problem (\ref{mu}). If $p<N,$ then for every
$k>0,$
\begin{align*}
\left\vert \left\{  \left\vert U\right\vert \geqq k\right\}  \right\vert  &
\leqq C(N,p)k^{-(p-1)N/(N-p)}(\left\vert \mu\right\vert (\Omega))^{N/(N-p)},\\
\left\vert \left\{  \left\vert \nabla U\right\vert \geqq k\right\}
\right\vert  &  \leqq C(N,p)k^{-N(p-1)/(N-1)}(\left\vert \mu\right\vert
(\Omega))^{N/(N-1)}.
\end{align*}
If $p=N,$ then $U\in BMO,$ and
\[
\left\vert \left\{  \left\vert \nabla U\right\vert \geqq k\right\}
\right\vert \leqq C(N,K_{N}(\Omega))k^{-N}(\left\vert \mu\right\vert
(\Omega))^{N/(N-1)}.
\]

\end{proposition}

\begin{remark}
\label{estlk}As a consequence, if $p<N,$ then for any { }$\sigma\in\left(
0,N/(N-p\right)  $ and $\tau\in\left(  0,N/(N-1)\right)  ,$
\begin{equation}
(\int_{\Omega}\left\vert U\right\vert ^{(p-1)\sigma}dx)^{1/\sigma}\leqq
C(N,p,\sigma)\left\vert \Omega\right\vert ^{1/\sigma-(N-p)/N}\left\vert
\mu\right\vert (\Omega),\text{ } \label{sig}%
\end{equation}%
\begin{equation}
(\int_{\Omega}\left\vert \nabla U\right\vert ^{(p-1)\tau}dx)^{1/\tau}\leqq
C(N,p,\tau)\left\vert \Omega\right\vert ^{1/\tau-(N-1)/N}\left\vert
\mu\right\vert (\Omega), \label{tet}%
\end{equation}
If $p=N,$ then $\sigma>0$ is arbitrary, and the constant also depends on
$K_{\Omega}.$ If $p>2-1/N,$ then $U\in W_{0}^{1,q}(\Omega)$ for every
$q<(p-1)N/(N-1)$.\medskip
\end{remark}

\begin{remark}
\label{rem1} Uniqueness of the solutions of (\ref{mu}) is still an open
problem, when $p\neq2,N$ and $\mu\not \in \mathcal{M}_{0}(\Omega)$; see the
recent results of \cite{TruWa}, \cite{Mae}.

Otherwise, let $U\in W_{0}^{1,p}(\Omega),$ such that $-\Delta_{p}U=\mu$ in
$\mathcal{D}^{\prime}\left(  \Omega\right)  .$ Then $\mu\in W^{-1,p^{\prime}%
}(\Omega),$ hence $\mu\in\mathcal{M}_{0}(\Omega),$ and $U$ is an renormalized
solution of (\ref{mu}).
\end{remark}

\begin{remark}
\label{pos}Let $U$ be any renormalized solution of (\ref{mu}), where $\mu$ is
given by (\ref{dec}). \medskip

(i) If $U\geqq0$ a.e. in $\Omega,$ then the singular part $\mu_{s}\geqq0,$ see
\cite[Definition 2.21]{DMOP}. This was also called Inverse Maximum Principle"
in \cite{PePonPor}. More generally, if $u\geqq A$ a.e. in $\Omega$ for some
real $A,$ there still holds $\mu_{s}\geqq0.$ Indeed $u-A$ is a local
renormalized solution, and it follows from \cite[Theorem 2.2]{B-V}.

(ii) If $U\in L^{\infty}(\Omega),$ then $U=T_{\left\Vert U\right\Vert
_{L^{\infty}(\Omega)}}(U)\in W_{0}^{1,p}(\Omega),$ thus $\mu_{s}=0$ and
$\mu=\mu_{0}\in\mathcal{M}_{0}(\Omega)\cap W^{-1,p^{\prime}}(\Omega).$ As a
consequence, if $L<\infty,$ any solution $u$ of (PU$\lambda)$ is in
$W_{0}^{1,p}(\Omega);$ if $\Lambda<\infty,$ any solution of (PV$\lambda)$ is
in $W_{0}^{1,p}(\Omega).$
\end{remark}

Many of our proofs are based on convergence results of \cite{DMOP}. Let us
recall their main theorem:

\begin{theorem}
[\cite{DMOP}]\label{fund}Let $\mu=\mu_{0}+\mu_{s}^{+}-\mu_{s}^{-},$ with
$\mu_{0}=F-\operatorname{div}g\in\mathcal{M}_{0}(\Omega),$ \ $\mu_{s}^{+}%
,\mu_{s}^{-}\in\mathcal{M}_{s}^{+}(\Omega).$ Let
\[
\mu_{n}=F_{n}-\operatorname{div}g_{n}+\rho_{n}-\eta_{n},\qquad\text{with
}F_{n}\in L^{1}(\Omega),g_{n}\in(L^{p^{\prime}}(\Omega))^{N},\rho_{n},\eta
_{n}\in\mathcal{M}_{b}^{+}(\Omega).
\]
Assume that $(F_{n})$ converges to $F$ weakly in $L^{1}(\Omega),$ $(g_{n})$
converges to $g$ strongly in $(L^{p^{\prime}}(\Omega))^{N}$ and
$(\operatorname{div}g_{n})$ is bounded in $\mathcal{M}_{b}(\Omega),$ and
$(\rho_{n})$ converges to $\mu_{s}^{+}$ and $(\eta_{n})$ converges to $\mu
_{s}^{-}$ in the narrow topology. Let $U_{n}$ be a renormalized solution of
\[
-\Delta_{p}U_{n}=\mu_{n}\hspace{0.5cm}\text{in }\Omega,\qquad U_{n}%
=0\quad\text{on }\partial\Omega.
\]
Then there exists a subsequence $(U_{\nu})$ converging a.e. in $\Omega$ to a
renormalized solution $U$ of problem%
\[
-\Delta_{p}U=\mu\hspace{0.5cm}\text{in }\Omega,\qquad U=0\quad\text{on
}\partial\Omega.
\]
And $(T_{k}(U_{\nu}))$ converges to $T_{k}(U)$ strongly in $W_{0}^{1,p}%
(\Omega).$
\end{theorem}

\subsection{Reachable solutions}

A weaker notion of solution will be used in the sequel, developped in
\cite[Theorems 1.1 and 1.2]{DmMalu}:

\begin{definition}
Let $\mu\in\mathcal{M}_{b}(\Omega)$. A function $U$ is a reachable solution of
problem (\ref{mu}) if it satisfies one of the (equivalent) conditions:

\noindent(i) There exists $\varphi_{n}\in\mathcal{D}(\Omega)$ and $U_{n}\in
W_{0}^{1,p}(\Omega),$ such that $-\Delta_{p}U_{n}=\varphi_{n}$ in
$W^{-1,p^{\prime}}(\Omega),$ such that ($\varphi_{n})$ converges to $\mu$
weakly* in $\mathcal{M}_{b}(\Omega),$ and $\left(  U_{n}\right)  $ converges
to $U$ a.e. in $\Omega.$

\noindent(ii) $U$ is measurable and finite a.e. in $\Omega$, such that
$T_{k}(U)$ belongs to $W_{0}^{1,p}(\Omega)$ for any $k>0,$ and there exists
$M>0$ such that $\int_{\Omega}\left\vert \nabla T_{k}(U)\right\vert
^{p}dx\leqq M(k+1)$ for any $k>0,$ and $\left\vert \nabla U\right\vert
^{p-1}{\in}L^{1}(\Omega),$ and
\begin{equation}
-\Delta_{p}U=\mu\hspace{0.5cm}\text{in }\mathcal{D}^{\prime}\left(
\Omega\right)  . \label{dprim}%
\end{equation}
(iii) $U$ is measurable and finite a.e., such that $T_{k}(U)$ belongs to
$W_{0}^{1,p}(\Omega)$ for any $k>0,$ and there exists $\mu_{0}\in
\mathcal{M}_{0}(\Omega)$ and $\mu_{1},\mu_{2}\in\mathcal{M}_{b}^{+}(\Omega),$
such that $\mu=\mu_{0}+\mu_{1}-\mu_{2}$ and for any $h\in W^{1,\infty
}(\mathbb{R})$ such that $h^{\prime}$ has a compact support, and any
$\varphi\in\mathcal{D}(\Omega),$
\begin{equation}
\int_{\Omega}\left\vert \nabla U\right\vert ^{p-2}\nabla U.\nabla
(h(U)\varphi)dx=\int_{\Omega}h(U)\varphi d\mu_{0}+h(\infty)\int_{\Omega
}\varphi d\mu_{1}-h(-\infty)\int_{\Omega}\varphi d\mu_{2}. \label{trio}%
\end{equation}

\end{definition}

\begin{remark}
\label{rich} Any reachable solution satisfies $\left\vert \nabla U\right\vert
^{p-1}{\in}L^{\tau}(\Omega),$ {for any }$\tau\in\left[  1,N/(N-1)\right)  ,$
and (the cap$_{p}$-quasi continuous representative of) $U$ is finite cap$_{p}%
$-quasi everywhere in $\Omega$, from \cite[Theorem 1.1]{DmMalu} and
\cite[Remark 2.11]{DMOP}. Moreover, from \cite{DmMalu}, for any $k>0,$ there
exist $\alpha_{k},\beta_{k}\in\mathcal{M}_{0}(\Omega)\cap\mathcal{M}_{b}%
^{+}(\Omega),$ concentrated on the sets $\left\{  U=k\right\}  $ and $\left\{
U=-k\right\}  $ respectively, converging weakly* to $\mu_{1},\mu_{2},$ such
that%
\[
-\Delta_{p}(T_{k}(U))=\mu_{0,k}=\mu_{0}\llcorner_{\left\{  \left\vert
U\right\vert <k\right\}  }+\alpha_{k}-\beta_{k}\qquad\text{in }\mathcal{D}%
^{\prime}(\Omega).
\]
Obviously, any renormalized solution is a reachable solution. The notions
coincide for $p=2$ and $p=N.$
\end{remark}

\subsection{Second member in $L^{1}(\Omega).$}

In the sequel we often deal with the case where the second member is in
$L^{1}(\Omega).$ Then the notion of renormalized solution coincides with the
notions of reachable solution, and entropy solution introduced in
\cite{BBGGPV}, and SOLA solution given in \cite{DaA}, see also \cite{BoGaOr}.

\begin{definition}
\label{grang}We call $\mathcal{W}$ $(\Omega)$ the space of functions $U$ such
that there exists $F\in L^{1}(\Omega)$ such that $U$ is a renormalized
solution of problem
\[
-\Delta_{p}U=F\hspace{0.5cm}\text{in }\Omega,\qquad U=0\quad\text{on }%
\partial\Omega.
\]
Then $U$ is unique, we set
\begin{equation}
U=\mathcal{G}(F). \label{ugf}%
\end{equation}
In the same way we call $\mathcal{W}_{loc}(\Omega)$ the space of fuctions $U$
such that there exists $F\in L_{loc}^{1}(\Omega)$ such that $U$ is a local
renormalized solution of equation $-\Delta_{p}U=F$ in $\Omega.$
\end{definition}

\begin{remark}
\label{cp} From uniqueness, the Comparison Principle holds:

If $U_{1}$ and $U_{2}\in\mathcal{W}(\Omega)$ and $-\Delta_{p}U_{1}\geqq
-\Delta_{p}U_{2}$ a.e. in $\Omega,$ then $U_{1}\geqq U_{2}$ a.e. in $\Omega.$
\end{remark}

\begin{remark}
\label{conv} Theorem \ref{fund} implies in particular:

If ($F_{n})$ converges to $F$ weakly in $L^{1}(\Omega),$ and $U_{n}%
=\mathcal{G}(F_{n}),$ then there exists a subsequence ($U_{\nu})$ converging
a.e. to some function $U,$ such that $U=\mathcal{G}(F).$
\end{remark}

\subsection{More regularity results}

All the proofs of this paragraph are given in the Appendix. First we deduce a
weak form of the Picone inequality:

\begin{lemma}
\label{Pic}Let $U\in W_{0}^{1,p}(\Omega)$, and $V\in\mathcal{W}$ $(\Omega),$
such that $U\geqq0$ and $-\Delta_{p}V\geqq0$ a.e. in $\Omega,$ and
$V\not \equiv 0.$ Then $U^{p}(-\Delta_{p}V)/V^{p-1}\in L^{1}(\Omega)$ and
\begin{equation}
\int_{\Omega}\left\vert \nabla U\right\vert ^{p}dx\geqq\int_{\Omega}%
U^{p}V^{1-p}(-\Delta_{p}V)dx. \label{picone}%
\end{equation}

\end{lemma}

Next we prove a regularity Lemma, giving estimates of $u$ and its gradient in
optimal $L^{k}$ spaces, available for \textbf{any renormalized solution}. It
improves the results of \cite{BoGa}, \cite{Gre}, \cite{ABFOT}, \cite{CaSa} and
extends the estimates of the gradient given in \cite{KiZh}, \cite{KiZh2} for
solutions $U\in W_{0}^{1,p}(\Omega)$. Estimates in Marcinkiewicz or Lorentz
spaces are given in \cite{KilLi}, \cite{AlFeTr}.

\begin{lemma}
\label{boot}Let $1<p\leqq N.$ Let $U=\mathcal{G}(F)$ be the renormalized
solution of problem
\begin{equation}
-\Delta_{p}U=F\hspace{0.5cm}\text{in }\Omega,\qquad U=0\quad\text{on }%
\partial\Omega. \label{E29a}%
\end{equation}
with $F\in L^{m}(\Omega),$ $1<m<N.$ Set $\overline{m}=Np/(Np-N+p)$.

\noindent(i) If $m>N/p,$ then $U\in L^{\infty}(\Omega).$

\noindent(ii) If $m=N/p,$ then $U\in L^{k}(\Omega)$ for any $k\geqq1.$

\noindent(iii) $If$ $m<N/p,$ then $U^{p-1}\in L^{k}(\Omega)$ for
$k=Nm/(N-pm).\medskip$

\noindent(iv) $\left\vert \nabla U\right\vert ^{(p-1)}\in L^{k}(\Omega)$ for
$k=Nm/(N-m)$. In particular if $\overline{m}\leqq m,$ then $U\in W_{0}%
^{1,p}(\Omega).\medskip$
\end{lemma}

Using this Lemma, we get regularity results under growth conditions, extending
well known results in case $p=2$, $f\equiv1$:

\begin{proposition}
\label{cig}Let $1<p\leqq N.$ Let $U=\mathcal{G}(h)$ where $h\in L^{1}%
(\Omega),$ and
\[
\left\vert h(x\right\vert \leqq f(x)(\left\vert U\right\vert ^{Q}%
+1)\qquad\text{a.e. in }\Omega,
\]
with $f\in L^{r}(\Omega),r>1$ and $Q>0.$ If $p<N;$ then

\noindent(i) If $Q\geqq p-1$ and $Qr^{\prime}<Q_{1}$ (hence $r>N/p),$ then
$U\in W_{0}^{1,p}(\Omega)\cap L^{\infty}(\Omega).\medskip$

\noindent(ii) If $Q>p-1$ and $Qr^{\prime}=Q_{1}$ and $\left\vert U\right\vert
^{p-1}\in L^{\sigma}(\Omega)$ for some $\sigma>N/(N-p),$ then $U\in
W_{0}^{1,p}(\Omega)$ and $U\in L^{k}(\Omega)$ for any $k\geqq1.\medskip$

\noindent(iii) If $Q\geqq p-1$ and if $U\in W_{0}^{1,p}(\Omega),$ and
($Q+1)r^{\prime}<p^{\ast},$ then $U\in L^{\infty}(\Omega);$ if ($Q+1)r^{\prime
}=p^{\ast},$ then $U\in L^{k}(\Omega)$ for any $k\geqq1.\medskip$

\noindent(iv) If $Q<p-1$ and $r>N/p,$ then $U\in W_{0}^{1,p}(\Omega)\cap
L^{\infty}(\Omega).\medskip$

\noindent(v) If $Q<p-1$ and $r=N/p,$ then $U\in W_{0}^{1,p}(\Omega)$ and $U\in
L^{k}(\Omega)$ for any $k\geqq1.\medskip$

\noindent(vi) If $Q<p-1$ and $r<N/p$ and $Qr^{\prime}<Q_{1},$ then $U^{k}\in
L^{1}(\Omega)$ for any $k<d=Nr(p-1-Q)/(N-pr).$ Either ($Q+1)r^{\prime}%
<p^{\ast}$ and then $U\in W_{0}^{1,p}(\Omega),$ or ($Q+1)r^{\prime}\geqq
p^{\ast},$ then $\left\vert \nabla U\right\vert ^{t}\in L^{1}(\Omega)$ for any
$t<\theta=Nr(p-1-Q)/(N-(Q+1)r).$

\noindent If $p=N,$ then $U\in W_{0}^{1,N}(\Omega)\cap L^{\infty}(\Omega),$
and $\left\vert \nabla U\right\vert ^{N(N-1)m/(N-m)}\in L^{1}(\Omega)$ for any
$m<min(r,N)$.
\end{proposition}

\begin{remark}
It may happen that $U\not \in W_{0}^{1,p}(\Omega)$ for $Q\geqq p-1$, and
condition (ii) is quite sharp: let $p=2$ and $\Omega=B(0,1);$ there exists a
positive radial function $U\in L^{N/(N-2)}(\Omega)$ such that
\[
-\Delta U=U^{N/(N-2)}\hspace{0.5cm}\text{in }\Omega,\quad U=0\quad\text{on
}\partial\Omega,\quad\text{and }\lim_{x\rightarrow0}\left\vert x\right\vert
^{N-2}\left\vert \ln\left\vert x\right\vert \right\vert ^{(N-2)/2}U(x)=c_{N},
\]
where $c_{N}>0,$ see \cite{Ra}. Then $U\not \in L^{\sigma}(\Omega)$ for
$\sigma>N/(N-2),$ hence $U\not \in W_{0}^{1,2}(\Omega).$ It satisfies the
equation $-\Delta U=fU^{Q}$ with $Q=N/(N-2),f\equiv1,$ and also with $Q=1,$
$f=U^{2/(N-2)}\in L^{N/2}(\Omega).$
\end{remark}

Next we we prove local estimates of the second member $F$ when $F\in$
$L_{loc}^{1}(\Omega)$ and $F\geqq0,$. following an idea of \cite{B-VPo}:

\begin{lemma}
\label{secm}Let $U\in$ $\mathcal{W}_{loc}(\Omega)$ such that $-\Delta
_{p}U=F\geqq0$ a.e. in $\Omega.$ For any $x_{0}\in\Omega$ and any ball
$B(x_{0},4\rho)\subset\Omega,$ and any $\sigma\in\left(  0,N/(N-p)\right)  ,$
there exists a constant $C=C(N,p,\sigma),$ such that
\begin{equation}
\int_{B(x_{0},\rho)}Fdx\leq C\rho^{N(1-1/\sigma)-p}\left(  \int_{B(x_{0}%
,2\rho)}U^{(p-1)\sigma}dx\right)  ^{1/\sigma}. \label{B}%
\end{equation}

\noindent If $U\in W_{loc}^{1,p}(\Omega),$ there exists a constant $C=C(N,p)$
such that
\begin{equation}
\int_{B(x_{0},\rho)}Fdx\leqq C\rho^{N-p}\inf\text{ess}_{B(x_{0},\rho)}U^{p-1}.
\label{Ha}%
\end{equation}

\end{lemma}

Finally we mention a result of \cite{Pon}, which is a direct consequence of
the Maximum Principle when $p=2,$ but is not straightforward for $p\neq2,$
since no Comparison Principle is known for measures:

\begin{lemma}
\label{Ponce}Let $h$ be a Caratheodory function from $\Omega\times\left[
0,\infty\right)  $ into $\left[  0,\infty\right)  .$ Let $\mu_{s}%
\in\mathcal{M}_{s}^{+}(\Omega)$ and $u$ be a renormalized nonnegative solution
of
\begin{equation}
-\Delta_{p}U=h(x,U)+\mu_{s}\hspace{0.5cm}\text{in }\Omega,\qquad
U=0\quad\text{on }\partial\Omega. \label{umu}%
\end{equation}
Suppose that $\sup_{t\in\left[  0,u(x)\right]  }h(x,t)=F(x)\in L^{1}(\Omega).$
Then there exists a renormalized nonnegative solution $V$ of
\begin{equation}
-\Delta_{p}V=h(x,V)\hspace{0.5cm}\text{in }\Omega,\qquad V=0\hspace
{0.5cm}\text{on }\partial\Omega. \label{vsmu}%
\end{equation}

\end{lemma}

\section{Correlation between the two problems\label{Cor}}

\subsection{The pointwise change of unknowns}

(i) Let $\beta$ satisfy (\ref{hypb}). Let for any $t\in\left[  0,L\right)  $
\[
\Psi(t)=\int_{0}^{t}e^{\gamma(\theta)/(p-1)}d\theta,\qquad\gamma(t)=\int
_{0}^{t}\beta(\theta)d\theta;
\]
then $\Psi(\left[  0,L\right)  )=\left[  0,\Lambda\right)  ,\Lambda
=\Psi(L)\leqq\infty,$ and the function%
\begin{equation}
\tau\in\left[  0,\Lambda\right)  \mapsto g(\tau)=e^{\gamma(\Psi^{-1}%
(\tau))/(p-1)}-1=\frac{1}{p-1}\int_{0}^{\tau}\beta(\Psi^{-1}(s))ds
\label{defg}%
\end{equation}
satisfies (\ref{hypg}) and $\Psi^{-1}=H,$ where
\begin{equation}
H(\tau)=\int_{0}^{\tau}\frac{ds}{1+g(s)}. \label{hg}%
\end{equation}
(ii) Conversely let $g$ satisfying (\ref{hypg}), then $H(\left[
0,\Lambda\right)  )=\left[  0,L\right)  $, $L=H(\Lambda),$ and the function
$t\in\left[  0,L\right)  \mapsto\beta(t)=(p-1)g^{\prime}(H^{-1}(t))$ satisfies
(\ref{hypb}), and $H=\Psi^{-1}:$ indeed\
\[
H(\tau)=\int_{0}^{\tau}\frac{ds}{1+g(s)}=\int_{0}^{\tau}e^{-\gamma(\Psi
^{-1}(s))/(p-1)}ds=\int_{0}^{\Psi^{-1}(\tau)}e^{-\gamma(\theta))/(p-1)}%
\Psi^{\prime}(\theta)d\theta=\Psi^{-1}(\tau).
\]

\noindent Then $\beta$ and $g$ are linked by the relations, at any point
$\tau=\Psi(t),$
\begin{equation}
\beta(t)=(p-1)g^{\prime}(\tau),\quad1+g(\tau)=e^{\gamma(t)/(p-1)}.
\label{rela}%
\end{equation}
In particular $\beta$ is nondecreasing if and only if $g$ is convex.

\begin{remark}
\label{tru}One easily gets the following properties:
\[
L=\infty\Longrightarrow\Lambda=\infty;\qquad L<\infty\Longleftrightarrow
1/(1+g(s))\in L^{1}\left(  \left(  0,\Lambda\right)  \right)  ;
\]%
\[
\Lambda<\infty\Longleftrightarrow e^{\gamma(t)/(p-1)}\in L^{1}\left(  \left(
0,L\right)  \right)  ;\qquad\gamma(L)<\infty\Longleftrightarrow\beta\in
L^{1}\left(  \left(  0,L\right)  \right)  \Longleftrightarrow g\text{
bounded;}%
\]%
\[
\underline{\lim}_{t\rightarrow L}\beta(t)>0\Longrightarrow\underline{\lim
}_{t\rightarrow\Lambda}g(s)/s>0;
\]%
\[
\lim_{t\longrightarrow L}\beta(t)=\infty\Longrightarrow\lim_{s\longrightarrow
\Lambda}g(s)/s=\infty,\quad\text{and conversely if }\beta\text{ is
nondecreasing near }L\text{.}%
\]
If $\Lambda<\infty,$ then
\[
\lim_{v\longrightarrow\Lambda}g(v)=\infty\;\text{ if }\beta(u)\not \in
L^{1}((0,L));\qquad\lim_{v\longrightarrow\Lambda}g(v)=e^{\gamma(L)/(p-1)}%
-1\;\text{if }\beta(u)\in L^{1}((0,L)).
\]
Notice that \textbf{the correlation between }$g$\textbf{ and }$\beta$\textbf{
is not monotone}; we only have: if $g_{1}^{\prime}\leqq g_{2}^{\prime},$ then
$\beta_{1}\leqq\beta_{2}.$ Also it is \textbf{not symmetric} between $u$ and
$v$ : we always have $u\leqq v;$ moreover $\nabla u=\nabla v/(1+g(v)),$ thus
$u$ can be expected more regular than $v$ when $\lim_{v\longrightarrow\Lambda
}g(v)=\infty.$
\end{remark}

\begin{remark}
\label{dou}(i) If $u$ is a renormalized solution of (\ref{PA}), then by
definition $\beta(u)\left\vert \nabla\,u\right\vert ^{p}\in L^{1}(\Omega);$ if
$v$ is a renormalized solution of (\ref{PS}), then $f(1+g(v))^{p-1}\in
L^{1}(\Omega).\medskip$

(ii)) For any $v\in W_{0}^{1,p}(\Omega),$ then $u=H(v)\in W_{0}^{1,p}%
(\Omega).$ $\medskip$

(iii) If $L=\infty$ and \textbf{ }$\underline{\lim}_{t\rightarrow\infty}%
\beta(t)>0,$ and $u$ is a renormalized solution of (\ref{PA}), then
$u\in\,W_{0}^{1,p}(\Omega);$ indeed $\beta(t)\geqq m>0$ for $t\geqq K_{0}>0,$
thus
\[
\int_{\Omega}\left\vert \nabla\,u\right\vert ^{p}dx=\int_{\left\{  u\geqq
K_{0}\right\}  }\left\vert \nabla\,u\right\vert ^{p}dx+\int_{\left\{  u\leqq
K_{0}\right\}  }\left\vert \nabla\,u\right\vert ^{p}dx\leqq\frac{1}{m}%
\int_{\Omega}\beta(u)\left\vert \nabla\,u\right\vert ^{p}dx+\int_{\Omega
}\left\vert \nabla\,T_{K_{0}}(u)\right\vert ^{p}dx.
\]

\end{remark}

\subsection{Examples}

Here we give examples, where the correlation can be given (quite) explicitely,
giving good models for linking the behaviour of $\beta$ near $L$ and $g$ near
$\Lambda.$ The computation is easier by starting from a given function $g$ and
computing $u$ from (\ref{hdv}) and then $\beta$ by (\ref{bdu}). The examples
show how the correlation is sensitive with respect to $\beta:$ \textbf{a small
perturbation on }$\beta$\textbf{ can imply a very strong perturbation on }$g.$
Examples $1,2,5,6$ are remarkable, since they lead to very well known
equations in $v$. Example 10 is a model of a new type of problems in $v$,
presenting a singularity, which can be qualified as quenching problem. The
arrow $\leftrightarrow$ indicates the formal link between the two problems.
\medskip

\textbf{1}$%
{{}^\circ}%
)$\textbf{ Cases where }$\beta$\textbf{ is defined on }$\left[  0,\infty
\right)  $ \textbf{( }$L=\infty=\Lambda).$\medskip

\noindent1) \textit{ }$\beta$\textit{ constant, }$g$ \textit{linear: }\medskip

Let $\beta(u)=p-1,$ $g(v)=v,$ $u=\ln(1+v),$
\[
-\Delta_{p}u=(p-1)\left\vert \nabla u\right\vert ^{p}+\lambda f(x)\hspace
{0.5cm}\leftrightarrow\hspace{0.5cm}-\Delta_{p}v=\lambda f(x)(1+v)^{p-1}.
\]
\medskip2) $g$\textit{ of power type and sublinear:} \medskip

Let $0<Q<p-1;$ setting $\alpha=Q/(p-1)<1,$ and $\beta(u)=(p-1)\alpha
/(1+(1-\alpha)u),$ we find ($1+g(v))^{p-1}=(1+v)^{Q}$; and $(1-\alpha
)u=(1+v)^{1-\alpha}-1:$%

\[
-\Delta_{p}u=\frac{(p-1)\alpha}{1+(1-\alpha)u}\left\vert \nabla u\right\vert
^{p}+\lambda f(x)\hspace{0.5cm}\leftrightarrow\hspace{0.5cm}-\Delta
_{p}v=\lambda f(x)(1+v)^{Q};
\]
here $g$ is concave and unbounded, thus $\beta$ is nonincreasing, and
$\beta(u)\thicksim C/u$ near $\infty,$ thus $\beta\not \in L^{1}\left(
\left(  0,\infty\right)  \right)  .$\medskip

\noindent3) $\beta$ \textit{of power type, }$g$ \textit{of logarithmic
type.}\medskip\ 

Let $\beta(u)=(p-1)u^{m},$ $m>0,$ then $g(v)\thicksim Cv(\ln v))^{m/(m+1)}$,
with $C=(m+1)^{m/(m+1)}.$ Indeed integrating by parts $\int_{1}^{u}%
s^{-(m+1)}e^{s^{m+1}/(m+1)}ds,$ we find that $v\thicksim u^{-m}e^{u^{m+1}%
/(m+1)}$ near $\infty,$ then $\ln v\thicksim u^{m+1}/(m+1)$.\ 

Conversely let $1+g(v)=(1+Cv)(1+\ln(1+Cv))^{m/(m+1)},$ $m>0,C>0,$ then
$Cu=(m+1)((1+\ln(1+v))^{1/(m+1)}-1),$ and $\beta(u)=(p-1)C((1+Cu/(m+1))^{m}%
+m/(m+1+Cu))$, then $\beta(u)\thicksim Ku^{m}$ near $\infty,$ with
$K=(p-1)C^{m+1}(m+1)^{-m}.$

\medskip

\noindent4) $\beta$\textit{ of exponential type, }$g$ \textit{of logarithmic
type}. \medskip

If $\beta(u)=(p-1)e^{u}$, then $g(v)\sim v\ln v$ near $\infty.$ Indeed
integrating by parts the integral $\int_{0}^{u}e^{-s}e^{e^{s}}ds$ we get
$v\sim e^{e^{u}-u-1}$near $\infty.$

If $\beta(u)=(p-1)(e^{u}+1),$ we find precisely $1+g(v)=(1+v)(1+\ln(1+v))$ and
$u=\ln(1+\ln(1+v)):$
\[
-\Delta_{p}u=(e^{u}+1)\left\vert \nabla u\right\vert ^{p}+\lambda
f(x)\hspace{0.5cm}\leftrightarrow\hspace{0.5cm}-\Delta_{p}v=\lambda
f(x)((1+v)(1+\ln(1+v)))^{p-1}.
\]

If $\beta(u)=(p-1)(e^{e^{u}+u}+e^{u}+1),$ we verify that $e^{e+1}%
v=e^{e^{e^{u}}}-e^{e}$ and $1+g(v)\sim v\ln v\ln(\ln v))$ near $\infty
.$\medskip

\textbf{2}$%
{{}^\circ}%
)$\textbf{ Cases where }$\beta$\textbf{ has an asymptote (}$L<\infty$
)\textbf{, but }$g$\textbf{ is defined on }$\left[  0,\infty\right)  .$ It is
the case where $1/(1+g(s))\in L^{1}\left(  \left(  0,\infty\right)  \right)
.$\medskip

\noindent5) $g$\textit{ of power type and superlinear: }\medskip

\textit{Let }$Q>p-1;$ setting $\alpha=Q/(p-1)>1,$\textit{ }and $\beta
(u)=(p-1)\alpha/(1-(\alpha-1)u),$ with $L=1/(\alpha-1),$ we find
($1+g(v))^{p-1}=(1+v)^{Q}$ and $(\alpha-1)u=1-(1+v)^{1-\alpha}:$%

\[
-\Delta_{p}u=\frac{(p-1)\alpha}{1-(\alpha-1)u}\left\vert \nabla u\right\vert
^{p}+\lambda f(x)\hspace{0.5cm}\leftrightarrow\hspace{0.5cm}-\Delta
_{p}v=\lambda f(x)(1+v)^{Q}.
\]
Another example is the case $\beta(u)=2(p-1)tgu.$ with $L=\pi/2,$ where
$1+g(v)=1+v^{2},$ and $u=Arctgv.$\medskip

\noindent6) $g$\textit{ of exponential type:\medskip\ }

Let $\beta(u)=(p-1)/(1-u)$ with $L=1$, then $1+g(v)=e^{v},$ and $u=1-e^{-v}$:%
\[
-\Delta_{p}u=\frac{p-1}{1-u}\left\vert \nabla u\right\vert ^{p}+\lambda
f(x)\hspace{0.5cm}\leftrightarrow\hspace{0.5cm}-\Delta_{p}v=\lambda
f(x)e^{v}.
\]

\noindent7) $g$\textit{ of logarithmic type:\medskip}

Let $\beta(u)=(p-1)k/(1-u)^{k+1},k>0$ with $L=1,$ then we obtain
$g(v)\thicksim kv(\ln v))^{(k+1)/k}$ near $\infty.$ Conversely, if
$1+g(v)=(1+kv)(1+\ln(1+kv))^{(k+1)/k},$ then $\beta(u)=(p-1)(k/(1-u)^{k+1}%
+(k+1)/(1-u))$, thus $\beta(u)\thicksim(p-1)k/(1-u)^{k+1}$ near $1.$ Observe
that $\beta$ has a stronger singularity than the one of example $6$, but $g$
has a slow growth. \medskip

\noindent8) $g$\textit{ of strong exponential type:\medskip}

Let $\beta(u)=(1-u)^{-1}(1-(\ln(e/(1-u)))^{-1})$ with $L=1,$ then
$1+g(v)=e^{e^{v}-v-1},$ $u=1-e^{1-e^{v}}.$ Notice that $\beta$ has a
singularity of the same type as the one example $6.$ \medskip

\textbf{3}$%
{{}^\circ}%
)$\textbf{ Cases where }$\beta$\textbf{ and }$g$\textbf{ have an asymptote
(}$L<\infty$\textbf{ and }$\Lambda<\infty).$\medskip

\noindent9$)$ Let $Q>0.$ Setting $\alpha=Q/(p-1)>0,$ and $\beta(u)=(p-1)\alpha
/(1-(\alpha+1)u),$ with $L=1/(\alpha+1),$ we obtain ($1+g(v))^{p-1}%
=(1-v)^{-Q}$ and $(\alpha+1)u=1-(1-v)^{\alpha+1}:$%
\[
-\Delta_{p}u=\frac{(p-1)\alpha}{1-(\alpha+1)u}\left\vert \nabla u\right\vert
^{p}+\lambda f(x)\hspace{0.5cm}\leftrightarrow\hspace{0.5cm}-\Delta_{p}%
v=\frac{\lambda f(x)}{(1-v)^{Q}}%
\]
\medskip\ 

\noindent10) $\beta(u)=(p-1)u/(1-u^{2}),$ then $1+g(v)=1/\cos v,$ and $u=\sin
v$.

\subsection{Proof of the correlation Theorem}

For proving Theorem \ref{TP}, we cannot use approximations by regular
functions, because of \ to the possible nonuniqueness of the solutions of
(\ref{mu}) for $p\neq2,N$, see Remark \ref{rem1}. Then we use the equations
satisfied by the truncations. Such an argument was also used in \cite{Malu} in
order to simplify the proofs of \cite{DMOP}.

\begin{remark}
\label{cho}(i) If $u$ is a solution of (\ref{PA}), where $0\leqq u(x)<L$ a.e.
in $\Omega$, and if $L<\infty,$ then $\alpha_{s}=0$ from Remark \ref{pos}, and
$u=T_{L}(u)\in W_{0}^{1,p}(\Omega)\cap L^{\infty}(\Omega).$ If $v$ is solution
of (\ref{PS}) and $\Lambda<\infty,$ then $\mu_{s}=0$ and $v=T_{\Lambda}(v)\in
W_{0}^{1,p}(\Omega)\cap L^{\infty}(\Omega).$

(ii) If $u$ is a solution of (\ref{PA}), the set $\left\{  u=L\right\}  $ has
a p-capacity $0.$ It folllows from \cite[Remark 2.11]{DMOP}, if $L=\infty,$
from \cite[Proposition 2.1]{DMOP} applied to ($u-L)^{+}$ if $L<\infty$. In the
same way if $v$ is a solution of (\ref{PS}), the set $\left\{  v=\Lambda
\right\}  $ has a p-capacity $0.$
\end{remark}

\begin{lemma}
\label{Tr}Suppose that $u$ is a renormalized solution of (\ref{PA}), where
$0\leqq u(x)<L$ a.e. in $\Omega$, \textbf{or} that $v=\Psi(u)$ is a
renormalized solution of (\ref{PS}), where $0\leqq v(x)<\Lambda$ a.e. in
$\Omega.$ For any $K>0,$ $k>0$ there exists $\alpha_{K},\mu_{k}\in
\mathcal{M}_{0}(\Omega)\cap\mathcal{M}_{b}^{+}(\Omega)$ such that the
truncations satisfy the equations
\begin{equation}
-\Delta_{p}T_{K}(u)=\beta(T_{K}(u))\left\vert \nabla\,T_{K}(u)\right\vert
^{p}+\lambda f\chi\left\{  u<K\right\}  +\alpha_{K}\qquad\text{in }%
\mathcal{D}^{\prime}(\Omega), \label{EquK}%
\end{equation}%
\begin{equation}
-\Delta_{p}T_{k}(v)=\lambda f(1+g(v))^{p-1}\chi_{\left\{  v<k\right\}  }%
+\mu_{k}\qquad\text{in }\mathcal{D}^{\prime}(\Omega), \label{eqvk}%
\end{equation}
and
\begin{equation}
\mu_{k}=(1+g(k))^{p-1}\alpha_{K},\qquad\text{for any }k=\Psi(K)>0. \label{amu}%
\end{equation}
Moreover, if $u$ is a solution of (\ref{PA}) then $\alpha_{K}$ converges in
the \medskip narrow topology to $\alpha_{s}$ as $K\nearrow L;$ if $v$ is a
solution of (\ref{PS}) then $\mu_{k}$ converges in the \medskip narrow
topology to $\mu_{s}$ as $k\nearrow\Lambda.$
\end{lemma}

\begin{proof}
(i) Let $v$ be a renormalized solution of (\ref{PS}), and $u=H(v).$ Then
$f(1+g(v))^{p-1}\in L^{1}(\Omega).$ For any $k\in\left(  0,\Lambda\right)  ,$
let $K=H(k),$ then $T_{k}(v)\in W_{0}^{1,p}(\Omega),$ and $T_{K}(u)=H\left(
T_{k}(v)\right)  $ $\in W_{0}^{1,p}(\Omega)$. Observe that
\begin{equation}
(1+g(T_{k}(v)))^{p-1}=e^{\gamma(T_{K}(u))},\quad\text{and }\nabla
T_{k}(v)=e^{\gamma(T_{K}(u))/(p-1)}\nabla T_{K}(u). \label{hop}%
\end{equation}
Thus $\nabla v=e^{\gamma(u)/(p-1)}\nabla u,$ then $\left\vert \nabla
v\right\vert ^{p-1}=e^{\gamma(u)}\left\vert \nabla u\right\vert ^{p-1}$ a.e.
in $\Omega.$ From (\ref{fli}) (\ref{flij}), there exists $\mu_{k}%
\in\mathcal{M}_{0}(\Omega)\cap\mathcal{M}_{b}^{+}(\Omega),$ concentrated on
$\left\{  v=k\right\}  $ such that $\mu_{k}\rightarrow\mu_{s}$ in the narrow
topology as $k\longrightarrow\infty,$ and $T_{k}(v)$ satisfies (\ref{eqvk}),
that means
\[
\int_{\Omega}\left\vert \nabla\,T_{k}(v)\right\vert ^{p-2}\nabla
\,T_{k}(v).\nabla\varphi\,dx=\lambda\int_{\left\{  v<k\right\}  }%
f(1+g(v))^{p-1}\varphi\,dx+\int_{\Omega}\varphi\,d\mu_{k},
\]
for any $\varphi\in\,W_{0}^{1,p}(\Omega)\cap\,L^{\infty}(\Omega)$. For given
$\psi\in\,W_{0}^{1,p}(\Omega)\cap\,L^{\infty}(\Omega)$, taking $\varphi
=e^{-\gamma(T_{K}(u))}\psi,$ we obtain
\begin{align*}
\int_{\Omega}\left\vert \nabla\,T_{K}(u)\right\vert ^{p-2}\nabla
\,T_{K}(u).\nabla\psi\,dx  &  =\int_{\Omega}\beta(T_{K}(u))\left\vert
\nabla\,T_{K}(u)\right\vert ^{p}\psi\,dx\\
&  +\lambda\int_{\left\{  U<k\right\}  }f\psi\,dx+\frac{1}{(1+g(k))^{p-1}}%
\int_{\Omega}\psi d\mu_{k}.
\end{align*}
In other words, $T_{K}(u)$ satisfies equation (\ref{EquK}) where $\alpha_{K}$
is given by (\ref{amu}). If $\Lambda<\infty,$ then $\mu_{s}=0,$ and
$v=T_{\Lambda}(v)\in W_{0}^{1,p}(\Omega)\cap L^{\infty}(\Omega),$ and $\mu
_{k}$ converges to $0$ in $\mathcal{D}^{\prime}(\Omega)$ as $k\nearrow
\Lambda,$ hence weakly$^{\ast}$ in $\mathcal{M}_{b}(\Omega).$And taking
$\varphi=T_{k}(v),$
\begin{align*}
\lim_{k\nearrow\Lambda}k\mu_{k}(\Omega)  &  =\lim_{k\nearrow\Lambda}%
(\int_{\Omega}\left\vert \nabla T_{k}(v)\right\vert ^{p}dx-\int_{\Omega
}\lambda f(1+g(v))^{p-1}v\chi_{\left\{  v<k\right\}  }dx)\\
&  =\int_{\Omega}\left\vert \nabla v\right\vert ^{p}dx-\int_{\Omega}\lambda
f(1+g(v))^{p-1}vdx=0,
\end{align*}
thus $\mu_{k}$ converges to $0$ in the narrow topology. Hence in any case
($\Lambda$ finite or not), $\mu_{k}$ converges to $\mu_{s}$ in the \medskip
narrow topology as $k\nearrow\Lambda.$

(ii) Let $u$ be a renormalized solution of (\ref{PA}) and $v=\Psi(u)$. Then
$\beta(u))\left\vert \nabla\,u)\right\vert ^{p}\in L^{1}\left(  \Omega\right)
.$ For any $K\in\left(  0,L\right)  ,$ let $k=\Psi(K)\in\left(  0,\Lambda
\right)  .$ Then $T_{k}(v)=$ $\Psi\left(  T_{K}(u)\right)  \in W_{0}%
^{1,p}(\Omega)$. From (\ref{fli}) (\ref{flij}), there exists $\alpha_{K}%
\in\mathcal{M}_{0}(\Omega)\cap\mathcal{M}_{b}^{+}(\Omega),$ concentrated on
the set $\left\{  u=K\right\}  $, such that $\alpha_{K}$ converges to
$\alpha_{s}$ in the narrow topology, as $k\longrightarrow\infty,$ and
$T_{K}(u)$ satisfies (\ref{EquK}), that means
\begin{equation}
\int_{\Omega}\left\vert \nabla T_{K}(u)\right\vert ^{p-2}\nabla T_{K}%
(u).\nabla\psi dx=\int_{\left\{  U<K\right\}  }\beta(u)\left\vert \nabla
u\right\vert ^{p}\psi dx+\int_{\left\{  U<K\right\}  }\lambda f\psi
dx+\int_{\Omega}\psi d\alpha_{K}, \label{tfu}%
\end{equation}
for any $\psi\in W_{0}^{1,p}(\Omega)\cap L^{\infty}(\Omega).$ Taking
$\psi=e^{\gamma(T_{K}(U))}\varphi,$ with $\varphi\in W_{0}^{1,p}(\Omega)\cap
L^{\infty}(\Omega),$
\begin{align}
\int_{\Omega}\left\vert \nabla T_{k}(v)\right\vert ^{p-2}\nabla T_{k}%
(v).\nabla\varphi dx  &  =\int_{\left\{  U<K\right\}  }\lambda fe^{\gamma
(T_{K}(u))}\varphi dx+\int_{\Omega}e^{\gamma(T_{K}(u))}\varphi d\alpha
_{K},\nonumber\\
&  =\int_{\left\{  v<k\right\}  }\lambda f(1+g(v))^{p-1}\varphi
dx+(1+g(k))^{p-1}\int_{\Omega}\varphi d\alpha_{K}\nonumber\\
&  =\int_{\left\{  v<k\right\}  }\lambda f(1+g(v))^{p-1}\varphi dx+\int
_{\Omega}\varphi d\mu_{k}, \label{tfv}%
\end{align}
or equivalently (\ref{eqvk}) holds, where $\mu_{k}$ is given by (\ref{amu}).
If $L<\infty,$ then $\alpha_{s}=0$ from Remark \ref{pos}, and $u=T_{L}(u)\in
W_{0}^{1,p}(\Omega)\cap L^{\infty}(\Omega).$ And $T_{K}(u)$ converges to $u$
strongly in $W_{0}^{1,p}(\Omega)$ as $K\nearrow L.$ Then $\Delta_{p}T_{K}(u)$
converges to $\Delta_{p}u$ in $W^{-1,p^{\prime}}(\Omega)$, and $\beta
(T_{K}(u))\left\vert \nabla\,T_{K}(u)\right\vert ^{p}$ converges to
$\beta(u))\left\vert \nabla\,u)\right\vert ^{p}$ and $\lambda f\chi_{\left\{
u<K\right\}  }$ converges to $\lambda f$ strongly in $L^{1}\left(
\Omega\right)  .$ Taking $\psi=T_{K}(u)$ in (\ref{tfu}), it follows that
\begin{align*}
\lim_{K\nearrow L}K\alpha_{K}(\Omega)  &  =\lim_{K\nearrow L}(\int_{\Omega
}\left\vert \nabla T_{K}(u)\right\vert ^{p}dx-\int_{\Omega}\beta
(T_{K}(u))\left\vert \nabla T_{K}(u)\right\vert ^{p}dx-\int_{\left\{
u<K\right\}  }\lambda fT_{K}(u)dx)\\
&  =\int_{\Omega}\left\vert \nabla u\right\vert ^{p}dx-\int_{\Omega}%
\beta(u)\left\vert \nabla u\right\vert ^{p}dx-\int_{\Omega}\lambda fudx=0,
\end{align*}
thus $\alpha_{K}$ converges to $0$ in the narrow topology as $K\nearrow L$.
Hence in any case ($L$ finite or not), $\alpha_{K}$ converges to $\alpha_{s}$
in the narrow topology as $K\nearrow L.\medskip$
\end{proof}

\begin{proof}
[Proof of Theorem \ref{TP}](i) Let $v$ be a solution of (\ref{PS}), where
$0\leqq v(x)<\Lambda$ a.e. in $\Omega,$ and $u=H(v).$ Taking $\varphi_{k}=$
$1-1/(1+g(T_{k}(s)))^{p-1}$, as a test function in (\ref{eqvk}), we find
\begin{align*}
\int_{\left\{  u<K\right\}  }\beta(u)\left\vert \nabla\,u\right\vert ^{p}dx
&  =(p-1)\int_{\left\{  v<k\right\}  }\frac{\,g^{\prime}(v)\left\vert
\nabla\,v)\right\vert ^{p}}{(1+g(v))^{p}}dx=\lambda\int_{\left\{  v<k\right\}
}f(1+g(v))^{p-1}\varphi_{k}dx+\phi(k)\int_{\Omega}d\mu_{k}\\
&  \leqq\lambda\int_{\Omega}f(1+g(v))^{p-1}dx+\int_{\Omega}d\mu_{k}\leqq C
\end{align*}
where $C>0$ is independent of $k;$ then $\beta(u)\left\vert \nabla
\,u\right\vert ^{p}\in\,L^{1}(\Omega)$. And from (\ref{amu}), $\alpha_{K}$
converges in the narrow topology to a singular measure $\alpha_{s}$: either
lim$_{k\longrightarrow\infty}g(k)=\infty,$ equivalently $L<\infty$ or
$L=\infty,$ $\beta\not \in L^{1}((0,\infty)),$ and then $\alpha_{s}=0;$ or $g$
is bounded, equivalently $L=\infty$ and $\beta\in L^{1}((0,\infty))$ and then
$\alpha_{s}=(1+g(\infty))^{p-1}\mu_{s}.$

Since $T_{k}(u)\in W_{0}^{1,p}(\Omega),$ it is also a renormalized solution of
equation (\ref{EquK}). From \cite[Theorem 3.4]{DMOP} there exists a
subsequence converging to a renormalized solution $U$ of
\[
-\Delta_{p}U=\beta(u)\left\vert \nabla\,u\right\vert ^{p}+\lambda f+\alpha_{s}%
\]
and $T_{k}(u)$ converges a.e in $\Omega$ to $u,$ thus (the quasicontinuous
representative of) $U$ is equal to $u.$ Then $u$ is solution of (\ref{PA}%
).$\medskip$

(ii) Let $u$ be a solution of (\ref{PA}), where $0\leqq u(x)<L$ a.e. in
$\Omega,$ and $v=\Psi(u).$ Taking $\psi=e^{\gamma(T_{K}(u))}-1=(1+g(T_{k}%
(v))^{p-1}-1$ as a test function if (\ref{EquK}), we get after simplification
\begin{multline}
\int_{\Omega}\beta(T_{K}(u))\left\vert \nabla T_{K}(u)\right\vert ^{p}%
dx=\int_{\left\{  u<K\right\}  }\lambda f\psi dx+\int_{\Omega}\psi d\alpha
_{K}\nonumber\\
=\int_{\left\{  v<k\right\}  }\lambda f((1+g(v))^{p-1}-1)dx+((1+g(k))^{p-1}%
-1)\int_{\Omega}d\alpha_{K}\nonumber\\
=\int_{\left\{  v<k\right\}  }\lambda f((1+g(v))^{p-1}-1)dx+\mu_{k}%
(\Omega)-\alpha_{K}(\Omega). \label{kok}%
\end{multline}
Since $\beta(u)\left\vert \nabla u\right\vert ^{p}\in L^{1}\left(
\Omega\right)  ,$ then $\phi=f(1+g(v))^{p-1}\in L^{1}\left(  \Omega\right)  ,$
and the measures $\mu_{k}$ are bounded independently of $k.$ There exists a
sequence $(k_{n})$ converging to $\Lambda$ such that ($\mu_{k_{n}})$ converges
weakly$^{\ast}$ to a measure $\mu.$ Let $v_{n}=T_{k_{n}}(v),$ then ($v_{n})$
converges to $v=\Psi(u)$ a.e. in $\Omega.$ From \cite[Section 5.1]{DMOP}
applied to $v_{n}=T_{k_{n}}(v),$ solution of (\ref{eqvk}) for $k=k_{n},$
$\left\vert \nabla v_{n}\right\vert ^{p-1}$ is bounded in $L^{\tau}\left(
\Omega\right)  $ for any $\tau<N/(N-1),$ and $\nabla v_{n}$ converges to
$\nabla v$ a.e. in $\Omega,$ and $\left\vert \nabla v_{n}\,\right\vert
^{p-2}\nabla v_{n}$ converges to $\left\vert \nabla v\,\right\vert
^{p-2}\nabla v$ strongly in $L^{\tau}\left(  \Omega\right)  $ for any
$\tau<N/(N-1).$ And $\lambda f((1+g(v))^{p-1}\chi_{\left\{  v<k_{n}\right\}
}$ converges to $\phi$ strongly in $L^{1}\left(  \Omega\right)  $ from the
Lebesgue Theorem. Then $v$ satisfies
\begin{equation}
-\Delta_{p}v=\phi+\mu\qquad\text{ in }\mathcal{D}^{\prime}(\Omega);
\label{eqd}%
\end{equation}
thus $\mu$ is uniquely determined, and $\mu_{k}$ converges weakly$^{\ast}$ to
$\mu$ as $k\nearrow\Lambda.$ Then $v$ is \textbf{reachable} solution of this
equation. Let us set $M=$ $\phi+\mu.\medskip$

\textbf{Case }$p=2$ or $p=N.$\textbf{ }Then from uniqueness, $v$ is a
renormalized solution of (\ref{eqd}). There exists $m\in\mathcal{M}_{0}%
^{+}(\Omega)$ et $\eta_{s}\in\mathcal{M}_{s}^{+}(\Omega)$ such that
$M=m+\eta_{s},$ and from the definition of renormalized solution, for any
$k>0,$ there exists $\eta_{k}\in\mathcal{M}_{0}^{+}(\Omega)$ concentrated on
the set $\left\{  v=k\right\}  ,$ converging to $\eta_{s}$ in the narrow
topology, and
\[
-\Delta_{p}T_{k}(v)=m\llcorner_{\left\{  v<k_{n}\right\}  }+\eta_{k_{n}}%
\qquad\text{in }\mathcal{D}^{\prime}(\Omega),
\]
but we have also
\[
-\Delta_{p}T_{k_{n}}(v)=\lambda f(1+g(v))^{p-1}\chi_{\left\{  v<k_{n}\right\}
}+\mu_{k_{n}}\qquad\text{in }\mathcal{D}^{\prime}(\Omega),
\]
thus $\eta_{k_{n}}=\mu_{k_{n}},$ and $\mu=\eta_{s},$ and
\[
-\Delta_{p}v=f(1+g(v))^{p-1}+\eta_{s}\qquad\text{ in }\mathcal{D}^{\prime
}(\Omega);
\]
hence in the renormalized sense; and $\mu_{k_{n}}$ converges to $\eta_{s}$ in
the narrow topology.$\medskip$

\textbf{General case.} From \cite{DmMalu}, there exists $m\in\mathcal{M}%
_{0}(\Omega)$ and $\eta\in\mathcal{M}_{b}^{+}(\Omega)$ such that $M=m+\eta,$
and there exists a sequence $\left(  k_{n}\right)  $ tending to $\infty,$ such
that there exists $M_{k_{n}}\in W^{-1,p^{\prime}}(\Omega)\cap\mathcal{M}%
_{b}(\Omega)$ such that $-\Delta_{p}T_{k_{n}}(v)=M_{k_{n}}$ in $\mathcal{D}%
^{\prime}(\Omega),$ and $\eta_{k_{n}}=M_{k_{n}}\llcorner_{\left\{
v=k_{n}\right\}  }\in\mathcal{M}_{0}^{+}(\Omega)$ and $M_{k_{n}}%
=m\llcorner_{\left\{  v<k_{n}\right\}  }+\eta_{k_{n}},$ and $\left(
\eta_{k_{n}}\right)  $ converges weakly* to $M;$ and for any for any $h\in
W^{1,\infty}(\mathbb{R})$ such that $h^{\prime}$ has a compact support, and
any $\varphi\in\mathcal{D}(\Omega)$
\[
\int_{\Omega}\left\vert \nabla v\right\vert ^{p-2}\nabla v.\nabla
(h(v)\varphi)dx=\int_{\Omega}h(v)\varphi dm+h(\infty)\int_{\Omega}\varphi
d\eta.
\]
But $M_{k_{n}}=m\llcorner_{\left\{  v<k_{n}\right\}  }+\eta_{k_{n}}=\phi
\chi_{\left\{  v<k_{n}\right\}  }+\mu_{k_{n}},$ hence $\eta_{k_{n}}=\mu
_{k_{n}}$ and $m\llcorner_{\left\{  v<k_{n}\right\}  }=\phi\chi_{\left\{
v<k_{n}\right\}  },$ and $\left\{  v=\infty\right\}  $ is of capacity $0,$
thus $m=\phi,$ and $\mu=\eta,$ thus (\ref{PSG})holds.

Moreover if $\Lambda<\infty$, then $L<\infty,$ and $u,v\in W_{0}^{1,p}%
(\Omega)\cap L^{\infty}(\Omega)$, and $\mu=\alpha=0$, and $u,v$ are
variational solution of (PU$\lambda),$ and (PV$\lambda).$ If $g$ is bounded,
in particular if $L=\Lambda=\infty$ and $\beta\in L^{1}((0,\infty)),$ then
$\mu=e^{\gamma(\infty)}\alpha_{s},$ thus $\mu$ is singular; and $\mu_{k}$
converges in the narrow topology, thus $v$ is a renormalized solution of
(\ref{PS}). If $\Lambda=\infty$ and $\beta\not \in L^{1}((0,\infty)),$ then
$\alpha_{s}=0$ from (\ref{amu})$.$
\end{proof}

\section{The case $\beta$ constant, g linear\label{lin}}

We begin by the case of problems (\ref{M}) and (\ref{E5}), where
$f\not \equiv 0,$ and
\[
\beta(u)\equiv p-1,\text{ or equivalently }g(v)=v,
\]
where the eigenvalue $\lambda_{1}(f)$ defined at (\ref{VP}) is involved.

\subsection{ Some properties of $\lambda_{1}(f)$}

(i) Let $f\in L^{1}(\Omega)$, $f\geqq0,$ $f\not \equiv 0,$ such that
$\lambda_{1}(f)>0$. Let $C>0.$ Then for any $\varepsilon>0$, there exists
$K_{\varepsilon}=K_{\varepsilon}(\varepsilon,p,C)>0$ such that for any $v,w\in
W_{0}^{1,p}(\Omega),$
\begin{equation}
\int_{\Omega}f(C+\left\vert v\right\vert )^{p}dx\leqq(1+\varepsilon
)\int_{\Omega}f\left\vert v\right\vert ^{p}dx+K_{\varepsilon}\leqq
\frac{1+\varepsilon}{\lambda_{1}(f)}\left\Vert \nabla v\right\Vert
_{L^{p}(\Omega)}^{p}+K_{\varepsilon} \label{epsi}%
\end{equation}%
\begin{align}
\int_{\Omega}f(C+\left\vert v\right\vert )^{p-1}\left\vert w\right\vert dx  &
\leqq\left(  \int_{\Omega}f(C+\left\vert v\right\vert )^{p}dx\right)
^{1/p^{\prime}}\left(  \int_{\Omega}f\left\vert w\right\vert ^{p}dx\right)
^{1/p}\nonumber\\
&  \leqq\frac{1}{\lambda_{1}(f)^{1/p}}\left(  \int_{\Omega}f(C+\left\vert
v\right\vert )^{p}dx\right)  ^{1/p^{\prime}}\left\Vert \nabla w\right\Vert
_{L^{p}(\Omega)}. \label{majo}%
\end{align}
Thus $f(C+\left\vert v\right\vert )^{p-1}\in W^{-1,p^{\prime}}(\Omega)\cap
L^{1}(\Omega)$, in particular $f\in W^{-1,p^{\prime}}(\Omega),$ and with new
$\varepsilon$ and $K_{\varepsilon},$
\begin{equation}
\left\Vert f(C+\left\vert v\right\vert )^{p-1}\right\Vert _{W^{-1,p^{\prime}%
}(\Omega)}\leqq\frac{1+\varepsilon}{\lambda_{1}(f)}\left\Vert \nabla
v\right\Vert _{L^{p}(\Omega)}^{p-1}+K_{\varepsilon}. \label{lon}%
\end{equation}

\noindent(ii) If $f\in L^{r}(\Omega),$ with $r\geqq N/p>1,$ or $r>1=N/p,$ then
$\lambda_{1}(f)>0,$ and $\lambda_{1}(f)$ is attained at some first nonnegative
eigenfunction $\phi_{1}\in W_{0}^{1,p}(\Omega)$ of problem (\ref{E15}), from
\cite{LuPr}. If $r>N/p,$ then $\phi_{1}\in L^{\infty}(\Omega),$ from
Proposition \ref{cig}, and $\phi_{1}$ is locally H\"{o}lder continuous, from
\cite{Cu}. If $r=N/p>1$, then $\phi_{1}\in L^{k}(\Omega)$ for any $k\geqq1.$
\medskip

\noindent(iii) If $0\in\Omega,$ $p<N$ and $f(x)=1/\left\vert x\right\vert
^{p},$ then $f\not \in L^{N/p}(\Omega),$ but $\lambda_{1}(f)>0$ from the Hardy
inequality, given by $\lambda_{1}(f)=((N-p)/p)^{p}$ and $\lambda_{1}(f)$ is
not attained.

\subsection{Proof of Theorem \ref{T2}}

Theorem \ref{T2} is a consequence of Theorem \ref{TP} and of the two following
results. The first one is relative to the case without measure:

\begin{theorem}
\label{sim}If $\lambda>\lambda_{1}(f)\geqq0,$ or $\lambda=\lambda_{1}(f)>0$
and $f\in L^{N/p}(\Omega),p<N$, then problem
\begin{equation}
-\Delta_{p}v=\lambda f(1+v)^{p-1}\hspace{0.5cm}\text{in }\Omega,\qquad
v=0\hspace{0.5cm}\text{on }\partial\Omega. \label{PL}%
\end{equation}
admits no renormalized solution, and problem (\ref{M}) has no solution. If
$0<\lambda<\lambda_{1}(f)$ there exists a unique positive solution $v\in
W_{0}^{1,p}(\Omega).$ If moreover $f\in L^{r}\left(  \Omega\right)  ,r>N/p$,
then $v\in L^{\infty}(\Omega).$ If $f\in L^{N/p}(\Omega),$ $p<N,$ then $v\in
L^{k}(\Omega)$ for any $k>1.$
\end{theorem}

\begin{proof}
(i) If (\ref{PL}) has a solution then also problem (\ref{M}) has a solution
$u\in\mathcal{W},$ from Theorem \ref{TP}. And $u\in W_{0}^{1,p}(\Omega)$ from
Remark \ref{dou}. Taking $\varphi=\psi^{p}$ with $\psi\in\mathcal{D}^{\prime
}\left(  \Omega\right)  ,$ $\psi\geqq0$ as a test function we obtain
\begin{align*}
(p-1)\int_{\Omega}\left\vert \nabla u\right\vert ^{p}\psi^{p}dx+\lambda
\int_{\Omega}f\psi^{p}dx  &  =p\int_{\Omega}\psi^{p-1}\left\vert \nabla
u\right\vert ^{p-2}\nabla u.\nabla\psi dx\\
&  \leqq\int_{\Omega}\left\vert \nabla\psi\right\vert ^{p}dx+(p-1)\int
_{\Omega}\left\vert \nabla u\right\vert ^{p}\psi^{p}dx;
\end{align*}
then from the Young inequality,
\[
\lambda\int_{\Omega}f\psi^{p}dx\leqq\int_{\Omega}\left\vert \nabla
\psi\right\vert ^{p}dx;
\]
by density we obtain that $\lambda\leqq\lambda_{1}(f)$. In particular if
$\lambda_{1}(f)=0$ there is no solution for $\lambda>0.\medskip$

\noindent(ii) Assume $\lambda=\lambda_{1}(f)>0$ and $f\in L^{N/p}(\Omega
),p<N$. Taking an eigenfunction $\phi_{1}\in W_{0}^{1,p}(\Omega)$ as above, we
get
\begin{equation}
\int_{\Omega}\left\vert \nabla\phi_{1}\right\vert ^{p}dx=\lambda_{1}%
(f)\int_{\Omega}f\phi_{1}^{p}dx. \label{E24}%
\end{equation}
Consider a sequence of nonnegative functions $\psi_{n}\in\mathcal{D}(\Omega)$
converging to $\phi_{1}$ strongly in $W_{0}^{1,p}(\Omega)$. Taking $\psi
_{n}^{p}\in W_{0}^{1,p}(\Omega)\cap L^{\infty}(\Omega)$ as a test function, we
find
\begin{equation}
(p-1)\int_{\Omega}\left\vert \nabla u\right\vert ^{p}\psi_{n}^{p}%
dx+\lambda_{1}(f)\int_{\Omega}f\psi_{n}^{p-1}dx=p\int_{\Omega}\psi_{n}%
^{p-1}\left\vert \nabla u\right\vert ^{p-2}\nabla u.\nabla\psi_{n}dx.
\label{gar}%
\end{equation}
For any function $\phi\in W_{0}^{1,p}(\Omega),$ we set
\[
L(u,\phi):=(p-1)\left\vert \nabla u\right\vert ^{p}\phi^{p}+\left\vert
\nabla\phi\right\vert ^{p}-p\phi^{p-1}\left\vert \nabla u\right\vert
^{p-2}\nabla u.\nabla\phi,
\]%
\[
L_{1}(u,\phi):=(p-1)\left\vert \nabla u\right\vert ^{p}\phi^{p}+\left\vert
\nabla\phi\right\vert ^{p}-p\phi^{p-1}\left\vert \nabla u\right\vert
^{p-1}\left\vert \nabla\phi\right\vert .
\]
Thus $0\leqq L_{1}(u,\phi)\leqq L(u,\phi).$ From (\ref{gar}),
\[
\int_{\Omega}L_{1}(u,\psi_{n})dx+\lambda_{1}(f)\int_{\Omega}f\psi_{n}%
^{p}dx\leqq\int_{\Omega}L(u,\psi_{n})dx+\lambda_{1}(f)\int_{\Omega}f\psi
_{n}^{p}dx=\int_{\Omega}\left\vert \nabla\psi_{n}\right\vert ^{p}dx;
\]
then from the Fatou Lemma applied to a subsequence,
\[
\int_{\Omega}L_{1}(u,\phi_{1})dx+\lambda_{1}(f)\int_{\Omega}f\phi_{1}%
^{p}dx\leqq\int_{\Omega}L(u,\phi_{1})dx+\lambda_{1}(f)\int_{\Omega}f\phi
_{1}^{p}dx=\int_{\Omega}\left\vert \nabla\phi_{1}\right\vert ^{p}dx,
\]
hence from (\ref{E24}), we obtain $L_{1}(u,\phi_{1})=L(u,\phi_{1})=0$ a.e. in
$\Omega.$ Then
\[
\phi_{1}\left\vert \nabla u\right\vert =(p-1)\left\vert \nabla\phi
_{1}\right\vert ,\text{ and }\left\vert \nabla u\right\vert ^{p-2}\nabla
u.\nabla\phi_{1}=\left\vert \nabla u\right\vert ^{p-1}\left\vert \nabla
\phi_{1}\right\vert \text{\quad\quad a.e. in }\Omega,
\]%
\[
\nabla u=(p-1)\frac{\nabla\phi_{1}}{\phi_{1}}=\nabla({\ln}(\phi_{1}%
^{p-1}))\text{\quad\quad{a.e. in }}{\Omega;}%
\]
then $u=\ln(\phi_{1}^{p-1})+k$, or $\phi_{1}^{p-1}=e^{u-k}\geqq e^{-k}$ a.e.
in $\Omega$, which is contradictory.\medskip

\noindent(iii) Assume that $0<\lambda<\lambda_{1}(f).$ Then $f\in
W^{-1,p^{\prime}}(\Omega)$ from above, thus $v_{1}=\mathcal{G}(\lambda f)\in
W_{0}^{1,p}(\Omega)$ and $v_{1}\geqq0$, see Remark \ref{cp}, and
$f(1+v_{1})^{p-1}\in W^{-1,p^{\prime}}(\Omega)$. By induction we define
$v_{n}=\mathcal{G}(\lambda f(v_{n-1}+1)^{p-1}\in W_{0}^{1,p}(\Omega)$, then
\begin{equation}
-\Delta_{p}v_{n}=\lambda f(v_{n-1}+1)^{p-1}\qquad\text{in }W^{-1,p^{\prime}%
}(\Omega). \label{E19aa}%
\end{equation}
Taking $v_{n}$ as test function in (\ref{E19aa}), then from (\ref{epsi}),
\[
\int_{\Omega}\left\vert \nabla v_{n}\right\vert ^{p}dx=\lambda\int_{\Omega
}f(v_{n-1}+1)^{p-1}v_{n}dx\leqq\lambda\int_{\Omega}f(v_{n}+1)^{p}dx\leqq
\frac{\lambda(1+\varepsilon)}{\lambda_{1}(f)}\int_{\Omega}\left\vert \nabla
v_{n}\right\vert ^{p}dx+\lambda K_{\varepsilon}.
\]
Taking $\varepsilon>0$ small enough, it follows that ($v_{n})$ is bounded in
$W_{0}^{1,p}(\Omega)$. The sequence is nondecreasing, thus it converges weakly
in $W_{0}^{1,p}(\Omega)$, and a.e. in $\Omega$ to $v=\sup v_{n}.$ For any
$w\in W_{0}^{1,p}(\Omega)$, $\left\vert f(v_{n-1}+1)^{p-1}w\right\vert \leqq
f(1+v)^{p-1}\left\vert w\right\vert =h$ and $h\in L^{1}(\Omega)$ from
(\ref{majo}), thus $f(v_{n-1}+1)^{p-1}$ converges to $f(1+v)^{p-1}$ weakly in
$W^{-1,p^{\prime}}(\Omega)$. Then $v$ is solution of (\ref{PL}), by compacity
of $(-\Delta_{p})^{-1},$ see \cite{P}. The regularity follows from Proposition
\ref{cig} (iii).\medskip

Uniqueness is based on Lemma \ref{Pic}. Let $v,\hat{v}\in W_{0}^{1,p}(\Omega)$
be two nonnegative solutions. Then $v\not \equiv 0$ and $\hat{v}\not \equiv 0$
since $f\not \equiv 0$. Since $-\Delta_{p}v\in W^{-1,p^{\prime}}(\Omega)\cap
L^{1}(\Omega)$ and $(-\Delta_{p}v)v\geqq0$, we obtain $\int_{\Omega}%
(-\Delta_{p}v)vdx=\int_{\Omega}\left\vert \nabla v\right\vert ^{p}dx,$ hence
\begin{equation}
\int_{\Omega}(\frac{-\Delta_{p}v}{v^{p-1}}+\frac{\Delta_{p}\hat{v}}{\hat
{v}^{p-1}})v^{p}dx\geqq0;\qquad\int_{\Omega}(-\frac{\Delta_{p}\hat{v}}{\hat
{v}^{p-1}}+\frac{-\Delta_{p}v}{v^{p-1}})\hat{v}^{p-1}dx\geqq0; \label{int}%
\end{equation}
but
\[
\int_{\Omega}(\frac{-\Delta_{p}v}{v^{p-1}}+\frac{\Delta_{p}\hat{v}}{\hat
{v}^{p-1}})(v^{p}-\hat{v}^{p})dx=\lambda\int_{\Omega}f((1+\frac{1}{v}%
)^{p-1}-(1+\frac{1}{\hat{v}})^{p-1})(v^{p}-\hat{v}^{p})dx\leqq0;
\]
then the two integrals in (\ref{int}) are zero, hence
\[
\int_{\Omega}(\left\vert \nabla v\right\vert ^{p}-p\frac{\hat{v}^{p-1}%
}{v^{p-1}}\left\vert \nabla\hat{v}\right\vert ^{p-2}\nabla\hat{v}.\nabla
v+(p-1)\left\vert \nabla\hat{v}\right\vert ^{p}\frac{\hat{v}^{p}}{v^{p}%
})dx=0,
\]
thus $v=k\hat{v}$ for some $k>0$. Then $f((1+kv)^{p-1}-(k+kv)^{p-1})=0$ a.e.
in $\Omega,$ thus $k=1.$\medskip
\end{proof}

The second result is valid for measures which are not necessarily singular; it
extends \cite[Theorem 2.6]{AAP} relative to $p=2.$ The proof of a more general
result will be given at Theorem \ref{exa}:

\begin{theorem}
\label{Y}If $0<\lambda<\lambda_{1}(f),$ for \textbf{any} measure $\mu
\in\mathcal{M}_{b}^{+}\left(  \Omega\right)  ,$ there exists at least one
renormalized solution $v\geqq0$ of problem
\[
-\Delta_{p}v=\lambda f(1+v)^{p-1}+\mu\hspace{0.5cm}\text{in }\Omega,\qquad
v=0\hspace{0.5cm}\text{on }\partial\Omega.
\]

\end{theorem}

\section{Problem (PV$\lambda)$ without measures \label{PVW}}

Next we study problem (PV$\lambda)$ for a general function $g.$

\subsection{The range of existence for $\lambda$}

The existence of solutions of (PV$\lambda)$ depends on the assumptions on $g$
and $f$ and the value of $\lambda.$ We will sometimes make growth assumptions
on $g$ of the form (\ref{hmq}) for some $Q>0$ and then our assumptions on $f$
will depend on $Q.$\ \medskip

We begin by a simple existence result, where $g$ only satisfies (\ref{hypo}),
$\Lambda\leqq\infty,$ with no growth condition, under a weak assumption on
$f,$ satisfied in particular when $f\in L^{r}\left(  \Omega\right)  ,r>N/p.$

\begin{proposition}
\label{elin}Assume (\ref{hypo}) and $\mathcal{G}(f)\in L^{\infty}(\Omega)$.
Then for $\lambda>0$ small enough, problem (PV$\lambda)$ has a minimal bounded
solution $\underline{v}_{\lambda}$ such that $\left\Vert v\right\Vert
_{L^{\infty}(\Omega)}<\Lambda.$ \medskip

Conversely, if $L=H(\Lambda)<\infty,$ (in particular if $\Lambda<\infty)$ and
if there exists $\lambda>0$ such (PV$\lambda)$ has a renormalized solution,
then $\mathcal{G}(f)\in L^{\infty}(\Omega)$.
\end{proposition}

\begin{proof}
Let $w=\mathcal{G}(f)\in W_{0}^{1,p}(\Omega)\cap L^{\infty}(\Omega).$ Let
$a>0$ such that $a\left\Vert w\right\Vert _{L^{\infty}(\Omega)}<\Lambda.$ Let
$\lambda_{0}=a((1+g(a\left\Vert w\right\Vert _{L^{\infty}(\Omega)}%
)))^{-(p-1)}$ and $\lambda\leqq\lambda_{0}$ be fixed. Then
\[
-\Delta_{p}(aw)=af(x)=\lambda_{0}((1+g(a\left\Vert w\right\Vert _{L^{\infty
}(\Omega)})))^{(p-1)}\geqq\lambda(1+g(aw))^{p-1}%
\]
since $g$ is nondecreasing. Between the subsolution $0$ and the supersolution
$aw,$ there exists a minimal solution $\underline{v}_{\lambda}$ obtained as
the nondecreasing limit of the iterative scheme
\begin{equation}
v_{n}=\mathcal{G(}\lambda f(x)(1+g(v_{n-1}))^{p-1}),n\geqq1. \label{she}%
\end{equation}
Then $\left\Vert \underline{v}_{\lambda}\right\Vert _{L^{\infty}(\Omega)}\leqq
a\left\Vert w\right\Vert _{L^{\infty}(\Omega)}<\Lambda.$

Conversely, let $v$ be a solution of (PV$\lambda).$ Then $u=H(v)$ is a
solution of (PU$\lambda)$ and $L\geqq u\geqq\lambda^{1/(p-1)}\mathcal{G}(f)$
a.e. in $\Omega,$ hence $\mathcal{G}(f)\in L^{\infty}(\Omega).$
\end{proof}

\begin{remark}
The converse result is sharp. Take $f=1/\left\vert x\right\vert ^{p}$ with
$0\in\Omega,$ then $\mathcal{G}(f)\not \in L^{\infty}(\Omega)$. Hence if
$L<\infty$ there is no solution of (PV$\lambda)$ for any $\lambda>0;$ for
example, there is no solution of problem%
\[
-\Delta_{p}v=\frac{\lambda}{\left\vert x\right\vert ^{p}}(1+v)^{Q}%
\hspace{0.5cm}\text{in }\Omega,\qquad v=0\quad\text{on }\partial\Omega.
\]
for $Q>p-1.$ Otherwise from Theorem \ref{T2}, for $Q=p-1$ and $0<\lambda
<\lambda_{1}(f),$ there exists a solution; in that case $H(\infty)=\infty.$
\end{remark}

\begin{remark}
When $\Lambda<\infty,$ and $g$ has an asymptote at $\Lambda,$ it may exist
solutions with $\left\Vert v\right\Vert _{L^{\infty}(\Omega)}=\Lambda.$
Consider example $9$ of paragraph 3.2 with $p=2$ and $\Omega=B(0,1).$ Here
$1+g(v)=(1-v)^{-Q},$ $Q>0,$ and $\beta(u)=Q(1-(Q+1)u).$ For $\lambda
=2((N-2)Q+N)/(Q+1)^{2},$ problem (PU$\lambda)$ admits the solution
$u=(1-r^{2})/(Q+1)$. Then $v=\Psi(u)=1-r^{2/(Q+1)}\in W_{0}^{1,2}(\Omega),$
and $\left\Vert v\right\Vert _{L^{\infty}(\Omega)}=1.$
\end{remark}

The range of $\lambda\geqq0$ for which problem (PV$\lambda)$ has a solution
depends a priori on the regularity of the solutions. We introduce three
classes of solutions. In case $\Lambda<\infty$ the notion of solution includes
the fact that $0\leqq v(x)<\Lambda$ a.e. in $\Omega.$

\begin{definition}
(i) Let $S_{r}$ be the set of $\lambda\geqq0$ such that (PV$\lambda)$ has a
renormalized solution $v,$ that means $w\in\mathcal{W}.$

(ii) Let $S_{\ast}$ be the set of $\lambda\geqq0$ such that (PV$\lambda)$ has
a variational solution $v,$ that means $v\in W_{0}^{1,p}(\Omega).$

(iii) Let $S_{b}$ be the set of $\lambda\geqq0$ such that (PV$\lambda)$ has a
renormalized solution $v$ such that $\left\Vert v\right\Vert _{L^{\infty
}(\Omega)}<\Lambda.$
\end{definition}

\begin{remark}
The sets $S_{r},S_{\ast},$ $S_{b}$ are intervals:
\begin{equation}
S_{r}=\left[  0,\lambda_{r}\right)  ,\qquad S_{\ast}=\left[  0,\lambda^{\ast
}\right)  ,\qquad S_{b}=\left[  0,\lambda_{b}\right)  \qquad\text{with
}\lambda_{b}\leqq\lambda^{\ast}\leqq\lambda_{r}\leqq\infty. \label{bda}%
\end{equation}
Indeed if $\lambda_{0}$ belongs to some of these sets, and $v_{\lambda_{0}}$
is a solution of (PV$\lambda_{0})$, then $v_{\lambda_{0}}$ is a supersolution
of (PV$\lambda)$ for any $\lambda<\lambda_{0}.$ Between the subsolution $0$
and $v_{\lambda_{0}},$ there exists a minimal solution of (PV$\lambda)$,
obtained as the nondecreasing limit of the iterative scheme (\ref{she}%
).$\medskip$

In case $\Lambda=\infty,$ $S_{b}$ is the set of of $\lambda\geqq0$ such that
(PV$\lambda)$ has a solution $v\in W_{0}^{1,p}(\Omega)\cap L^{\infty}%
(\Omega).$ For any $\lambda<\lambda_{b}$ there exists a minimal bounded
solution $\underline{v}_{\lambda}$. And $\lambda_{b}\leqq\lambda^{\ast}$ since
any renormalized bounded solution is in $W_{0}^{1,p}(\Omega)$ from Remark
\ref{pos}.\medskip

In case $\Lambda<\infty,$ then $\lambda_{r}=\lambda^{\ast},$ since $S_{r}=$
$S_{\ast},$ from Remark \ref{pos}. Moreover $\lambda^{\ast}<\infty.$ Indeed
any solution $v$ of (PV$\lambda)$ satisfies $\lambda\mathcal{G}(f)\leqq
v<\Lambda$ a.e. in $\Omega$, and $\mathcal{G}(f)\not \equiv 0.$
\end{remark}

A main question is to know if $\lambda_{b}=\lambda^{\ast}=\lambda_{r},$ as it
is the case when $g(v)=v,$ from Theorem \ref{T2}, where $\lambda^{\ast
}=\lambda_{1}\left(  f\right)  .$ It was shown when $g$ is defined on $\left[
0,\infty\right)  $ and convex in \cite{BrCMR} for $p=2.$ The method used was
precisely based on the transformation $u=H(v)$, even if problem (PU$\lambda)$
was not introduced. By using the equations satisfied by truncations as in the
proof of Theorem \ref{TP}, we can extend the kea point of the proof:

\begin{theorem}
\label{cle} Let $g_{1},g_{2}\in C^{1}(\left[  0,\Lambda\right)  )$ be
nondecreasing, with $0<g_{2}\leqq g_{1}$ on $\left[  0,\Lambda\right)  .$ Let
$v\in\mathcal{W}$ $(\Omega)$ such that $-\Delta_{p}v\geqq0$ a.e. in $\Omega,$
and $0\leqq v<\Lambda$ a.e. in $\Omega.\ $ Set
\[
H_{1}(\tau)=\int_{0}^{\tau}\frac{ds}{g_{1}(s)},\qquad H_{2}(\tau)=\int
_{0}^{\tau}\frac{ds}{g_{2}(s)},
\]
Assume that
\begin{equation}
0\leqq g_{2}^{\prime}\circ H_{2}^{-1}\leqq g_{1}^{\prime}\circ H_{1}%
^{-1}\qquad\text{on }\left[  0,H_{1}\left(  \Lambda\right)  \right)  .
\label{phil}%
\end{equation}
Then $\bar{v}=H_{2}^{-1}(H_{1}(v))$ $\in\mathcal{W}$, and $\bar{v}\leqq v,$
and
\begin{equation}
-\Delta_{p}\bar{v}\geqq\left(  \frac{g_{2}(\bar{v})}{g_{1}(v)}\right)
^{p-1}(-\Delta_{p}v)\qquad\text{in }L^{1}(\Omega). \label{pheq}%
\end{equation}
\medskip
\end{theorem}

\begin{proof}
We can assume that $g_{1}(0)=1.$ Let $u=H_{1}(v),$ and $F=-\Delta_{p}v.$
Applying Theorem \ref{TP} with $g=g_{1}-1$ and $f=Fg_{1}(v)^{1-p}\leqq F,$ the
function $u$ is a renormalized solution of
\[
-\Delta_{p}u=\beta_{1}(u)\left\vert \nabla u\right\vert ^{p}+Fg_{1}%
(v)^{1-p}\hspace{0.5cm}\text{in }\Omega,\qquad u=0\quad\text{on }%
\partial\Omega,
\]
where $\beta_{1}(u)=(p-1)g_{1}^{\prime}(v)=(p-1)g_{1}^{\prime}(H_{1}%
^{-1}(u)).$ Let
\[
\bar{v}=H_{2}^{-1}(u)=(H_{2}^{-1}\circ H_{1})(v)
\]
then $\bar{v}\leqq v,$ because $g_{2}\leqq g_{1}.$ Moreover $g_{1}^{\prime
}(v)\geqq g_{2}^{\prime}(\bar{v}),$ thus we can write
\[
\beta_{1}(u)\left\vert \nabla u\right\vert ^{p}=(p-1)g_{2}^{\prime}(\bar
{v})\left\vert \nabla u\right\vert ^{p}+\eta=\beta_{2}(u)\left\vert \nabla
u\right\vert ^{p}+\eta,
\]
with $\eta\in L^{1}\left(  \Omega\right)  ,$ $\eta\geqq0;$ thus
\[
-\Delta_{p}u=\beta_{2}(u)\left\vert \nabla u\right\vert ^{p}+\bar{f}%
\]
with $\bar{f}=Fg_{1}(v)^{1-p}+\eta.$ From Lemma \ref{Tr}, the truncations
$T_{k}(v),T_{K}(u),T_{k}(\bar{v})$ satisfy respectively%
\begin{align*}
-\Delta_{p}T_{k}(v)  &  =F\chi_{\left\{  v\leqq k\right\}  }+\mu_{k},\\
-\Delta_{p}T_{K}(u)  &  =\beta_{1}(T_{K}(u))\left\vert \nabla\,T_{K}%
(u)\right\vert ^{p}+Fg_{1}(v)^{1-p}\chi\left\{  u\leqq K\right\}  +\alpha
_{K},\\
-\Delta_{p}T_{k}(\bar{v})  &  =\bar{f}g_{2}(\bar{v})^{p-1}\chi_{\left\{
\bar{v}\leqq k\right\}  }+\bar{\mu}_{k},
\end{align*}
in $\mathcal{D}^{\prime}(\Omega),$ where
\[
\alpha_{K}=g_{1}(v)^{1-p}\mu_{k},\qquad\bar{\mu}_{k}=(g_{2}(k)/g_{1}%
(k))^{p-1}\mu_{k}.
\]
As in the proof of Theorem \ref{TP}, we obtain $\bar{f}g_{2}(\bar{v})^{p-1}\in
L^{1}\left(  \Omega\right)  ,$ and $\bar{f}g_{2}(\bar{v})^{p-1}\chi_{\left\{
\bar{v}<k\right\}  }$ converges to $\bar{f}g_{2}(\bar{v})^{p-1}$ strongly in
$L^{1}\left(  \Omega\right)  .$ Moreover $\mu_{k}$ converges to $0$ in the
\medskip narrow topologyas $k\rightarrow\Lambda$, thus $\lim\mu_{k}\left(
\Omega\right)  =0;$ and $g_{2}(k)\leqq g_{1}(k),$ thus $\lim\bar{\mu}%
_{k}\left(  \Omega\right)  =0,$ and $\bar{\mu}_{k}$ converges to $0$ in the
\textbf{narrow} topology. Then from Theorem \ref{fund}, $\bar{v}$ is a
\textbf{renormalized} solution of
\[
-\Delta_{p}\bar{v}=\bar{f}g_{2}(\bar{v})^{p-1}\hspace{0.5cm}\text{in }%
\Omega,\qquad\bar{v}=0\quad\text{on }\partial\Omega.
\]
Then $-\Delta_{p}\bar{v}\in L^{1}(\Omega),$ and $\bar{v}$ satisfies
(\ref{pheq}).
\end{proof}

\begin{remark}
Assumption (\ref{phil}) is equivalent to the concavity of the function
$\overline{\phi}=H_{2}^{-1}\circ H_{1};$ and (\ref{pheq}) means that
\[
-\Delta_{p}\overline{\phi}(v)\geqq(\overline{\Phi}^{\prime}(v))^{p-1}%
(-\Delta_{p}v)\qquad\text{in }L^{1}\left(  \Omega\right)  .
\]
If we take \textbf{any} concave function $\phi$ this inequality is formal. For
the particular choice $\phi=\overline{\phi},$ the inequality is \textbf{not
formal}, since no measure appears.
\end{remark}

Our main result covers in particular Theorem \ref{truc}. Some convexity
assumptions are weakened:

\begin{theorem}
\label{impo}Let $g$ satisfying (\ref{hypo}), and $H$ be defined by (\ref{hg})
on $\left[  0,\Lambda\right)  $, and $f\in L^{1}\left(  \Omega\right)  .$ In
case $\Lambda=\infty,L=H(\Lambda)=\infty$ we suppose $f\in L^{r}\left(
\Omega\right)  $, $r>N/p.$ Assume that for some $\lambda>0$ there exists a
renormalized solution $v$ of
\[
-\Delta_{p}v=\lambda f(x)(1+g(v))^{p-1}\hspace{0.5cm}\text{in }\Omega,\qquad
v=0\quad\text{on }\partial\Omega
\]
such that $0\leqq v(x)<\Lambda$ a.e. in $\Omega.\medskip$

\noindent(i) Suppose that $H\times(1+g)$ is convex on $\left[  0,\Lambda
\right)  $, or that $g$ is convex near $\Lambda.$ Then for any $\varepsilon
\in\left(  0,1\right)  $ there exists a bounded solution $w,$ such that
$\left\Vert w\right\Vert _{L^{\infty}\left(  \Omega\right)  }<\Lambda$ of
\begin{equation}
-\Delta_{p}w=\lambda(1-\varepsilon)^{p-1}f(x)(1+g(w))^{p-1}\hspace
{0.5cm}\text{in }\Omega,\qquad w=0\quad\text{on }\partial\Omega. \label{peps}%
\end{equation}
In other words, $\lambda_{b}=\lambda^{\ast}=\lambda_{r}.\medskip$

\noindent(ii) Suppose that $g$ is convex on $\left[  0,\Lambda\right)  .$ Then
for any $\varepsilon\in\left(  0,1\right)  $ there exists a bounded solution
$w$ such that $\left\Vert w\right\Vert _{L^{\infty}\left(  \Omega\right)
}<\Lambda$ of
\begin{equation}
-\Delta_{p}w=\lambda f(x)(1+g(w)-\varepsilon)^{p-1}\hspace{0.5cm}\text{in
}\Omega,\qquad w=0\quad\text{on }\partial\Omega. \label{per}%
\end{equation}
In particular if $\lambda^{\ast}<\infty,$ for any $c>0,$ there exists no
solution of problem
\begin{equation}
-\Delta_{p}v=\lambda^{\ast}f(x)(1+g(v)+c)^{p-1}\hspace{0.5cm}\text{in }%
\Omega,\qquad v=0\quad\text{on }\partial\Omega. \label{chou}%
\end{equation}

\end{theorem}

\begin{proof}
(i) \textbf{First case:} $L=H(\Lambda)=\int_{0}^{\Lambda}ds/(1+g(s))<\infty
.\medskip$

\noindent$\bullet$ First suppose $H\times(1+g)$ convex on $\left[
0,\Lambda\right)  .$ We take $g_{1}=1+g$ and $g_{2}=(1-\varepsilon)g_{1}.$
Then
\[
H_{2}=H/(1-\varepsilon),\qquad H_{2}^{-1}(u)=H^{-1}((1-\varepsilon
)u)=\Psi((1-\varepsilon)u).
\]
Condition (\ref{phil}) is equivalent to $(1-\varepsilon)ug^{\prime}%
(\Psi((1-\varepsilon)^{1/(p-1)}u)\leqq ug^{\prime}(\Psi(u)).$ In terms of $u,$
it means that the function $u\longmapsto u\beta(u)$ is non decreasing; in
terms of $v$ it means that $H\times g^{\prime}$ is nondecreasing. This is true
when $H\times(1+g)$ is convex, since ($H\times(1+g))^{\prime}=1+H\times
g^{\prime}.$ Then from Proposition \ref{cle}, the function $\bar{v}%
=\Psi((1-\varepsilon)H(v))$ satisfies
\[
-\Delta_{p}\bar{v}\geqq\lambda(1-\varepsilon)^{p-1}f(x)(1+g(\bar{v}))^{p-1}.
\]
Thus there exists a solution $w$ of (\ref{peps}) such that $w\leqq\bar{v}.$
And $\bar{v}(x)\leqq\Psi((1-\varepsilon)L)<\Lambda$ a.e. in $\Omega,$ hence
$w$ is bounded, and moreover $\left\Vert w\right\Vert _{L^{\infty}\left(
\Omega\right)  }<\Lambda.\medskip$

\noindent$\bullet$ Next suppose $g$ convex on $\left[  A,\Lambda\right)  ,$
with $0\leqq A<\Lambda.$ Let $M=1+g(A).$ Taking $\varepsilon>0$ small enough,
we construct a convex nondecreasing function $g_{1}\in C^{1}(\left[
0,\Lambda\right)  )$ such that $g_{1}\geqq1+g,$ and
\[
g_{1}(s)=M\text{ on }\left[  0,A-c\right]  ,\quad g_{1}(s)\leqq
M(1+2\varepsilon)\text{ on }\left[  0,A+d\right]  ,\quad g_{1}(s)=1+g(s)\text{
on }\left[  A+d,\infty\right)  ,
\]
with $c=2\varepsilon M$ and $d\leqq2\varepsilon Mg^{\prime}(A)$: we use a
portion of circle tangent to the graph of $1+g$ and to the line of ordinate
$M$ ; in case $g^{\prime}(A)=0$ we take $g_{1}=1+g).$ We set $g_{2}%
=(1-\varepsilon)g_{1}.$ The function $\bar{v}=H_{2}^{-1}(H_{1}(v))=\Psi
_{1}((1-\varepsilon)H_{1}(v))$ satisfies%
\[
-\Delta_{p}\bar{v}\geqq\lambda f(x)F_{\varepsilon}^{p-1},\qquad\text{where
}F_{\varepsilon}=(1-\varepsilon)\frac{g_{1}(\bar{v})}{g_{1}(v)}(1+g(v)),
\]
and $\bar{v}\leqq v.$ On the set $\left\{  \bar{v}\leqq v\leqq A+d\right\}  ,$
we find $M\leqq g_{1}(\bar{v})$ and $g_{1}(v)\leqq M(1+2\varepsilon),$ thus
$F_{\varepsilon}\geqq(1-3\varepsilon)(1+g(\bar{v})).$ On the set $\left\{
A+d\leqq\bar{v}\leqq v\right\}  ,$ we get $g_{1}(\bar{v})=g_{1}(v)=1+g(v),$
thus $F_{\varepsilon}\geqq(1-\varepsilon)(1+g(\bar{v})).$ On the set $\left\{
\bar{v}\leqq A+d\leqq v\right\}  ,$ there holds $M\leqq g_{1}(\bar{v})\leqq
M(1+2\varepsilon)\leqq1+g(v),$ thus again $F_{\varepsilon}\geqq(1-\varepsilon
)(1+g(\bar{v})).$ Then again $-\Delta_{p}\bar{v}\geqq\lambda(1-\varepsilon
)^{p-1}f(x)(1+g(\bar{v}))^{p-1},$ and we conclude as above by replacing
$\varepsilon$ by $3\varepsilon.\medskip$

\textbf{Second case:} $L=\infty.$ Here $\bar{v}$ can be unbounded. Extending
\cite{BrCMR}, we perform a bootstrapp based on Lemma \ref{boot}. The function
$H_{1}$ is concave, thus
\begin{equation}
H_{1}(v)-H_{1}(\bar{v})\leqq(v-\bar{v})H_{1}^{\prime}(\bar{v}))=\frac
{v-\bar{v}}{g_{1}(\bar{v})}\leqq\frac{v}{g_{1}(\bar{v})} \label{hac}%
\end{equation}
and $H_{1}(\bar{v})=(1-\varepsilon)H_{1}(v),$ hence $\varepsilon(1+g(\bar
{v}))\leqq\varepsilon g_{1}(\bar{v})\leqq v/H_{1}(v)\leqq C(1+v)$ for some
$C>0.$ Then $(1+g(w))^{p-1}\in L^{\sigma}\left(  \Omega\right)  $ for any
$\sigma\in\left[  1,N/(N-p)\right)  .$ Since $f\in L^{r}\left(  \Omega\right)
$, $r>N/p,$ from H\"{o}lder inequality, there exists $m_{1}>1$ such that
$fg(w)^{p-1}\in L^{m_{1}}(\Omega).$ If $p=N,$ then $\bar{v}\in L^{\infty
}\left(  \Omega\right)  $ from Lemma \ref{boot} and we conclude as above. Next
assume $p<N.$ We can suppose $m_{1}<N/p.$ Setting $w_{1}=w,$ $w_{1}$ is a
solution of (PV$(1-\varepsilon)^{p-1}\lambda),$ and $-\Delta_{p}w_{1}\in
L^{m_{1}}(\Omega);$ from Lemma \ref{boot}, $w_{1}^{s}\in L^{1}(\Omega)$ with
$s=(p-1)Nm_{1}/(N-pm_{1}).$ Replacing $1+g$ by $(1-\varepsilon)(1+g)$ we
construct in the same way a solution $w_{2}$ of (PV$(1-\varepsilon
)^{2(p-1)}\lambda)$ such that $g(w_{2})\leqq C(1+w_{1})),$ By induction we
construct a solution $w_{n}$ of (PV$(1-\varepsilon)^{n(p-1)}\lambda)$ such
that $g(w_{n})\leqq C(1+w_{n-1})),$ thus $fg(w_{n})^{p-1}\in L^{m_{n}}%
(\Omega),$ with $1/m_{n}-1/r=1/m_{n-1}-p/N$ . There exists a finite $\bar{n}=$
$\bar{n}(r,p,N)$ such that $m_{\bar{n}}>N/p,$ thus $w_{\bar{n}+1}\in
L^{\infty}(\Omega)$ from Lemma \ref{boot}. Since $\varepsilon$ is arbitrary,
we obtain a bounded solution of (\ref{peps}).\medskip

(ii) Suppose that $g$ is convex on $\left[  0,\Lambda\right)  .$ We take
$g_{1}=1+g$ and $g_{2}=g_{1}-\varepsilon,$ then (\ref{phil}) is satisfied,
because $g^{\prime}$ is nondecreasing and $H_{1}\leqq H_{2}$. Then we
construct a solution $w$ of (\ref{per}), such that $w\leqq\bar{v}=H_{2}%
^{-1}(H_{1}(v)).$ Here we only find $w(x)\leqq\bar{v}(x)<L$ a.e. in $\Omega,$
by contradiction, but not $\left\Vert w\right\Vert _{L^{\infty}\left(
\Omega\right)  }<\Lambda.$ As above (\ref{hac}) holds. And $H_{1}%
(v)=H_{2}(\bar{v}),$ hence
\[
H_{1}(v)-H_{1}(\bar{v})=\varepsilon\int_{0}^{\bar{v}}\frac{ds}{g_{1}%
(s)(g_{1}(s)-\varepsilon)}ds\geqq\varepsilon\int_{0}^{\bar{v}}\frac{ds}%
{g_{1}(s)^{2}}ds
\]
Then there exists $C>0$ such that $H_{1}(v)-H_{1}(\bar{v})\geqq C\varepsilon,$
a.e. on the set $\left\{  \bar{v}>1\right\}  ,$ thus $g_{1}(\bar{v})\leqq
v/\varepsilon C(A)$ on this set. Hence there exists $C_{\varepsilon}>0$ such
that $\varepsilon g_{1}(v_{1})\leqq C_{\varepsilon}(1+v).$ Replacing $g$ by
$g-n\varepsilon,$ in a finite number of steps as above we find a solution of
(\ref{peps}), since $\varepsilon$ is arbitrary.\medskip

Assume that there exists a solution of (\ref{chou}) for some $c>0.$ Then%
\[
-\Delta_{p}v=\lambda^{\ast}(1+c)^{p-1}f(x)(1+g(v)/(1+c))^{p-1}\hspace
{0.5cm}\text{in }\Omega.
\]
Considering $g/(1+c)$ and $\varepsilon=c/2(1+c),$ there exists a bounded
solution $w$ such that $\left\Vert w\right\Vert _{L^{\infty}\left(
\Omega\right)  }<\Lambda,$ of
\[
-\Delta_{p}w=\lambda^{\ast}f(x)(1+g(w)+c/2)^{p-1}\hspace{0.5cm}\text{in
}\Omega,
\]
We take $\alpha>0$ small enough such that $\alpha\leqq c/2(1+\left\Vert
g(w)\right\Vert _{L^{\infty}\left(  \Omega\right)  })$. Then $w$ is a
supersolution of (PV$\lambda^{\ast}(1+\alpha)^{p-1}),$ thus there exists a
solution $y$ of this problem such that $\left\Vert y\right\Vert _{L^{\infty
}\left(  \Omega\right)  }<\Lambda,$ which contradicts the definition of
$\lambda^{\ast}.$
\end{proof}

\subsection{Cases where $g$ has a slow growth}

In the linear case $g(v)=v$, we have shown that $\lambda^{\ast}=\lambda
_{1}(f).$ Next we consider the cases where $g$ has a \textbf{slow growth},
that means $g$ satisfies (\ref{hmq}) for some $Q\in\left(  0,Q_{1}\right)
.\medskip$

First suppose that $g$ is \textbf{at most linear} near $\infty$ and show a
variant of Theorem \ref{sim}:

\begin{corollary}
\label{almo}Assume that $\Lambda=\infty,$ and $g$ satisfies (\ref{hmq}) with
$Q=p-1,$ that means
\begin{equation}
0\leqq M_{p-1}^{1/(p-1)}=\overline{\lim}_{\tau\longrightarrow\infty}%
\frac{g(\tau)}{\tau}<\infty, \label{maj1}%
\end{equation}
Then $\lambda^{\ast}\geqq M_{p-1}\lambda_{1}(f):$ if $M_{p-1}\lambda
<\lambda_{1}(f)$ there exists at least a solution $v\in W_{0}^{1,p}(\Omega)$
to problem (PV$\lambda$); if ($1+g(v))/v$ is decreasing, the solution is
unique.\medskip

If $f\in L^{r}\left(  \Omega\right)  ,r>N/p$, any solution satisfies $v\in
L^{\infty}(\Omega),$ thus $\lambda_{b}=\lambda^{\ast}$. If $f\in
L^{N/p}(\Omega)$ and $p<N,$ any solution $v\in L^{k}(\Omega)$ for any $k>1.$
\end{corollary}

\begin{proof}
Let $M>M_{p-1}$ such that $M\lambda<\lambda_{1}(f).$ There exists $A>0$ such
that $(1+g(s))^{p-1}\leqq M(A+s)^{p-1}$ on $\left[  0,\infty\right)  .$
Defining $v_{1}=\mathcal{G}(\lambda f)\in W_{0}^{1,p}(\Omega)$ as in the
linear case of Theorem \ref{sim}, and $v_{n}$ $=\mathcal{G}(\lambda
f(1+g(v_{n-1}))^{p-1})\in W_{0}^{1,p}(\Omega),$ we find from (\ref{lon})
\[
\int_{\Omega}\left\vert \nabla v_{n}\right\vert ^{p}dx\leqq\lambda
M\int_{\Omega}f(A+v_{n-1})^{p-1}v_{n}dx\leqq\frac{\lambda M(1+\varepsilon
)}{\lambda_{1}(f)}\int_{\Omega}\left\vert \nabla v_{n}\right\vert
^{p}dx+\lambda K_{\varepsilon},
\]
with a new $K_{\varepsilon}>0,$ and conclude as in the linear case. Uniqueness
follows from Lemma \ref{Pic}, and regularity from Proposition \ref{cig},
(iii).\medskip
\end{proof}

Corollary \ref{almo} obviously applies to the case where $g$ is
\textbf{sublinear} near $\infty,$ that means $g$ satisfies (\ref{hmq}) with
$Q<p-1$, and shows that if $\lambda_{1}(f)>0,$ then $\lambda^{\ast}=\infty$.
In fact existence of a renormalized solution can be obtained for some
functions $f$ without assuming $\lambda_{1}(f)>0,$ as it was observed in
\cite{PorSe}:

\begin{proposition}
\label{subli}Assume that $p<N,$ $\Lambda=\infty,$ and $g$ satisfies
(\ref{hmq}) with $Q\in\left(  0,p-1\right)  $ and $f\in L^{r}(\Omega)$,
$r\in(1,N/p),$ with $Qr^{\prime}<Q_{1}.\medskip$

Then for any $\lambda>0$ there exists a renormalized solution $v$ of
(PV$\lambda)$ such that $v^{d}\in L^{1}(\Omega)$ for $d=Nr(p-1-Q)/(N-pr).$ In
particular $\lambda_{r}=\infty.\medskip$

If ($Q+1)r^{\prime}\leqq p^{\ast}$, then $v\in W_{0}^{1,p}(\Omega),$ thus
$\lambda^{\ast}=\infty.\medskip$

If ($Q+1)r^{\prime}>p^{\ast},$ then $\left\vert \nabla v\right\vert ^{\theta
}\in L^{1}(\Omega)$ for $\theta=Nr(p-1-Q)/(N-(Q+1)r).$
\end{proposition}

\begin{proof}
Let $M>0$ such that $(1+g(t))^{p-1}\leqq M(1+t)^{^{Q}}$ for $t\geqq0.$ For any
fixed $n\in\mathbb{N},$ there exists $v_{n}\in W_{0}^{1,p}(\Omega)$ such that
\[
-\Delta_{p}v_{n}=\lambda T_{n}(f(x)(1+g(v_{n}))^{p-1}).
\]
It is obtained for example as the limit of the nondecreasing iterative sheme
$v_{n,k}=\mathcal{G}(\lambda T_{n}(f(x)(1+g(v_{n,k-1}))^{p-1}))$, $k\geqq1,$
$v_{n,0}=0.$ We take $\phi_{\beta}(v_{n})$ as a test function, where
$\phi_{\beta}(w)=\int_{0}^{w}(1+\left\vert t\right\vert )^{-\beta}dt$, for
given real $\beta<1.$ Setting $\alpha=1-\beta/p$ and $w_{n}=(1+v_{n})^{\alpha
}-1,$ we get
\[
\frac{1}{\alpha^{p}}\int_{\Omega}\left\vert \nabla w_{n}\right\vert
^{p}dx=\int_{\Omega}\frac{\left\vert \nabla v_{n}\right\vert ^{p}}%
{(1+v_{n})^{\beta}}dx\leqq(1-\beta)^{-1}\lambda M\int_{\Omega}f(1+v_{n}%
)^{1-\beta+Q}dx.
\]
From the Sobolev injection, There exists $C>0$ such that
\[
\left(  \int_{\Omega}w_{n}^{p^{\ast}}dx\right)  ^{p/p^{\ast}}\leqq
C\int_{\Omega}f(1+w_{n})^{(1-\beta+Q)/\alpha}dx\leqq C\left\Vert f\right\Vert
_{L^{1}\left(  \Omega\right)  }+C\left\Vert f\right\Vert _{L^{r}\left(
\Omega\right)  }\left(  \int_{\Omega}w_{n}^{(1-\beta+Q)r^{\prime}/\alpha
}dx\right)  ^{1/r^{\prime}}%
\]
Taking $\beta=((Q+1)r^{\prime}-p^{\ast})/(r^{\prime}-N/(N-p)<1,$ we find
$(1-\beta+Q)r^{\prime}/\alpha=p^{\ast}.$ Then ($w_{n})$ is bounded in
$W_{0}^{1,p}(\Omega),$ thus ($v_{n}^{d})$ is bounded in $L^{1}(\Omega).$ If
($Q+1)r^{\prime}\leqq p^{\ast}$ then $\beta\leqq0,$ thus ($v_{n})$ is bounded
in $W_{0}^{1,p}(\Omega).$ If ($Q+1)r^{\prime}>p^{\ast},$ then $\beta>0,$ and
\[
\int_{\Omega}\left\vert \nabla v_{n}\right\vert ^{\theta}dx\leqq\left(
\int_{\Omega}\frac{\left\vert \nabla v_{n}\right\vert ^{p}}{(1+v_{n})^{\beta}%
}dx\right)  ^{1/\theta}\left(  \int_{\Omega}(1+v_{n})^{d}dx\right)
^{\beta\theta/dp};
\]
thus ($\left\vert \nabla v_{n}\right\vert ^{\theta})$ is bounded in
$L^{1}(\Omega),$ where $\theta<p.$ Then ($f(x)g(v_{n}))^{p-1})$ is bounded in
$L^{\sigma}(\Omega)$ with $\sigma=rd(rQ+d)>1.$ From Remark \ref{conv}$,$ up to
a subsequence, ($v_{n})$ converges a.e. in $\Omega$ to a renormalized solution
of the problem with the same regularity.
\end{proof}

\begin{remark}
(i)The fact that $\lambda_{r}=\infty$ is much more general, as it will be
shown at Theorem \ref{exa}.

\noindent(ii) The regularity of the solution constructed at Proposition
\ref{subli} is a little stronger that the one exspected from Proposition
(\ref{cig}) (vi). We do not know if any solution of the problem has the same regularity.
\end{remark}

Our next result concerns any function $g$ with a slow growth, without
assumption of convexity. It is a direct consequence of Proposition \ref{cig}:

\begin{proposition}
\label{san}Assume that $\Lambda=\infty,$ and $g$ satisfies (\ref{hmq}) with
$Q\in\left[  p-1,Q_{1}\right)  $ and $f\in L^{r}(\Omega)$ with $Qr^{\prime
}<Q_{1}.\medskip$

Then any renormalized solution of (PV$\lambda)$ is in $W_{0}^{1,p}(\Omega)\cap
L^{\infty}(\Omega).$ Thus $\lambda_{b}=\lambda^{\ast}=\lambda_{r}.$
\end{proposition}

\begin{remark}
It holds in particular when $p=N,$ $g$ satisfies (\ref{hmq}) for some $Q\geqq
N-1$ and $f\in L^{r}(\Omega),r>1.$
\end{remark}

\subsection{Superlinear case: Extremal solutions}

In this paragraph we assume for simplicity that $g$ is defined on $\left[
0,\infty\right)  .$

\begin{definition}
Assume that $0<\lambda_{b}\leqq\lambda^{\ast}\leqq$ $\lambda_{r}<\infty.$ The
function
\[
v^{\ast}=\sup_{\lambda\nearrow\lambda_{b}}\underline{v}_{\lambda},
\]
where $\underline{v}_{\lambda}$ is the minimal bounded solution of
(PV$\lambda)$ is called extremal.\medskip\ 
\end{definition}

\begin{remark}
\label{cinq}Assume that $g$ is at least linear near $\infty:$ $\underline
{\lim}_{\tau\longrightarrow\infty}g(\tau)/\tau>0$ (it holds in particular when
$g$ is convex, $g\not \equiv 0$).\medskip

\noindent(i) Then $\lambda_{r}<\infty.$ Indeed there exists $c>0$ such that
$1+g(\tau)\geqq c(1+\tau)$ for any $\tau\in\left[  0,\infty\right)  .$ If
(PV$\lambda)$ has a solution, then there exists a solution of problem
\[
-\Delta_{p}v=\lambda c^{1/(p-1)}f(x)(1+v)^{p-1}\hspace{0.5cm}\text{in }%
\Omega,\qquad w=0\quad\text{on }\partial\Omega.
\]
Then $\lambda\leqq c^{-1/(p-1)}\lambda_{1}(f)$ from Theorem \ref{sim}.\medskip

\noindent(ii) The function $v^{\ast}$ is well defined with values in $\left[
0,\infty\right]  $ as soon as $\mathcal{G}(f)<\infty.$ For simplification, we
will assume in the main results that $f\in L^{r}\left(  \Omega\right)
,r>N/p.$\medskip\ 
\end{remark}

Next we study the case $g$ superlinear near $\infty:$%
\begin{equation}
g\in C^{1}(\left[  0,\infty\right)  ),g(0)=0\text{ and }g\text{ is
nondecreasing, and lim}_{s\longrightarrow\infty}\frac{g(s)}{s}=\infty.
\label{hypg}%
\end{equation}

Here the first question is to know if $v^{\ast}$\textit{ satisfies the limit
problem }(PV$\lambda_{b})$\textit{ and in what sense}. \medskip

The case $p=2$ was studied in \cite{BrCMR} for $g$ convex, with $f=1.$ In fact
the proof does uses the convexity, and extends to more general $f$ and we
recall it below.

\begin{lemma}
[\cite{BrCMR}]Assume $p=2$ and (\ref{hypg}), $f\in L^{r}\left(  \Omega\right)
,r>N/2$. Then $v^{\ast}$ is a very weak solution of (PV$\lambda_{b}),$ that
means $v^{\ast}\in L^{1}(\Omega)$, $g(v^{\ast})\in L^{1}(\Omega,\rho dx)$
where $\rho$ is the distance to $\partial\Omega,$ and
\begin{equation}
-\int_{\Omega}v^{\ast}\Delta\zeta dx=\int_{\Omega}fg(v^{\ast})\zeta
dx,\qquad\forall\zeta\in C^{2}\left(  \overline{\Omega}\right)  ,\zeta=0\text{
on }\partial\Omega. \label{vws}%
\end{equation}

\end{lemma}

\begin{proof}
Let $\lambda_{n}\nearrow\lambda_{b}$ and $v_{n}=\underline{v}_{\lambda_{n}};$
multiplying the equation relative to $v_{n}$ by a first eigenfunction
$\Phi_{1}>0$ of the Laplacian with the weight $f,$ one finds
\[
\lambda_{1}(f)\int_{\Omega}fv_{n}\Phi_{1}dx=\lambda_{n}\int_{\Omega
}f(1+g(v_{n}))\Phi_{1}dx;
\]
and the superlinearity of $g$ implies that $\int_{\Omega}f(1+g(v_{n}))\Phi
_{1}dx$ is bounded, thus $(fg(v_{n}))$ is bounded in $L^{1}(\Omega,\rho dx).$
Using the test function $\varphi=\mathcal{G}(1),$ it implies that $(v_{n})$ is
bounded in $L^{1}(\Omega)$ from the H\"{o}pf Lemma. Then $v^{\ast}\in
L^{1}(\Omega)$ and satisfies (\ref{vws}).\medskip
\end{proof}

When moreover $g$ is convex, it was proved in \cite{Ne} that $v^{\ast}$ is
more regular, in particular $g(v^{\ast})\in L^{1}(\Omega),$ by using stability
properties of $v^{\ast}.$ Thus $v^{\ast}$ is a renormalized solution of
(PV$\lambda^{\ast}).$ In case $p\neq2$ there is no notion of very weak solution.

\subsubsection{Without convexity}

Without convexity we obtain a local result:

\begin{proposition}
\label{ploc}Assume (\ref{hypg}) and $f\in L^{r}\left(  \Omega\right)  ,r>N/p$.
Then $v^{\ast}$ is a local renormalized solution of (PV$\lambda_{b}).$ In
particular $T_{k}(v^{\ast})\in W_{loc}^{1,p}(\Omega)$ for any $k>0,$ $v^{\ast
}{}^{p-1}\in L_{loc}^{\sigma}(\Omega),$ for any $\sigma\in\left[
1,N/(N-p)\right)  ,$ and ($\left\vert \nabla v^{\ast}\right\vert )^{p-1}\in
L_{loc}^{\tau}(\Omega),$ for any $\tau\in\left[  1,N/(N-1)\right)  ,$ and
\[
-\Delta_{p}v^{\ast}=\lambda^{\ast}f(1+g(v^{\ast}))^{p-1}\qquad\text{in
}\mathcal{D}^{\prime}(\Omega).
\]

\end{proposition}

For the proof we use the following Lemma:

\begin{lemma}
\label{estif}Assume $f\in L^{1}(\Omega),$ and $g$ satisfies (\ref{hypg}). Let
$\left(  \lambda_{n}\right)  $ be a sequence of reals such that \underline
{$\lim$}$\lambda_{n}>0,$ and $(v_{n})$ be a sequence of renormalized solutions
of problem (PV$\lambda_{n}$). Then ($fg(v_{n})^{p-1})$ is bounded in
$L_{loc}^{1}(\Omega),$ and $(v_{n}^{p-1})$ is bounded in $L_{loc}^{\sigma
}(\Omega),$ for any $\sigma\in\left[  1,N/(N-p)\right)  .$
\end{lemma}

\begin{proof}
[Proof of Lemma \ref{estif}]From Lemma \ref{secm}, for any $x_{0}$ such that
$B(x_{0},4\rho)\subset\Omega,$ there exists a constant $C=C(N,p)$ such that
\[
\lambda_{n}\int_{B(x_{0},\rho)}f(1+g(v_{n}))^{p-1}dx\leqq C\rho^{N-p}%
\min_{B(x_{0},\rho)}v_{n}^{p-1}\leqq\frac{C\rho^{N-p}}{\int_{B(x_{0},\rho
)}fdx}\int_{B(x_{0},\rho)}fv_{n}{}^{p-1}dx.
\]
Then there exist $c=c(N,p,\rho,f,x_{0},\underline{\lim}\lambda_{n})>0$ such
that%
\[
\int_{B(x_{0},\rho)}fg(v_{n})^{p-1}dx\leqq c\int_{B(x_{0},\rho)}fv_{n}{}%
^{p-1}dx.
\]
From (\ref{hypg}), there exists $A>0$ such that $g(t)\geqq(2c)^{1/(p-1)}t$ for
any $t\geqq A,$ thus
\[
\int_{B(x_{0},\rho)}fg(v_{n})^{p-1}dx\leqq c\int_{B(x_{0},\rho)}fv_{n}{}%
^{p-1}dx\leqq2A^{p-1}c\int_{B(x_{0},\rho)}fdx\leqq2A^{p-1}c\left\Vert
f\right\Vert _{L^{1}(\Omega)},
\]
and the claim is proved. Moreover we deduce that
\[
\min_{B(x_{0},\rho)}v_{n}^{p-1}\leqq c^{\prime}=c^{\prime}(N,p,\rho
,f,g,x_{0});
\]
from the weak Harnack inequality, $(v_{n}^{p-1})$ is bounded in $L_{loc}%
^{\sigma}(\Omega),$ for any $\sigma\in\left[  1,N/(N-p)\right)  .\medskip$
\end{proof}

\begin{proof}
[Proof of Proposition \ref{ploc}]Let $\lambda_{n}\nearrow\lambda_{b},$ and
$v_{n}=\underline{v}_{\lambda_{n}}.$ From Lemma \ref{estif}, ($fg(v_{n}%
)^{p-1})$ is bounded in $L_{loc}^{1}(\Omega),$ and $(v_{n}^{p-1})$ is bounded
in $L_{loc}^{\sigma}(\Omega),$ for any $\sigma\in\left[  1,N/(N-p)\right)  .$
Then from \cite[Theorem 3.2]{B-V}, there exists a subsequence converging a.e.
in $\Omega$. And ($v_{n})$ is nondecreasing thus the whole sequence converges
to $v^{\ast}$. And $g$ is nondecreasing, thus $fg(v^{\ast})^{p-1}\in$
$L_{loc}^{1}(\Omega)$ from the Beppo-Levi Theorem, and ($fg(v_{n})^{p-1})$
converges to $fg(v^{\ast})^{p-1}$ weakly in $L_{loc}^{1}(\Omega);$ thus
($\lambda_{n}f(x)(1+g(v_{n}))^{p-1})$ converges to $\lambda^{\ast
}f(x)(1+g(v^{\ast}))^{p-1}$ weakly in $L_{loc}^{1}(\Omega).$ From
\cite[Theorem 3.3]{B-V}, $v^{\ast}$ is a local renormalized solution of
(PV$\lambda_{b}).\medskip$
\end{proof}

Our next results use the Euler function linked to the problem. From the
Maximum Principle, problem (PV$\lambda)$ is equivalent to
\begin{equation}
-\Delta_{p}v=\lambda f(x)\varphi(v)=\lambda f(x)(1+g(v^{+}))^{p-1}%
\hspace{0.5cm}\text{in }\Omega,\qquad v=0\hspace{0.5cm}\text{on }%
\partial\Omega. \label{PQ}%
\end{equation}
where
\begin{equation}
\varphi(t)=(1+g(t^{+}))^{p-1};\qquad\Phi(t)=\int_{0}^{t}\varphi(s)ds=\int
_{0}^{t}(1+g(s^{+}))^{p-1}ds, \label{fifi}%
\end{equation}
thus $\Phi\in C^{1}(\mathbb{R})$. For any $f\in L^{1}(\Omega)$ and any $v\in
W_{0}^{1,p}(\Omega)$ such that $f\Phi(v)\in L^{1}(\Omega),$ we set
\begin{equation}
J_{\lambda}(v)=\frac{1}{p}\int_{\Omega}\left\vert \nabla v\right\vert
^{p}dx-\lambda\int_{\Omega}f\Phi(v)dx. \label{defi}%
\end{equation}
In particular the function $J_{\lambda}$ is defined on $W_{0}^{1,p}%
(\Omega)\cap L^{\infty}(\Omega).$ Let us recall some important properties of
$J_{\lambda}.$

\begin{proposition}
[\cite{CaSa}]\label{cas}Assume $f\in L^{1}(\Omega)$ and (\ref{hypo}). Let
$\lambda>0$ such that there exists a supersolution $\bar{v}\in W_{0}%
^{1,p}(\Omega)$ of (PV$\lambda).$ Then $J_{\lambda}$ is defined on
$\mathcal{K}_{\bar{v}}=\left\{  v\in W_{0}^{1,p}(\Omega):0\leqq v\leqq\bar
{v}\right\}  $ and attains its minimum on $\mathcal{K}_{\bar{v}}$ at some
point $v$ which is a solution of (PV$\lambda).$ In particular if
$0<\lambda<\lambda_{b},$ then
\[
J_{\lambda}(\underline{v}_{\lambda})=\min_{\mathcal{K}_{\underline{v}%
_{\lambda}}}J_{\lambda}(v)\leqq0.
\]

\end{proposition}

\begin{remark}
\label{post}In fact $J_{\lambda}(\underline{v}_{\lambda})<0.$ Indeed if
$J_{\lambda}(\underline{v}_{\lambda})=0,$ then for any $t\in\left(
0,1\right)  ,$ $J_{\lambda}(t\underline{v}_{\lambda})\geqq0,$ thus
\[
t^{p}\int_{\Omega}\left\vert \nabla\underline{v}_{\lambda}\right\vert
^{p}dx\geqq p\lambda\int_{\Omega}f\Phi(t\underline{v}_{\lambda})dx\geqq
p\lambda t\int_{\Omega}f\underline{\underline{v}_{\lambda}}dx
\]
thus $f\underline{v}_{\lambda}=0,$ and $f>0,$ thus $\underline{v}_{\lambda
}=0,$ which is contradictory.
\end{remark}

Next we give a global result under the well-known Ambrosetti-Rabinowitz
condition on $g:$

\begin{proposition}
\label{extr}Assume (\ref{hypg}), $f\in L^{r}\left(  \Omega\right)  ,r>N/p$
and
\begin{equation}
\underline{\lim}_{t\rightarrow\infty}t\varphi(t)/\Phi(t)=k>p. \label{AR}%
\end{equation}
Then $v^{\ast}\in W_{0}^{1,p}(\Omega)$ and is a variational solution of
(PV$\lambda_{b}$)$.$
\end{proposition}

\begin{proof}
Let $\lambda_{n}\nearrow\lambda_{b},$ and $v_{n}=\underline{v}_{\lambda_{n}}.$
From Proposition \ref{cas},
\[
J_{\lambda_{n}}(v_{n})=\frac{1}{p}\int_{\Omega}\left\vert \nabla
v_{n}\right\vert ^{p}dx-\lambda_{n}\int_{\Omega}f\Phi(v_{n})dx\leqq0
\]
and
\begin{equation}
\int_{\Omega}\left\vert \nabla v_{n}\right\vert ^{p}dx=\lambda_{n}\int
_{\Omega}f(1+g(v_{n}))^{p-1}v_{n}dx; \label{ven}%
\end{equation}
then there exists $B>0$ and $C>0$ such that
\[
0\geqq pJ_{\lambda_{n}}(v_{n})=\lambda_{n}\int_{\Omega}f(v_{n}\varphi
(v_{n})-p\Phi(v_{n}))dx\geqq\frac{1}{2}\lambda_{n}(k-p)\int_{\left\{
v_{n}\geqq B\right\}  }f\Phi(v_{n}))dx-C\lambda_{n}%
\]
thus $f\Phi(v_{n})$ is bounded in $L^{1}(\Omega),$ and also $\int_{\Omega
}\left\vert \nabla v_{n}\right\vert ^{p}dx$ is bounded; then there exists a
subsequence converging weakly in $W_{0}^{1,p}(\Omega),$ and necessarily to
$v^{\ast}.$ From Proposition \ref{ploc}, $v^{\ast}$ is a solution of
(PV$\lambda_{b}$) in $\mathcal{D}^{\prime}(\Omega),$ thus in the variational sense.
\end{proof}

\begin{remark}
Proposition applies in particular when $\underline{\lim}_{t\rightarrow\infty
}tg^{\prime}(t)/g(t)=m>1.$ It follows from the L'Hospital rule, since
$(t\varphi(t))^{\prime}/\Phi^{\prime}(t)=1+(p-1)tg^{\prime}(t)/(1+g(t)))$ for
any $t>0.$ This improves the result of \cite{CaSa}, where moreover it is
supposed that $g(t)\leqq C(1+t^{m}),$ and extends also the one of \cite{BrVa}.
\end{remark}

\subsubsection{With convexity}

Here we assume that $g$ satisfies is superlinear and convex near $\infty.$
Recall that $\lambda_{b}=\lambda^{\ast}=\lambda_{r}<\infty$ from Theorem
\ref{impo} and Remark \ref{cinq}. We first define some functions linked to $g$
and give their asymptotic properties.

\begin{lemma}
\label{fun}Assume (\ref{hypg}) with $g$ convex near $\infty.$Let for any
$t\geqq0$%
\begin{equation}
j(t)=tg^{\prime}(t)-g(t),\qquad\mathcal{J}(t)=t\varphi(t)-p\Phi(t), \label{Fj}%
\end{equation}
\begin{equation}
h(t)=\int_{0}^{t}g^{\prime}(s)(g^{\prime}(t)-g^{\prime}(s))ds=g(t)g^{\prime
}(t)-\int_{0}^{t}g^{\prime2}(s)ds. \label{Fh}%
\end{equation}
Then $\lim_{t\rightarrow\infty}j(t)/g^{\prime}(t)=\infty,$ $\lim
_{t\rightarrow\infty}\mathcal{J}(t)/\varphi(t)=\infty.$ and $\lim
_{t\rightarrow\infty}h(t)/j(t)=\infty.$
\end{lemma}

\begin{proof}
(i) The function $j$ is nondecreasing near $\infty$, since $g$ is convex near
$\infty.$ Thus $j$ has a limit $L$ in $\left(  -\infty,\infty\right]  .$ Let
us show that $L=\infty;$ indeed if $L$ is finite, then $tg^{\prime}(t)\leqq
g(t)+\left\vert L\right\vert +1$ for large $t,$ thus ($g(t)+\left\vert
L\right\vert +1)/t$ is nonincreasing, which contradicts (\ref{hypg}). First
assume that $g\in C^{2}\left(  \left(  0,\infty\right)  \right)  $ and
$g"(t)>0$: from the l'Hospital rule,
\[
\lim_{t\rightarrow\infty}j(t)/g^{\prime}(t)=\lim_{t\rightarrow\infty}%
j^{\prime}(t)/g^{\prime\prime}(t)=\lim_{t\rightarrow\infty}t=\infty.
\]
In the general case $g$ is convex for $t\geqq A\geqq0,$ and $\lim
_{t\rightarrow\infty}g^{\prime}(t)=\infty;$ thus for any $K>0,$ there exists
$t_{K}$ $>A+2K$ such that $g^{\prime}(t)\geqq2g^{\prime}(A+2K)$ for $t\geqq
t_{K}.$ Then for $t\geqq t_{K},$%
\begin{align*}
j(t)  &  =\int_{0}^{t}(g^{\prime}(t)-g^{\prime}(s))ds=\int_{0}^{A}(g^{\prime
}(t)-g^{\prime}(s))ds+\int_{A}^{t}(g^{\prime}(t)-g^{\prime}(s))ds\\
&  \geqq-g(A)+\int_{A}^{A+2K}(g^{\prime}(t)-g^{\prime}(s))ds\geqq
-g(a)+Kg^{\prime}(t),
\end{align*}
thus $\lim_{t\rightarrow\infty}j(t)/g^{\prime}(t)=\infty.$ And
\begin{equation}
\mathcal{J}^{\prime}(t)=(p-1)(1+g(t))^{p-2}(j(t)-1)=\varphi^{\prime
}(t)(j(t)-1)/g^{\prime}(t) \label{jip}%
\end{equation}
thus $\lim_{t\rightarrow\infty}\mathcal{J}^{\prime}(t)/\varphi^{\prime
}(t)=\infty;$ and $\lim_{t\longrightarrow\infty}\varphi(t)=\infty,$ thus
$\lim_{t\rightarrow\infty}\mathcal{J}(t)/\varphi(t)=\infty.\medskip$

(ii) First assume that $g\in C^{2}\left(  \left(  0,\infty\right)  \right)  $
and $g^{\prime\prime}(t)>0.$ Then $h(t)=\int_{0}^{t}g(s)g^{\prime\prime
}(s)ds,$ and from the l'Hospital rule,
\[
\lim_{t\rightarrow\infty}h(t)/j(t)=\lim_{t\rightarrow\infty}h^{\prime
}(t)/tg^{\prime\prime}(t)=\lim_{t\rightarrow\infty}g(t)/t=\infty.
\]
In the general case, for any $C>g^{\prime}(A),$ there exists $A_{1}>A>0$ such
that $g^{\prime}(s)\geqq2C,$ for $s\geqq A_{1}$ and $g^{\prime}(s)\leqq2C$ for
$s\leqq A_{1}$ and there exists $B>5A_{1}$ such that $g^{\prime}%
(t)\geqq2g^{\prime}(5A_{1})$ for $t\geqq B.$ Then denoting $C_{A}=h(A)-Cj(A),$
for $t\geqq B,$
\begin{align*}
h(t)-Cj(t)  &  =C_{A}+\int_{A}^{t}(g^{\prime}(s)-C)(g^{\prime}(t)-g^{\prime
}(s))ds\\
&  \geqq-\left\vert C_{A}\right\vert -CA_{1}(g^{\prime}(t)+2C)+\int_{A_{1}%
}^{5A_{1}}(g^{\prime}(s)-C)(g^{\prime}(t)-g^{\prime}(s))ds\\
&  \geqq-\left\vert C_{A}\right\vert -CA_{1}(g^{\prime}(t)+C)+2A_{1}%
Cg^{\prime}(t)=-\left\vert C_{A}\right\vert +CA_{1}(g^{\prime}(t)-C)
\end{align*}
thus $\lim_{t\rightarrow\infty}h(t)/j(t)=\infty.\medskip$
\end{proof}

The following result will be used also in next Paragraph. The proof is new,
using only the function $\mathcal{J}.$ Notice that the proof given in
\cite{AAP} for $p=2$ was not extendable.

\begin{proposition}
\label{for}Assume (\ref{hypg}) with $g$ convex near $\infty,$ and $f\in
L^{1}\left(  \Omega\right)  $. Let $\left(  \lambda_{n}\right)  $ be a
sequence of positive reals such that \underline{$\lim$}$\lambda_{n}>0,$ and
$(v_{n})$ be a sequence of solutions of (PV$\lambda_{n})$, such that $v_{n}\in
W_{0}^{1,p}(\Omega),$ $f\Phi(v_{n})\in L^{1}\left(  \Omega\right)  $, and
$J_{\lambda_{n}}(v_{n})\leqq c\in\mathbb{R}$.\medskip

Then $(\Delta_{p}v_{n})$ is bounded in $L^{1}(\Omega).$
\end{proposition}

\begin{proof}
The function $v_{n}\in W_{0}^{1,p}(\Omega)$ still satisfies (\ref{ven}), thus
\begin{equation}
pJ_{\lambda_{n}}(v_{n})=\lambda_{n}\int_{\Omega}f(v_{n}\varphi(v_{n}%
)-p\Phi(v_{n}))dx=\lambda_{n}\int_{\Omega}f\mathcal{J}(v_{n})dx\leqq cp,
\label{fri}%
\end{equation}
where $\mathcal{J}$ is defined at (\ref{Fj}). Then from Lemma \ref{fun},
$\int_{\Omega}f\mathcal{\varphi}(v_{n})dx$ is bounded, which means that
($\Delta_{p}v_{n})$ is bounded in $L^{1}(\Omega).\medskip$
\end{proof}

As a consequence, we prove that the extremal solution is a solution of
(PV$\lambda^{\ast}$) in a very simple way:

\begin{corollary}
\label{much}Assume (\ref{hypg}) with $g$ convex near $\infty,$ and $f\in
L^{r}\left(  \Omega\right)  $ with $r>N/p.$ Then the extremal solution
$v^{\ast}$ is a \textbf{renormalized} solution of (PV$\lambda^{\ast}$).
\end{corollary}

\begin{proof}
Let $\lambda_{n}\nearrow\lambda^{\ast},$ and $v_{n}=\underline{v}_{\lambda
_{n}}.$ Then $J_{\lambda_{n}}(v_{n})\leqq0$ from Proposition \ref{cas}. From
Proposition \ref{for}, ($fg(v_{n})^{p-1})$ is bounded in $L^{1}(\Omega),$ and
$(v_{n}^{p-1})$ is bounded in $L^{\sigma}(\Omega),$ for any $\sigma\in\left[
1,N/(N-p)\right)  .$ Then from \cite[Theorem 3.2]{B-V}, converges to $v^{\ast
}$ a.e. in $\Omega$, as in Proposition \ref{ploc}. From the Beppo-Levi
theorem, $fg(v^{\ast})^{p-1}\in$ $L^{1}(\Omega),$ and ($fg(v_{n})^{p-1}$)
converges to $fg(v)^{p-1}$ weakly in $L^{1}(\Omega);$ thus ($\lambda
_{n}f(x)(1+g(v_{n}))^{p-1})$ converges to $\lambda^{\ast}f(x)(1+g(v))^{p-1}$
weakly in $L^{1}(\Omega).$ From Remark \ref{conv}, $v$ is a renormalized
solution of (PV$\lambda^{\ast}$).\medskip
\end{proof}

Next we find again this result and get further informations on $v^{\ast}$ by
using stability properties of the minimal bounded solutions. This extend the
results of \cite{Ne} for $p=2$ and of \cite{Sa2} for $p>2$ with $f\equiv1.$
Here we use the function $h$ defined at (\ref{Fh}), introduced by \cite{Ne}.
We first extend the definition given in \cite{CaSa} for functions $v\in
W_{0}^{1,p}(\Omega):$

\begin{definition}
A renormalized solution $v$ of problem (PV$\lambda$) is called
\textbf{semi-stable} if the "second derivative of $J_{\lambda}$ is
nonnegative", in the sense
\begin{equation}
\int_{\left\{  \nabla v\neq0\right\}  }\left\vert \nabla v\right\vert
^{p-2}((p-2)(\frac{\nabla v.\nabla\psi}{\left\vert \nabla v\right\vert }%
)^{2}+\left\vert \nabla\psi\right\vert ^{2})dx\geqq(p-1)\lambda\int_{\Omega
}f(1+g(v))^{p-2}g^{\prime}(v)\psi^{2}dx, \label{sed}%
\end{equation}
for any $\psi\in\mathcal{D}(\Omega)$ if $p\geqq2;$ for any $\psi\in
\mathcal{D}(\Omega)$ such that $\psi\leqq Cv$ and $\left\vert \nabla
\psi\right\vert \leqq C\left\vert \nabla v\right\vert $ in $\Omega$ for some
$C>0$ if $p<2.$
\end{definition}

The integral on the left-hand side is well defined. Indeed either $p>2$ and
$\left\vert \nabla v\right\vert ^{p-1}\in L^{1}(\Omega),$ or $p<2$ and
\[
\int_{\left\{  \nabla v\neq0\right\}  }\left\vert \nabla v\right\vert
^{p-2}\left\vert \nabla\psi\right\vert ^{2}dx\leqq C\int_{\left\{  \left\vert
\nabla v\right\vert >1\right\}  }\left\vert \nabla v\right\vert ^{p-1}%
\left\vert \nabla\psi\right\vert dx+\int_{\left\{  0<\left\vert \nabla
v\right\vert \leqq1\right\}  }\left\vert \nabla\psi\right\vert ^{2}dx
\]
When $v\in W_{0}^{1,p}(\Omega)$, (\ref{sed}) is valid for any $\psi\in
W_{0}^{1,p}(\Omega),$ satisfying the conditions above when $p<2.$

\begin{proposition}
\label{Regned}Assume (\ref{hypg}) with $g$ convex near $\infty,$ and $f\in
L^{r}\left(  \Omega\right)  ,r>N/p.$ Let $h$ be defined at (\ref{Fh}). (i)
Then
\begin{equation}
f(1+g(v^{\ast}))^{p-1}h(v^{\ast})\in L^{1}(\Omega). \label{clo}%
\end{equation}

\noindent(ii) If $N<N_{0}=pp^{\prime}/(1+1/(p-1)r),$ then $v^{\ast}\in
L^{\infty}(\Omega).$

If $N>N_{0},$ then $v^{\ast}{}^{p-1}\in L^{k}(\Omega)$ for any $k<\bar{\sigma
},$ where $1/\bar{\sigma}=1-pp^{\prime}/N+1/r(p-1).$

If $N=N_{0},$ then $v^{\ast}\in L^{k}(\Omega)$ for any $k\geqq1.$\medskip

\noindent(iii) If $N<N_{1}=p(1+p^{\prime})/(1+p^{\prime}/r)$ then $v^{\ast}$
$\in W_{0}^{1,p}(\Omega).$

If $N>N_{1},$ $\left\vert \nabla v^{\ast}\right\vert ^{p-1}\in L^{\tau}%
(\Omega)$ for any $\tau<\bar{\tau}$ where $1/\bar{\tau}=1+1/(p-1)r-(p^{\prime
}+1)/N.$

If $N=N_{1},$ $\left\vert \nabla v^{\ast}\right\vert \in L^{s}(\Omega)$ for
any $s<p.$\medskip

\noindent(iv) If $\underline{\lim}_{t\rightarrow\infty}h(t)/t>0,$ then
$v^{\ast}$ $\in W_{0}^{1,p}(\Omega).$ It holds in particular if $\underline
{\lim}_{t\rightarrow\infty}(g^{\prime}(t)-g(t)/t)>0.$
\end{proposition}

\begin{proof}
(i) Let $\lambda_{n}\nearrow\lambda^{\ast},$ and $v_{n}=\underline{v}%
_{\lambda_{n}}.$ By hypothesis $g$ is convex for $t\geqq A.$ From
\cite[Proposition 2.2]{CaSa}, $v_{n}$ is semi-stable. Taking $\psi=g(v_{n})$
in (\ref{sed}) with $\lambda=\lambda_{n}$ and $v=v_{n}$, we get
\[
\int_{\Omega}\left\vert \nabla v_{n}\right\vert ^{p}g^{\prime2}(v_{n}%
)dx\geqq\lambda_{n}\int_{\Omega}f(1+g(v_{n}))^{p-2}g^{\prime}(v_{n}%
)g^{2}(v_{n})dx.
\]
Taking $S(v_{n}),$as a test function in (PV$\lambda_{n})$, where
$S(t)=\int_{0}^{t}g^{\prime2}(s)ds,$ we find
\[
\int_{\Omega}\left\vert \nabla v_{n}\right\vert ^{p}g^{\prime2}(v_{n}%
)dx=\lambda_{n}\int_{\Omega}f(1+g(v_{n}))^{p-1}S(v_{n})dx
\]
By difference we obtain
\begin{align*}
&  \int_{\Omega}f(1+g(v_{n}))^{p-2}((1+g(v_{n}))S(v_{n})-g^{\prime}%
(v_{n})g^{2}(v_{n}))dx\\
&  =\int_{\Omega}f(1+g(v_{n}))^{p-2}(S(v_{n})-g(v_{n})h(v_{n}))dx\geqq0.
\end{align*}
Observing that $S(t)\leqq g(t)g^{\prime}(t)+\left\vert h(A)\right\vert $ for
$t\geqq A,$ and $\lim_{t\rightarrow\infty}h(t)/g^{\prime}(t)=\infty,$ from
Lemma \ref{fun}, there exists $C>0$ such that
\[
\int_{\Omega}f(1+g(v_{n}))^{p-2}g(v_{n})h(v_{n}))dx\leqq C.
\]
And lim$_{t\rightarrow\infty}g(t)=\infty,$ thus $1+g(t)\leqq2g(t)$ for $t\geqq
A,$ hence ($f(1+g(v_{n}))^{p-1}h(v_{n}))$ is bounded in $L^{1}(\Omega),$ thus
(\ref{clo}) holds. Then $fg(v^{\ast})^{p-1}j(v^{\ast})\in L^{1}(\Omega)$ from
Lemma \ref{fun}, hence $fg(v^{\ast})^{p-1}g^{\prime}(v^{\ast})\in L^{1}%
(\Omega)$ and $fg(v^{\ast})^{p}/v^{\ast}\in L^{1}(\Omega)$. In particular we
find again that ($f(1+g(v^{\ast}))^{p-1)}\in L^{1}(\Omega),$ which was
obtained in a shorter way at Theorem \ref{much}.\medskip

\noindent(ii)The regularity of $v^{\ast}$ follows from the estimate
$f(g(v^{\ast}))^{p}/v^{\ast}\in L^{1}(\Omega):$ Taking $r^{\prime}%
<\sigma<N/(N-p),$ we have $v^{\ast}{}^{p-1}\in L^{\sigma}(\Omega).$ Defining
$\theta$ by $p/\theta=p-1+1/r+1/\sigma,$ we have $\theta\in\left(
1,p^{\prime}\right)  ,$ and from H\"{o}lder inequality,
\begin{align*}
\int_{\Omega}(fg(v^{\ast})^{p-1})^{\theta}dx  &  =\int_{\Omega}(\frac
{f^{1/p}g(v^{\ast})}{v^{\ast1/p}})^{(p-1)\theta}(f^{\theta/p}v^{\ast}%
{}^{\theta/p^{\prime}})dx\\
&  \leqq\left(  \int_{\Omega}\frac{fg(v^{\ast})^{p}}{v^{\ast}}dx\right)
^{\theta/p^{\prime}}\left(  \int_{\Omega}(f^{\theta/p}v^{\ast}{}%
^{\theta/p^{\prime}})^{p^{\prime}/(p^{\prime}-\theta)}dx\right)
^{1-\theta/p^{\prime}}\\
&  \leqq\left(  \int_{\Omega}\frac{fg(v^{\ast})^{p}}{v^{\ast}}dx\right)
^{\theta/p^{\prime}}\left(  \int_{\Omega}f^{r}dx\right)  ^{\theta/p}\left(
\int_{\Omega}v^{\ast}{}^{\sigma(p-1)}dx\right)  ^{\theta/p\sigma}%
\end{align*}
Then $fg(v^{\ast})^{p-1}\in$ $L^{\theta}(\Omega)$ with $\theta>1$. If $p=N,$
then from Lemma \ref{boot}, $v^{\ast}\in L^{\infty}(\Omega).$ Next assume
$p<N.$ Choosing $\sigma$ sufficiently close to $r^{\prime},$ one has
$\theta<N/p.$ From Lemma \ref{boot}, as soon as $\theta<N/p,$ we find
$v^{\ast}{}^{p-1}\in L^{\sigma_{1}}(\Omega)$ with $\sigma_{1}=N\theta
/(N-p\theta).$ For $\sigma$ sufficiently close to $r^{\prime},$ we also find
$\sigma_{1}>\sigma.$ Then we can define an increasing sequence ($\sigma_{\nu
})$ and a sequence ($\theta_{\nu}),$ as long as $\theta_{\nu}<N/p.$ If
($\sigma_{\nu})$ has a limit $\bar{\sigma},$ then $1/\bar{\sigma}%
=1-pp^{\prime}/N+1/r(p-1),$ and ($\theta_{\nu})$ converges to $\bar{\theta
}=(1+1/r(p-1)-p^{\prime}/N)^{-1}.$ It follows that $v^{\ast}\in L^{\infty
}(\Omega)$ if $N<N_{0}.$ If $N\geqq N_{0},$ $v^{\ast}{}^{p-1}\in L^{k}%
(\Omega)$ for any $k<\bar{\sigma}.\medskip$

\noindent(iii) Lemma \ref{boot} also gives estimates of the gradient: if
$p=N,$ then $v^{\ast}\in W_{0}^{1,N}(\Omega).$ If $p<N,$ ($\left\vert \nabla
v^{\ast}\right\vert ^{p-1})\in L^{\tau_{\nu}}(\Omega)$ with $1/\tau_{\nu
}=1/\theta_{\nu}-1/N,$ as long as $\theta_{\nu}<Np/(Np-N+p),$ and ($\tau_{\nu
})$ converges to $\bar{\tau}=(1+1/(p-1)r-(p^{\prime}+1)/N)^{-1}.$ Then
$v^{\ast}\in W_{0}^{1,p}(\Omega)$ if $\bar{\tau}>p^{\prime},$ that means if
$N<N_{1}$. If $N\geqq N_{1},$ ($\left\vert \nabla v^{\ast}\right\vert
^{p-1})\in L^{\tau}(\Omega)$ for any $\tau<\bar{\tau}.\medskip$

(iv) If $\underline{\lim}_{t\rightarrow\infty}h(t)/t>0,$ then
\[
\int_{\Omega}\left\vert \nabla v_{n}\right\vert ^{p}dx=\lambda_{n}\int
_{\Omega}f(1+g(v_{n}))^{p-1}v_{n}dx\leqq C,
\]
thus $v^{\ast}\in W_{0}^{1,p}(\Omega).$ It holds in particular when
$\underline{\lim}_{t\rightarrow\infty}j(t)/t>0,$ from Lemma \ref{fun}.
\end{proof}

\begin{remark}
if $p\geqq2,$ $v^{\ast}$ is semi-stable. Indeed $v_{n}=\underline{v}%
_{\lambda_{n}}$ satisfies (\ref{sed}) for any $\psi\in\mathcal{D}(\Omega)$.
And ($\left\vert \nabla v_{n}\right\vert )$ converges strongly in
$L^{1}(\Omega)$ to $\left\vert \nabla v^{\ast}\right\vert ^{p-1},$ so that we
can go to the limit from Lebesgue Theorem and Fatou Lemma.
\end{remark}

\begin{remark}
In case $p=2,$ $\Omega$ strictly convex, and $f=1,$ then $v^{\ast}\in
W_{0}^{1,2}(\Omega),$ for \textbf{any} function $g$ satisfying (\ref{hypg}),
from \cite{Ne1} . The proof uses the fact that $J_{\lambda^{\ast}}(v^{\ast
})\leqq0$ and Pohozaev identity; the kea point is that $v^{\ast}$ is regular
near the boundary, from results of \cite{Ram}. In the general case $p>1$ with
$f\equiv1$, if we can prove that $v^{\ast}$ is regular near the boundary, then
$v^{\ast}\in W_{0}^{1,p}(\Omega).$ Indeed Pohozaev identity extends to the
$p$-Laplacian, see \cite{GuVe}. For general $f$ we cannot get the result by
this way, even for $p=2$.
\end{remark}

\begin{remark}
In the exponential case $1+g(v)=e^{v},$ with $f\equiv1,$ it has been proved
that $v^{\ast}\in W_{0}^{1,p}(\Omega),$ and $v^{\ast}\in L^{\infty}(\Omega)$
whenever $N<N_{2}=4p/(p-1)+p,$ see \cite{AzPe1} and \cite{AzPePu}. In the
power case, $(1+g(v))^{p-1}=(1+v)^{m}$ the same happens$;$ if $N\geqq N_{2},$
and $m<m_{c},$ where
\[
m_{c}=\frac{(p-1)N-2\sqrt{(p-1)(N-1)}+2-p}{N-p-2-2\sqrt{(N-1)/(p-1)}}%
\]
then also $v^{\ast}\in L^{\infty}(\Omega),$ see \cite{Fe}. The same
conclusions hold when the function $g$ behaves like an exponential or a power,
see \cite{YeZh}, \cite{CaSa}, \cite{Sa}, and \cite{EE}. Up to our knowledge,
the gap between $N_{0}=pp^{\prime}$ and $N_{2}$ remains for general $g,$
excepted in the radial case, see \cite{CaCapSa}.
\end{remark}

We end this paragraph with a boundness property when $g$ has a slow growth:

\begin{proposition}
\label{vlan}Assume that $g$ satisfies (\ref{hypg})and (\ref{hmq}) for some
$Q\in\left(  p-1,Q_{1}\right)  ,$ and $g$ is convex near $\infty$, and $f\in
L^{r}(\Omega)$ with $Qr^{\prime}<Q_{1}.\medskip$

Then $v^{\ast}\in W_{0}^{1,p}(\Omega)\cap L^{\infty}(\Omega)$ and is a
variational solution of (PV$\lambda^{\ast}$)$.$
\end{proposition}

\begin{proof}
As in Proposition \ref{san}, it follows from Corollary \ref{much} and
Proposition \ref{cig} (i).\medskip
\end{proof}

\subsection{Boundedness and multiplicity under Sobolev conditions\label{sob}}

\noindent Next we assume only that $g$ is subcritical with the Sobolev
exponent:
\begin{equation}
\overline{\lim}_{\tau\longrightarrow\infty}\frac{g(\tau)^{p-1}}{\tau^{Q}%
}<\infty,\text{ for some }Q\in\left(  p-1,Q^{\ast}\right)  , \label{sousg}%
\end{equation}
and $f\in L^{r}(\Omega)$ with $(Q+1)r^{\prime}<p^{\ast}$. Then $J_{\lambda}$
is well defined on $W_{0}^{1,p}(\Omega)$ and $J_{\lambda}$ $\in C^{1}%
(W_{0}^{1,p}(\Omega))$.\medskip

\begin{proposition}
\label{essen}Assume (\ref{hypg}) and (\ref{sousg}), $g$ convex near $\infty,$
and $f\in L^{r}(\Omega)$ with $(Q+1)r^{\prime}<p^{\ast}$. Let $\left(
\lambda_{n}\right)  $ be a sequence of positive reals such that $\lim
\lambda_{n}=\lambda>0,$ and $(v_{n})$ be a sequence of solutions of
(PV$\lambda_{n})$ such that $v_{n}\in W_{0}^{1,p}(\Omega),$ and $J_{\lambda
_{n}}(v_{n})\leqq c\in\mathbb{R}$.\medskip

Then $\left(  v_{n}\right)  $ is bounded in $W_{0}^{1,p}(\Omega).$
\end{proposition}

\begin{proof}
We still have
\[
pJ_{\lambda_{n}}(v_{n})=\lambda_{n}\int_{\Omega}f(v_{n}\varphi(v_{n}%
)-p\Phi(v_{n}))dx=\lambda_{n}\int_{\Omega}f\mathcal{J}(v_{n})dx
\]
where $\mathcal{J}$ is defined at (\ref{Fj}). From Proposition \ref{for},
($fg(v_{n})^{p-1})$ is bounded in $L^{1}(\Omega).$ Following the method of
(\cite{JeTo}), suppose that up to a subsequence, $\lim\left\Vert
v_{n}\right\Vert _{W_{0}^{1,p}(\Omega)}=\infty,$ and consider $w_{n}%
=v_{n}/\left\Vert v_{n}\right\Vert _{W_{0}^{1,p}(\Omega)}.$ Up to a
subsequence, $\left(  w_{n}\right)  $ converges to a function $w$ weakly in
$W_{0}^{1,p}(\Omega)$ and strongly in $L^{k+1}(\Omega),$ for any $k<Q^{\ast}.$
For any $\zeta\in\mathcal{D}(\Omega),$%
\[
\int_{\Omega}\left\vert \nabla w_{n}\right\vert ^{p-2}\nabla w_{n}\nabla\zeta
dx=\left\Vert v_{n}\right\Vert _{W_{0}^{1,p}(\Omega)}^{1-p}\int_{\Omega
}(1+g(v_{n}))^{p-1}\zeta dx
\]
tends to $0,$ thus $w=0.$ Let $z_{n}=t_{n}v_{n},$ where
\[
t_{n}=\inf\left\{  t\in\left[  0,1\right]  :J_{\lambda_{n}}(tv_{n})=\max
_{s\in\left[  0,1\right]  }J_{\lambda_{n}}(sv_{n})\right\}  .
\]
In fact lim $J_{\lambda_{n}}(z_{n})=\infty.$ Indeed suppose that
$\underline{\lim}J_{\lambda_{n}}(z_{n})=M<\infty.$ For given $K>0$, setting
$u_{n}=Kw_{n},$ then up to a subsequence, $J_{\lambda_{n}}(u_{n})\leqq
J_{\lambda_{n}}(z_{n})\leqq M+1$ for large $n.$ And $\lim\int_{\Omega}%
f\Phi(u_{n})dx=0,$ from (\ref{sousg}) and the assumptions on $f$, hence $\lim
J_{\lambda_{n}}(u_{n})=\lambda K^{p}/p$ from (\ref{defi}). Taking $K$ large
enough leads to a contradiction. Then $t_{n}\in\left(  0,1\right)  $ for large
$n,$ thus%
\[
J_{\lambda_{n}}^{\prime}(z_{n})(z_{n})=\int_{\Omega}\left\vert \nabla
z_{n}\right\vert ^{p}dx-\lambda_{n}\int_{\Omega}f(1+g(z_{n}))^{p-1}z_{n}dx=0,
\]%
\[
\lambda_{n}^{-1}pJ_{\lambda_{n}}(z_{n})=\int_{\Omega}f(z_{n}\varphi
(z_{n})-p\Phi(z_{n}))dx=\int_{\Omega}f\mathcal{J}(z_{n})dx.
\]
And $\lim_{t\longrightarrow\infty}j(t)=\infty,$ from Lemma \ref{fun}. Thus
there exists $B>0$ such that $j(s)-1>0$ for $s\geqq B,$ hence $\mathcal{J}%
(B)\leqq\mathcal{J}(t)\leqq\mathcal{J}(\tau)$ for any $B\leqq t\leqq\tau$ from
(\ref{jip}). Moreover $z_{n}\leqq v_{n}$ a.e. in $\Omega,$ thus $\left\{
z_{n}>B\right\}  \subset\left\{  v_{n}>B\right\}  ,$ then with different
constants $C>0,$
\[
\int_{\Omega}f\mathcal{J}(z_{n})dx\leqq C+\int_{\left\{  z_{n}>B\right\}
}f\mathcal{J}(z_{n})dx\leqq C+\int_{\left\{  v_{n}>B\right\}  }f\mathcal{J}%
(v_{n})dx\leqq C+\int_{\Omega}f\mathcal{J}(v_{n})dx\leqq C+\lambda_{n}^{-1}pc
\]
therefore $\left(  J_{\lambda_{n}}(z_{n})\right)  $ is bounded, and we reach a
contradiction. Then $\left(  v_{n}\right)  $ is bounded in $W_{0}^{1,p}%
(\Omega).$
\end{proof}

As a consequence we obtain the boundedness of the extremal solution under
estimate \ref{sousg}, which achieves the proof of Theorem \ref{trema}:

\begin{proposition}
\label{flu} Assume (\ref{hypg}) and (\ref{sousg}), $g$ convex near $\infty,$
and $f\in L^{r}(\Omega)$ with $(Q+1)r^{\prime}<p^{\ast}$. Then the extremal
solution $v^{\ast}\in W_{0}^{1,p}(\Omega)\cap L^{\infty}(\Omega)$ and is a
variational solution of (PV$\lambda^{\ast}$)$.$
\end{proposition}

\begin{proof}
Considering $\lambda_{n}\nearrow\lambda^{\ast},$ the sequence of minimal
solutions $v_{n}=\underline{v}_{\lambda_{n}}$ satisfies $J_{\lambda_{n}}%
(v_{n})\leqq0$ from Proposition \ref{cas}. From Proposition \ref{essen},
$\left(  v_{n}\right)  $ is bounded in $W_{0}^{1,p}(\Omega),$ and converges to
$v^{\ast}$ a.e. in $\Omega,$ thus $v^{\ast}\in$ $W_{0}^{1,p}(\Omega)$ and is a
variational solution of (PV$\lambda^{\ast}$)$.$ Then $v\in L^{\infty}(\Omega)$
from Proposition \ref{cig} (iii). \medskip
\end{proof}

Next we show the multiplicity result of Theorem \ref{main}, where $f,g$
satisfy the assumptions of Proposition \ref{essen}. We still use the Euler
function $J_{\lambda}$ associated to (PV$\lambda).$ Here two difficulties
occur. For small $\lambda,$ $J_{\lambda}$ has the geometry of Mountain Path
near $0,$ but function $g$ can have a slow growth, and one cannot prove that
the Palais-Smale sequences are bounded in $W_{0}^{1,p}(\Omega);$ then we use a
result of \cite{Je} saying that there exist ($\lambda_{n})$ converging to
$\lambda,$ such that $J_{\lambda_{n}}$ has a critical point $v_{n}$, and we
prove that this sequence $(v_{n}$) is bounded. For larger $\lambda$ it is not
sure that $J_{\lambda}$ has the geometry of Mountain Path near the minimal
solution $\underline{v}_{\lambda}$ of (PV$\lambda),$ and we have to make
further assumptions on $g.$\medskip

\begin{proof}
[Proof of Theorem \ref{main}]For any $\lambda\in\left(  0,\lambda^{\ast
}\right)  $ there exists at least one solution, the minimal one $\underline
{v}_{\lambda},$ such that $J_{\lambda}(\underline{v}_{\lambda})<0,$ from
Proposition \ref{cas} and Remark \ref{post}.\medskip

\noindent(i) \textbf{Existence of a second solution for }$\lambda$\textbf{
small enough}.\medskip

From (\ref{hypg}) and (\ref{sousg}), there exists $\lambda_{0}\in\left(
0,\lambda^{\ast}\right)  $ such that for any $\lambda<\lambda_{0},$ there
exists $R_{\lambda}>0$ such that $\inf\left\{  J_{\lambda}(v):\left\Vert
v\right\Vert _{W_{0}^{1,p}(\Omega)}=R_{\lambda}\right\}  >0,$ and a function
$w_{\lambda}\in W_{0}^{1,p}(\Omega)$ with $\left\Vert w_{\lambda}\right\Vert
_{W_{0}^{1,p}(\Omega)}>R_{\lambda}$ and $J_{\lambda}(w_{\lambda})<0.$ Then
$J_{\lambda}$ has the geometry of the Mountain Path near $0$:
\begin{equation}
c_{\lambda}=\inf_{\theta\in\Gamma}\max_{t\in\left[  0,1\right]  }J_{\lambda
}(\theta(t))>0=\max(J_{\lambda}(0),J_{\lambda}(w_{\lambda})), \label{infsup}%
\end{equation}
where $\Gamma=\left\{  \theta\in C(\left[  0,1\right]  ,W_{0}^{1,p}%
(\Omega)):\theta(0)=0,\theta(1)=w_{\lambda}\right\}  .$ Let $\lambda_{1}%
\in\left(  0,\lambda_{0}\right)  $ be fixed. Let us show the existence of a
solution at the level $c_{\lambda_{1}}$. There exists $\delta>0$ such that the
family of functions $\left(  J_{\lambda}\right)  _{\alpha\in\left[
\lambda_{1}(1-\delta),\lambda_{1}(1+\delta)\right]  }$ also satisfy the
condition (\ref{infsup}):%
\begin{equation}
c_{\lambda}=\inf_{\theta\in\Gamma}\max_{t\in\left[  0,1\right]  }J_{\lambda
}(\theta(t))>0=\max(J_{\lambda}(0),J_{\lambda}(w_{\lambda_{1}})). \label{cond}%
\end{equation}
From \cite{Je}, for almost every $\lambda\in\left[  \lambda_{1}(1-\delta
),\lambda_{1}(1+\delta\right]  ,$ there exists a sequence $\left(
v_{\lambda,m}\right)  ,$ bounded in $W_{0}^{1,p}(\Omega),$ such that $\lim
J_{\lambda}(v_{\lambda,m})=c_{\lambda}$ and $\lim J_{\lambda}^{\prime
}(v_{\lambda,m})=0$ in $W^{-1,p^{\prime}}\left(  \Omega\right)  .$ From
(\ref{sousg}), the Palais-Smale condition holds: there exists a subsequence,
converging to a function $v_{\lambda}$ strongly in $W_{0}^{1,p}(\Omega),$ and
$J_{\lambda}(v_{\lambda})=c_{\lambda},$ and $J_{\lambda}^{\prime}(v_{\lambda
})=0,$ in other words $v_{\lambda}$ is a solution of (PV$\lambda$). This holds
for a sequence $\left(  \lambda_{n}\right)  $ converging to $\lambda_{1}.$ Let
$v_{n}=v_{\lambda_{n}},$ then $v_{n}$ is a solution of (PV$\lambda_{n}),$
thus
\[
J_{\lambda_{n}}(v_{n})=\lambda_{n}\int_{\Omega}f(v_{n}\varphi(v_{n}%
)-p\Phi(v_{n}))dx=c_{\lambda_{n}}\leqq c_{\lambda}+1.
\]
From Proposition \ref{essen}, $\left(  v_{n}\right)  $ is also bounded in
$W_{0}^{1,p}(\Omega).$ Up to a subsequence ($v_{n})$ converges to a function
$v$ weakly in $W_{0}^{1,p}(\Omega)$ and strongly in $L^{k}(\Omega)$ for any
$k<p^{\ast},$ and a.e. in $\Omega.$ Then ($\lambda_{n}f(1+g(v_{n}))^{p-1})$
converges to $\lambda_{1}f(1+g(v))^{p-1}$strongly in $L^{1}(\Omega)$. From
Remark \ref{conv}, $v$ is a solution of (PV$\lambda_{1}).$ And ($f(v_{n}%
\varphi(v_{n})-p\Phi(v_{n})))$ converges to $f(v\varphi(v)-p\Phi(v))$ strongly
in $L^{1}(\Omega)$ then ($J_{\lambda_{n}}(v_{n}))=(c_{\lambda_{n}})$ converges
to $J_{\lambda}(v),$ thus $J_{\lambda}(v)=c_{\lambda}.$\medskip

\noindent(i) \textbf{Existence of a second solution for }$\lambda
<\lambda^{\ast}.$\medskip

Let $\lambda_{1}<\lambda^{\ast}$ be fixed. Let $\lambda_{2}\in\left(
\lambda_{1},\lambda^{\ast}\right)  ,$ and let $\underline{v}_{\lambda_{1}%
},\underline{v}_{\lambda_{2}}$ be the minimal bounded solutions associated to
$\lambda_{1},\lambda_{2}.$ Then on $\left[  0,\underline{v}_{\lambda_{2}%
}\right]  $ there exists a solution $v_{0}$ minimizing $J_{\lambda_{1}.}.$
From Proposition \ref{cas}, $v_{0}$ is a solution of (PV$\lambda_{1})$ and
$\underline{v}_{\lambda_{1}}$ is minimal, thus $\underline{v}_{\lambda_{1}%
}\leqq v_{0}\leqq\underline{v}_{\lambda_{2}}.\medskip$

$\bullet$ First suppose $p=2$ and $g$ is convex. Then $v_{0}=\underline
{v}_{\lambda_{1}}$ and it is a strict minimum of $J_{\lambda_{1}}.$ Indeed
$\underline{v}_{\lambda_{2}}$ is semi-stable, thus for any $\varphi\in
W_{0}^{1,2}(\Omega),$%
\[
\int_{\Omega}\left\vert \nabla\varphi\right\vert ^{2}dx\geqq\lambda_{2}%
\int_{\Omega}fg^{\prime}(\underline{v}_{\lambda_{2}})\varphi^{2}dx;
\]
and $g^{\prime}(\underline{v}_{\lambda_{2}})\geqq g^{\prime}(\underline
{v}_{\lambda_{1}}),$ thus
\[
J_{\lambda_{1}}^{\prime\prime}(\underline{v}_{\lambda_{1}}).(\varphi
,\varphi)=\int_{\Omega}\left\vert \nabla\varphi\right\vert ^{2}dx-\lambda
_{1}\int_{\Omega}fg^{\prime}(\underline{v}_{\lambda_{1}})\varphi^{2}%
dx\geqq(1-\frac{\lambda_{1}}{\lambda_{2}})\int_{\Omega}\left\vert
\nabla\varphi\right\vert ^{2}dx;
\]
and $J_{\lambda_{1}}^{\prime}(\underline{v}_{\lambda_{1}})=0,$ then
$\underline{v}_{\lambda_{1}}$ is a \textit{strict} local minimum in
$W_{0}^{1,p}(\Omega).$ Then there exists $R_{\lambda_{1}}>0$ and
$w_{\lambda_{1}}\in W_{0}^{1,p}(\Omega)$ with $\left\Vert w_{\lambda
1}\right\Vert _{W_{0}^{1,p}(\Omega)}>R_{\lambda_{1}}$ such that
\[
\inf\left\{  J_{\lambda}(v):\left\Vert v-\underline{v}_{\lambda_{1}%
}\right\Vert _{W_{0}^{1,p}(\Omega)}=R_{\lambda_{1}}\right\}  >J_{\lambda_{1}%
}(\underline{v}_{\lambda_{1}})>J_{\lambda_{1}}(w_{\lambda_{1}}).
\]
Therefore $J_{\lambda_{1}}$ has the geometry of the Mountain Path near
$\underline{v}_{\lambda_{1}}.$ Using the results of \cite{Je} as above, we get
the existence of a solution of (PV$_{\lambda_{1}})$ at a level $c_{\lambda
_{1}}>J_{\lambda_{1}}(\underline{v}_{\lambda_{1}}),$ different from
$\underline{v}_{\lambda_{1}}.\medskip$

$\bullet$ Next suppose that $g$ satisfies condition (\ref{AR}), without
convexity assumption, and $f\in L^{\infty}(\Omega)$. If $v_{0}\neq
\underline{v}_{\lambda_{1}}$ we have constructed a second solution. Next
assume that $v_{0}=\underline{v}_{\lambda_{1}}.$ Since $f\in L^{\infty}%
(\Omega)$, $\underline{v}_{\lambda_{2}}$ and $v_{0}\in C^{1,\alpha}\left(
\bar{\Omega}\right)  .$ From \cite[Theorem 5.2]{AzPeMa}, $v_{0}$ is a local
minimum in $W_{0}^{1,p}(\Omega):$ it minimizes $J_{\lambda_{1}}$ in a ball
$B(v_{0},\delta)$ of $W_{0}^{1,p}(\Omega)$. From (\ref{AR}), we get
$t\varphi(t)\geqq(k+p)\Phi(t)/2$ for $t>A>0$.\ Here the Palais-Smale sequences
are bounded: if $v_{n}\in W_{0}^{1,p}(\Omega)$ satisfies $\lim J_{\lambda
_{1}.}(v_{n})=c$ and if $\xi_{n}=J_{\lambda_{1}.}^{\prime}(v_{n})$ tends to
$0$ in $W^{-1,p^{\prime}}(\Omega),$ one finds, with different constants
$C>0,$
\begin{align*}
\int_{\Omega}\left\vert \nabla v_{n}\right\vert ^{p}dx-\xi_{n}(v_{n})  &
=\lambda_{1}\int_{\Omega}fv_{n}^{+}\varphi(v_{n}^{+})dx-\int_{\Omega}%
fv_{n}^{-}dx\geqq\lambda_{1}\int_{\left\{  v_{n}\geqq A\right\}  }fv_{n}%
^{+}\varphi(v_{n}^{+})dx-C\left\Vert v_{n}\right\Vert _{W_{0}^{1,p}(\Omega)}\\
&  \geqq\lambda_{1}\frac{k+p}{2}\int_{\left\{  v_{n}\geqq A\right\}  }%
f\Phi(v_{n})dx-C\left\Vert v_{n}\right\Vert _{W_{0}^{1,p}(\Omega)}\\
&  \geqq\frac{k+p}{2p}\int_{\Omega}\left\vert \nabla v_{n}\right\vert
^{p}dx-C(1+\left\Vert v_{n}\right\Vert _{W_{0}^{1,p}(\Omega)})
\end{align*}
thus $(v_{n})$ is bounded in $W_{0}^{1,p}(\Omega).$ And there exists a
function $\tilde{v}$ such that $J_{\lambda_{1}.}(\tilde{v})<J_{\lambda_{1}%
.}(v_{1})$ and $\left\Vert \underline{v}_{\lambda_{1}}-\tilde{v}\right\Vert
\geqq1+\delta.$ Let
\[
\tilde{c}_{\lambda}=\inf_{\theta\in\Gamma}\max_{t\in\left[  0,1\right]
}J_{\lambda}(\theta(t))\geqq\max(J_{\lambda_{1}}(\underline{v}_{\lambda_{1}%
}),J_{\lambda_{1}}(\tilde{v}))
\]
where $\Gamma=\left\{  \theta\in C(\left[  0,1\right]  ,W_{0}^{1,p}%
(\Omega)):\theta(0)=\underline{v}_{\lambda_{1}},\theta(1)=w_{\lambda}\right\}
.$ And $\underline{v}_{\lambda_{1}}$ is a local minimum. Then either the
inequality is strict and there exists a solution at level $\tilde{c}_{\lambda
}$ from Moutain Path Theorem. Or $\tilde{c}_{\lambda}=J_{\lambda_{1}}(v_{1})$
and there exists a solution in $W_{0}^{1,p}(\Omega)\backslash B(\underline
{v}_{\lambda_{1}},\delta)$, from the variant of \cite{GhPr}.
\end{proof}

\begin{remark}
In case $p=N,$ assumptions of growth are not needed in Propositions \ref{vlan}
and \ref{flu}: for any $g$ satisfying (\ref{hypg}), convex near $\infty,$ and
$f\in L^{r}\left(  \Omega\right)  ,r>1,$ we have $v^{\ast}\in W_{0}%
^{1,N}(\Omega)\cap L^{\infty}(\Omega),$ from Proposition \ref{Regned}. However
assumption (\ref{hmq}) for some $Q>N-1$ is required in order to get the
multiplicity result of Theorem \ref{main}.
\end{remark}

\section{Problem (PV$\lambda)$ with measures\label{vmeas}}

Here we study the existence of a renormalized solution of problem
\begin{equation}
-\Delta_{p}v=\lambda f(1+g(v))^{p-1}+\mu\hspace{0.5cm}\text{in }\Omega,\qquad
v=0\hspace{0.5cm}\text{on }\partial\Omega\label{hec}%
\end{equation}
where $\mu\in\mathcal{M}_{b}^{+}(\Omega),$ $\mu\neq0.$ The problem is not easy
for $p\neq2.$ Indeed the convergence and stability results relative to problem
(\ref{mu}) are still restrictive, see Theorem \ref{fund}.

\begin{remark}
In order to obtain an existence result, an assumption of slow growth condition
is natural, as well as more assuptions on $f$. Take for example $p=2<N$ and
$g(v)=v^{Q}$ for some $Q>0,$ and let $\mu=\delta_{a}$ be a Dirac mass at some
point $a\in\Omega.$ If $v$ is a solution, then $v(x)\geqq C\left\vert
x-a\right\vert ^{2-N}$near $a;$ then necessarily $\left\vert x-a\right\vert
^{(2-N)Q}f\in L^{1}(\Omega);$ then $Q<N/(N-2)$ if $f\equiv1.$ More generally
if there exists a solution of (\ref{hec}), then $f\mathcal{G}\left(
\mu\right)  \in L^{1}(\Omega),$ where $\mathcal{G}\left(  \mu\right)  $ is the
potential of $\mu.$ This condition is always satisfied if $f\in L^{r}(\Omega)$
for some $r>N/2.$
\end{remark}

The existence result of Theorem \ref{meas} is a consequence of the next
theorem, where $\mu\in\mathcal{M}_{b}(\Omega)$ is arbitrary, without
assumption of sign. It improves a result announced in \cite[Theorem
1.1.]{Gre2} for $Q>1,$ with an incomplete proof. Our result covers the general
case $Q>0$, and gives better informations in the case $Q=p-1$. We give here a
detailed proof, valid for any $p\leqq N,$ where the approximation of the
measure is precised.

\begin{theorem}
\label{exa}Consider the problem
\begin{equation}
-\Delta_{p}U=\lambda h(x,U)+\mu\hspace{0.5cm}\text{in }\Omega,\qquad
U=0\hspace{0.5cm}\text{on }\partial\Omega, \label{PH}%
\end{equation}
where $\mu\in\mathcal{M}_{b}(\Omega),$ and
\[
\left\vert h(x,U)\right\vert \leqq f(x)(K+\left\vert U\right\vert ^{Q}),
\]
with $Q>0$ and $\lambda,K>0,$ and $f\in L^{r}(\Omega)$ with $Qr^{\prime}%
<Q_{1}.$ Then there exists a renormalized solution of (\ref{PH}) in any of the
following cases:%
\begin{align}
Q  &  =p-1\text{\quad and}\quad\lambda<\lambda_{1}(f);\label{c1}\\
0  &  <Q<p-1;\label{c2}\\
Q  &  >p-1\text{\quad and}\quad\lambda\left\Vert f\right\Vert _{L^{r}(\Omega
)}(\lambda K\left\Vert f\right\Vert _{L^{r}(\Omega)}+\left\vert \mu\right\vert
(\Omega))^{(Q-p+1)/(p-1)}\left\vert \Omega\right\vert ^{1/r^{\prime}-Q/Q_{1}%
}\leqq C \label{c3}%
\end{align}
for some $C=C(N,p,Q)$ for $p<N,$ and $C=C(N,Q,K_{N}(\Omega))$ for $p=N.$
\end{theorem}

\begin{proof}
\textbf{(i) Construction of a suitable approximation of }$\mu$. Let
\[
\mu=\mu_{1}-\mu_{2}+\mu_{s}^{+}-\mu_{s}^{-},\qquad\text{with }\mu_{1}=\mu
_{0}^{+},\mu_{2}=\mu_{0}^{-}\in\mathcal{M}_{0}^{+}(\Omega),\text{ \ }\mu
_{s}^{+},\mu_{s}^{-}\in\mathcal{M}_{s}^{+}(\Omega),
\]
thus $\mu_{1}(\Omega)+\mu_{2}(\Omega)+\mu_{s}^{+}(\Omega)+\mu_{s}^{-}(\Omega)$
$\leqq2\left\vert \mu(\Omega)\right\vert .$ Following the proof of
\cite{BoGaOr}, see also \cite{DrPorPri}, for $i=1,2,$ one has
\[
\mu_{i}=\varphi_{i}\gamma_{i},\qquad\text{with }\gamma_{i}\in\mathcal{M}%
_{b}^{+}(\Omega)\cap W^{-1,p^{\prime}}(\Omega)\text{ and }\varphi_{i}\in
L^{1}(\Omega,\gamma_{i}).
\]
Let $(K_{n})_{n\geqq1}$ a increasing sequence of compacts of union $\Omega,$
and set $\nu_{1,i}=T_{1}(\varphi_{i}\chi_{K_{1}})\gamma_{i}$ and $\nu
_{n,i}=T_{n}(\varphi_{i}\chi_{K_{n}})\gamma_{i}-T_{n-1}(\varphi_{i}%
\chi_{K_{n-1}})\gamma_{i}.$ By regularization there exist nonnegative
$\phi_{n,i}$ $\in\mathcal{D}(\Omega)$ such that $\left\Vert \phi_{n,i}%
-\nu_{n,i}\right\Vert _{W^{-1,p^{\prime}}(\Omega)}$ $\leqq2^{-n}\mu_{i}%
(\Omega).$ Then $h_{n,i}=\sum_{1}^{n}\phi_{k,i}\in\mathcal{D}(\Omega)$ and
($h_{n,i})$ converges strongly in $L^{1}(\Omega)$ to a function $h_{i},$ and
$\left\Vert h_{n,i}\right\Vert _{L^{1}(\Omega)}\leqq\mu_{i}(\Omega).$ Also
$G_{n,i}=\sum_{1}^{n}(\nu_{k,i}-\phi_{k,i})\in W^{-1,p^{\prime}}(\Omega
)\cap\mathcal{M}_{b}(\Omega)$ and $\left(  G_{n,i}\right)  $ converges
strongly in $W^{-1,p^{\prime}}(\Omega)$ to some $G_{i}$, and $\mu_{i}%
=h_{i}+G_{i},$ and $\left\Vert G_{n,i}\right\Vert _{\mathcal{M}_{b}(\Omega)}$
$\leqq2\mu_{i}(\Omega).$ Otherwise by regularization there exist nonnegative
$\lambda_{n}^{1}$ and $\lambda_{n}^{2}\in\mathcal{D}(\Omega)$ converging
respectively to $\mu_{s}^{+},\mu_{s}^{-}$ in the narrow topology, with
$\left\Vert \lambda_{n}^{1}\right\Vert _{L^{1}(\Omega)}\leqq\mu_{s}^{+}%
(\Omega),$ $\left\Vert \lambda_{n}^{2}\right\Vert _{L^{1}(\Omega)}\leqq\mu
_{s}^{-}(\Omega).$ Then the sequence of approximations of $\mu$ defined by
\[
\mu_{n}=h_{n,1}-h_{n,2}+G_{n,1}-G_{n,2}+\lambda_{n}^{1}-\lambda_{n}^{2}%
\]
satisfies the conditions of stability of Theorem \ref{fund}, and moreover is
bounded with respect to $\left\vert \mu\right\vert (\Omega)$ by a universal
constant:
\[
\left\vert \mu_{n}\right\vert (\Omega)\leqq4\left\vert \mu\right\vert
(\Omega).
\]

\noindent\textbf{(ii) The approximate problem.} For any fixed $n\in
\mathbb{N},$ we search a variational solution of
\begin{equation}
-\Delta_{p}U_{n}=\lambda T_{n}(h(x,U_{n}))+\mu_{n}, \label{var}%
\end{equation}
by using the Schauder Theorem. To any $V\in W_{0}^{1,p}(\Omega)$ we associate
the solution $U=\mathcal{F}_{n}(V)\in W_{0}^{1,p}(\Omega)$ of
\[
-\Delta_{p}U=\lambda T_{n}(h(x,V))+\mu_{n},
\]
where $T_{n}$ is the truncation function. We find $\left\Vert \nabla
U\right\Vert _{L^{p}(\Omega)}^{p}\leqq\lambda n\left\Vert U\right\Vert
_{L^{1}(\Omega)}+\left\Vert \mu_{n}\right\Vert _{W^{-1,p^{\prime}}(\Omega
)}\left\Vert U\right\Vert _{W_{0}^{1,p}(\Omega)},$ thus $\left\Vert
U\right\Vert _{W_{0}^{1,p}(\Omega)}\leqq C_{n}$ independent on $V.$ Let
$B_{n}=B(0,C_{n})$ be the ball of $W_{0}^{1,p}(\Omega)$ of radius $C_{n}.$
Then $\mathcal{F}_{n}$ is continuous and compact from $B_{n}$ into itself,
thus it has a fixed point $U_{n}.$ From Proposition \ref{Benilan} and Remark
\ref{estlk}, using (\ref{sig}) with $\sigma=Qr^{\prime}/(p-1),$ we have
\begin{align}
(\int_{\Omega}\left\vert U_{n}\right\vert ^{Qr^{\prime}}dx)^{(p-1)/Qr^{\prime
}}  &  \leqq C_{0}\left\vert \Omega\right\vert ^{\ell}\left(  \lambda
\int_{\Omega}\left\vert T_{n}(h(x,U_{n}))\right\vert dx+\left\vert \mu
_{n}(\Omega)\right\vert \right) \nonumber\\
&  \leqq C_{0}\left\vert \Omega\right\vert ^{\ell}\left(  \lambda\left\Vert
f\right\Vert _{L^{r}(\Omega)}\int_{\Omega}\left\vert U_{n}\right\vert
^{Qr^{\prime}}dx)^{1/r^{\prime}}+\lambda K\left\Vert f\right\Vert
_{L^{r}(\Omega)}+4\left\vert \mu\right\vert (\Omega)\right)  , \label{tre}%
\end{align}
with $\ell=(p-1)/Qr^{\prime}-(N-p)/N,$ and $C_{0}=C_{0}(N,p,Q,r)$ for $p<N,$
and $C_{0}=C_{0}(N,Q,r,K_{N}\left(  \Omega\right)  )$ for $p=N.$\medskip

\noindent\textbf{(iii) Case }$Q<p-1.$ Then from (\ref{tre}), $(\left\vert
U_{n}\right\vert ^{Qr^{\prime}})$ is bounded in $L^{1}(\Omega).$ In turn
($h(x,U_{n}))$ is bounded in $L^{1}(\Omega),$ thus $\left(  -\Delta_{p}%
U_{n}\right)  $ is bounded in $L^{1}(\Omega).$ Then ($\left\vert
U_{n}\right\vert ^{p-1})$ is bounded in $L^{s}(\Omega)$ for any $s\in\left[
1,N/(N-p)\right)  $ (any $s\geqq1$ if $p=N).$ Choosing $s>Qr^{\prime}/(p-1),$
it follows that ($\left\vert h(x,U_{n})\right\vert )$ is bounded in
$L^{1+\varepsilon}(\Omega)$ for some $\varepsilon>0.$ From Theorem \ref{fund}
we can extract a subsequence converging a.e. in $\Omega$ to a renormalized
solution of problem (\ref{PH}).\medskip

\noindent\textbf{(iv) Case }$Q=p-1.$ Assume $\lambda<\lambda_{1}(f).$ Let us
show again that $(\left\vert U_{n}\right\vert ^{Qr^{\prime}})$ is bounded in
$L^{1}(\Omega).$ If not, up to a subsequence, $a_{n}=\int_{\Omega}\left\vert
U_{n}\right\vert ^{(p-1)r^{\prime}}dx$ tends to $\infty$ and we set
$w_{n}=a_{n}^{-1/(p-1)r^{\prime}}U_{n}.$ Then $w_{n}\in W_{0}^{1,p}(\Omega)$,
$\int_{\Omega}w_{n}^{(p-1)r^{\prime}}dx=1,$ and satisfies
\begin{equation}
-\Delta_{p}w_{n}=\eta_{n}+\varphi_{n},\qquad\eta_{n}=a_{n}^{-1/r^{\prime}%
}\lambda T_{n}(h(x,U_{n})),\qquad\varphi_{n}=a_{n}^{-1/r^{\prime}}\mu_{n}.
\label{E35}%
\end{equation}
And ($\varphi_{n})$ converges to $0$ strongly in $L^{1}(\Omega),$ and $\left(
\eta_{n}\right)  $ is bounded in $L^{1}(\Omega)$, since $f\in L^{r}(\Omega)$
and
\[
\left\vert \eta_{n}\right\vert \leqq\psi_{n}=\lambda f(Ca_{n}^{-1/r^{\prime}%
}+\left\vert w_{n}\right\vert ^{p-1}).
\]
From \cite[Section 5.1]{DMOP}, up to a subsequence, $\left(  w_{n}\right)  $
converges a.e. in $\Omega$ to a function $w.$ And ($w_{n}^{(p-1)s})$ is
bounded in $L^{1}(\Omega)$, for any $s<N/(N-p)$, and $r^{\prime}<N/(N-p),$
thus ($\left\vert w_{n}\right\vert ^{(p-1)r^{\prime}})$ converges strongly in
$L^{1}$ $\left(  \Omega\right)  $ to $\left\vert w\right\vert ^{(p-1)r^{\prime
}};$ hence $w\not \equiv 0$. And ($\psi_{n})$ converges strongly in
$L^{1}(\Omega)$ to $\lambda f\left\vert w\right\vert ^{p-1}$, hence ($\eta
_{n})$ converges strongly to some $\eta\in L^{1}$ $\left(  \Omega\right)  .$
Therefore, $w$ is a renormalized solution of problem
\begin{equation}
-\Delta_{p}w=\eta,\hspace{0.5cm}\text{{in }}\Omega,\qquad\text{and }\left\vert
\eta\right\vert \leqq\lambda f(x)\left\vert w\right\vert ^{p-1}\text{ a.e. in
}\Omega. \label{E38}%
\end{equation}
From Proposition \ref{cig} (i), we get $w\in W_{0}^{1,p}(\Omega),$ since
$r>N/p$. Then
\[
\lambda_{1}(f)\int_{\Omega}\left\vert w\right\vert ^{p}dx\leqq\int_{\Omega
}\left\vert \nabla w\right\vert ^{p}dx\leqq\lambda\int_{\Omega}f\left\vert
w\right\vert ^{p}dx,
\]
which is contradictory. Then as above there exists a renormalized solution of
problem (\ref{PH}).\medskip

\noindent\textbf{(v) Case }$Q>p-1.$ Here the estimate of $(\left\vert
U_{n}\right\vert ^{Qr^{\prime}})$ does not hold, but we construct a special
approximation $(U_{n})$ satisfying the estimate: we still have, for any $V\in
W_{0}^{1,p}(\Omega)$ and $U=\mathcal{F}_{n}(V),$%

\begin{equation}
(\int_{\Omega}\left\vert U\right\vert ^{Qr^{\prime}}dx)^{(p-1)/Qr^{\prime}%
}\leqq C_{0}\left\vert \Omega\right\vert ^{\ell}\left(  \lambda\left\Vert
f\right\Vert _{L^{r}(\Omega)}\int_{\Omega}\left\vert V\right\vert
^{Qr^{\prime}}dx)^{1/r^{\prime}}+\lambda K\left\Vert f\right\Vert
_{L^{r}(\Omega)}+4\left\vert \mu\right\vert (\Omega)\right)  .\nonumber
\end{equation}
Setting
\[
x(V)=\int_{\Omega}\left\vert V\right\vert ^{Qr^{\prime}}dx)^{(p-1)/Qr^{\prime
}},\quad a=C_{0}\left\vert \Omega\right\vert ^{\ell}(\lambda K\left\Vert
f\right\Vert _{L^{r}(\Omega)}+4\left\vert \mu\right\vert (\Omega)),\quad
b=C_{0}\left\vert \Omega\right\vert ^{\ell}\lambda\left\Vert f\right\Vert
_{L^{r}(\Omega)},
\]
we find $x(U)\leqq a+bx(V)^{Q/(p-1)}.$ Since $Q>p-1,$ then $x(U)<x(V)$ as soon
as $a^{(Q-p+1)/(p-1)}b\leqq C(p,Q)$, which is assumed in (\ref{c3}) and
$x(V)\leqq y=y(a,b,p,Q)$ small enough. Using the Schauder Theorem in the set
of functions $V\in W_{0}^{1,p}(\Omega)$ such that $x(V)\leqq y$ and
$\left\Vert U\right\Vert _{W_{0}^{1,p}(\Omega)}\leqq C_{n}$ there exists a
solution $U_{n}\in W_{0}^{1,p}(\Omega)$ of (\ref{var}) such that $\int
_{\Omega}\left\vert U_{n}\right\vert ^{Qr^{\prime}}dx$ is bounded. We conclude
as before.
\end{proof}

\begin{remark}
In case $Q=p-1,$ condition (\ref{c1}) is sharp, from Theorem \ref{T2}. The
proof given above for $Q>p-1$ still works for $Q=p-1,$ but condition
(\ref{c3}) obtained in that case is not sharp.
\end{remark}

We end this paragraph by an non existence result.

\begin{proposition}
\label{nonex} Let $\mu_{s}\in\mathcal{M}_{s}^{+}(\Omega)$ be any singular measure.

(i) For any $\lambda>\lambda_{1}(f),$ or $\lambda=\lambda_{1}(f)$ and $f\in
L^{N/p}(\Omega),p<N,$ there is no solution $v$ of
\begin{equation}
-\Delta_{p}v=\lambda f(1+v)^{p-1}+\mu_{s}\hspace{0.5cm}\text{in }\Omega,\qquad
v=0\hspace{0.5cm}\text{on }\partial\Omega. \label{xx}%
\end{equation}

(ii) Let $g$ be defined on $\left[  0,\infty\right)  $ and $\underline{\lim
}_{t\rightarrow\infty}g(\tau)/\tau>0.$ If $\lambda>$ $\lambda_{r},$ there is
no solution $v$ of%
\[
-\Delta_{p}v=\lambda f(1+g(v))^{p-1}+\mu_{s}\hspace{0.5cm}\text{in }%
\Omega,\qquad v=0\hspace{0.5cm}\text{on }\partial\Omega.
\]

\end{proposition}

\begin{proof}
It follows from Lemma \ref{Ponce}: if there exists a solution with a measure,
there exists a solution without measure. In case (i) it follows also from
Theorems \ref{sim} and \ref{TP}: problem (\ref{M}) has no solution, thus the
same happens for problem (\ref{xx}).
\end{proof}

\section{Applications to problem (PU$\lambda$)\label{ret}}

From the existence results obtained for problem (PV$\lambda),$ we deduce
existence results for problem (PU$\lambda)$ by using Theorem \ref{TP}.
Starting from a function $\beta$ satisfying (\ref{hypb}), we associate to
$\beta$ the function $g$ defined by the change of unknown, namely by
(\ref{defg}). We recall that if $\beta$ is defined on $\left[  0,\infty
\right)  ,$ then also is $g;$ conversely if $g$ is defined on $\left[
0,\infty\right)  ,$ then $L<\infty$ if and only if $1/(1+g(v))\in L^{1}\left(
\left(  0,\infty\right)  \right)  ,$ from Remark \ref{tru}. In some results we
assume that $g$ satisfies (\ref{hmq}):%
\[
\overline{\lim}_{\tau\longrightarrow\infty}\frac{g(\tau)^{p-1}}{\tau^{Q}%
}<\infty
\]
for some $Q>0$ or equivalently
\begin{equation}
\overline{\lim}_{t\longrightarrow L}\frac{e^{\gamma(t)}}{\Psi^{Q}(t)}<\infty.
\label{mbq}%
\end{equation}

In the case $\beta$ constant Theorem \ref{T2} follows:\medskip

\begin{proof}
[Proof of Theorem \ref{T2}]Any renormalized solution $u$ of (\ref{M})
satisfies $(p-1)\left\vert \nabla u\right\vert ^{p}\in L^{1}(\Omega),$ thus
$u\in W_{0}^{1,p}(\Omega).$. If $\lambda<\lambda_{1}(f),$ there exists a
unique solution $v_{0}\in W_{0}^{1,p}(\Omega)$ of (\ref{PL}) fromTheorem
\ref{sim}. Then from Theorem \ref{TP}, $u_{0}=H(v_{0})$ is a solution of
(\ref{M}) such that $v_{0}=\Psi(u_{0})=e^{u_{0}}-1\in W_{0}^{1,p}(\Omega).$
Reciprocally, if $u$ is a solution of (\ref{M}), such that $v=\Psi(u)\in
W_{0}^{1,p}(\Omega),$ then from Theorem \ref{TP}, $v$ is a reachable solution
of
\[
-\Delta_{p}v=\lambda f(1+v)^{p-1}+\mu
\]
for some measure $\mu\in\mathcal{M}_{b}^{+}\left(  \Omega\right)  .$ Since
$v\in W_{0}^{1,p}(\Omega),$ then $\mu\in\mathcal{M}_{0}(\Omega).$ Then from
existence and uniqueness of the solutions of (\ref{mu}) when $\mu
\in\mathcal{M}_{0}(\Omega)$, $v$ is also a renormalized solution; as in the
proof of Theorem \ref{TP} (case $p=2$ or $N$), it follows that $\mu\in
M_{s}^{+}\left(  \Omega\right)  ,$ thus $\mu=0,$ and $v=v_{0,}$ then
$u=u_{0}.$ If $f\in L^{N/p}(\Omega),$ then $v_{0}\in L^{k}(\Omega)$ for any
$k>1,$ and also $u_{0},$ since $u_{0}\leqq v_{0}.$ If $f\in L^{r}\left(
\Omega\right)  ,r>N/p$, then $u_{0},v_{0}\in L^{\infty}(\Omega);$ and for any
$\mu_{s}\in M_{s}^{+}\left(  \Omega\right)  $ there exists a solution $v_{s}$
of (\ref{pps}) from Theorem \ref{meas}, thus a corresponding solution
$u_{s}\in W_{0}^{1,p}(\Omega)$ of (\ref{M}). The nonexistence follows from
Proposition \ref{nonex}.\medskip
\end{proof}

Our next result follows from Corollary \ref{almo}, Theorem \ref{meas} and
Propositions \ref{subli}, \ref{san}:

\begin{corollary}
Assume that $\beta$ satisfies (\ref{hypb}) with $L=\infty.\medskip$

(i) Suppose that $g$ satisfies (\ref{hmq}) with $Q=p-1$.

If $M_{p-1}\lambda<\lambda_{1}(f),$ there exists at least a solution $u\in
W_{0}^{1,p}(\Omega)$ to problem (PU$\lambda$). If moreover $f\in
L^{N/p}(\Omega),p<N,$ then $u\in L^{k}(\Omega)$ for any $k>1.$ If $f\in
L^{r}\left(  \Omega\right)  ,r>N/p$, then $u\in L^{\infty}(\Omega);$ and there
exists an infinity of less regular solutions $u_{s}\in$ of (PU$\lambda
$).\medskip

(ii) Suppose (\ref{hmq}) with $Q<p-1,$ and $f\in L^{r}(\Omega)$ with
$r\in(1,N/p)$ such that $Qr^{\prime}<Q_{1}.$

Then for any $\lambda>0$ there exists a renormalized solution $u$ of
(PV$\lambda)$ such that $v=\Psi(u)$ satisfies $v^{d}\in L^{1}(\Omega)$ for
$d=Nr(p-1-Q)/(N-pr).$ If ($Q+1)r^{\prime}\leqq p^{\ast}$, then $u\in
W_{0}^{1,p}(\Omega).$ If ($Q+1)r^{\prime}>p^{\ast},$ then $\left\vert \nabla
u\right\vert ^{\theta}\in L^{1}(\Omega)$ for $\theta=Nr(p-1-Q)/(N-(Q+1)r).$
There exists also an infinity of less regular solutions $u_{s}$ of
(PU$\lambda$).\medskip

(iii) Suppose (\ref{hmq}) with $p-1<Q<Q_{1},$ and $f\in L^{r}(\Omega)$ with
$Qr^{\prime}<Q_{1}.$

Then for $\lambda>0$ small enough, there exists a solution $u\in W_{0}%
^{1,p}(\Omega)\cap L^{\infty}(\Omega)$ of (PU$\lambda),$ and an infinity of
less regular solutions.
\end{corollary}

From Proposition \ref{elin} and Theorem \ref{impo} we deduce the following:

\begin{corollary}
(i) Assume (\ref{hypb}), and $f\in L^{r}\left(  \Omega\right)  ,r>N/p$. Then
for $\lambda>0$ small enough, there exists a minimal solution $\underline
{u}_{\lambda}\in W_{0}^{1,p}(\Omega)\cap L^{\infty}\left(  \Omega\right)  $ of
(PU$\lambda),$ with $\left\Vert \underline{u}_{\lambda}\right\Vert
_{L^{\infty}\left(  \Omega\right)  }<L.$\medskip

\noindent(ii) Suppose moreover that $\underline{\lim}_{t\rightarrow L}%
\beta(t)>0$ and $t\beta(t)$ is nondecreasing near $L,$ and $f\not \equiv 0,.$
Then there exists $\lambda^{\ast}>0$ such that

if $\lambda\in\left(  0,\lambda^{\ast}\right)  $ there exists a minimal
solution $\underline{u}_{\lambda}\in W_{0}^{1,p}(\Omega)\cap L^{\infty}\left(
\Omega\right)  $ of (PU$\lambda),$ with $\left\Vert \underline{u}_{\lambda
}\right\Vert _{L^{\infty}\left(  \Omega\right)  }<L;$

if $\lambda>\lambda^{\ast}$ there exists no renormalized solution.
\end{corollary}

From Theorems and \ref{trema}, and \ref{main} and Remark \ref{dou}, we obtain
the following:

\begin{corollary}
\label{plic}Assume (\ref{hypb}) and $f\in L^{r}\left(  \Omega\right)  ,r>N/p$,
$f\not \equiv 0.$ Suppose that $\beta$ is nondecreasing near $L$ and
$\lim_{t\longrightarrow L}\beta(t)=\infty,$ and $e^{\gamma(t)/(p-1)}%
\not \in L^{1}\left(  \Omega\right)  .$\medskip

(i) Then $u^{\ast}=\sup_{\lambda\nearrow\lambda^{\ast}}\underline{u}_{\lambda
}$ is a solution of (PU$\lambda^{\ast})$, and $u^{\ast}\in W_{0}^{1,p}%
(\Omega).$ If one of the conditions (i) (ii) (iii) of Theorem \ref{trema}
holds, then $\left\Vert u^{\ast}\right\Vert _{L^{\infty}\left(  \Omega\right)
}<L.$\medskip

(ii) Suppose moreover that(\ref{hmq}) holds with $Q<Q^{\ast}$, and $f\in
L^{r}(\Omega)$ with $(Q+1)r^{\prime}<p^{\ast}.$ Then for small $\lambda>0$
there exists at least two solutions of (PU$\lambda)$ such that $\left\Vert
u\right\Vert _{L^{\infty}\left(  \Omega\right)  }<L.$ It is true for any
$\lambda<\lambda^{\ast}$ when $p=2$ and $\beta$ is nondecreasing.
\end{corollary}

\subsection{Remarks on growth assumptions}

Condition (\ref{mbq}) is not easy to verify. It is implied by
\begin{equation}
\overline{\lim}_{t\longrightarrow L}\frac{\beta(t)}{\Psi^{Q/(p-1)-1}%
(t)}<\infty\label{dq}%
\end{equation}
from the L'Hospital rule. If moreover $\beta$ is nondecreasing, the two
conditions are equivalent.\medskip

\begin{remark}
If $\beta=\beta_{1}+\beta_{2},$ where $\beta_{1}\in L^{1}\left(  \left(
0,L\right)  \right)  $ and $\Lambda_{2}=\infty$ and $\beta_{2}$ satisfies
(\ref{mbq}), then $\beta$ satisfies (\ref{mbq}). Indeed setting $v=\Psi(u),$
$v_{1}=\Psi_{1}(u)$ and $v_{2}=\Psi_{2}(u),$ one finds $v_{2}\leqq v$ and
\[
\frac{1+g(v)}{v^{q}}\leqq e^{\gamma(L)/(p-1)}\frac{1+g_{2}(v_{2})}{v^{q}}\leqq
e^{\gamma(L)/(p-1)}\frac{1+g_{2}(v_{2})}{v_{2}^{q}}.
\]
In particular (\ref{mbq}) is satisfied with $Q=p-1$ by any $\beta$ of this
form, such that $\beta_{2}$ is bounded.\medskip
\end{remark}

Next we give a simple condition on $\beta$ ensuring (\ref{mbq}):

\begin{lemma}
Let $Q>0.$ Assume that $\beta\in C^{1}(\left[  0,L\right)  ),$ and $L=\infty$
or only $e^{\gamma(\theta)/(p-1)}\not \in L^{1}\left(  \left(  0,L\right)
\right)  ,$ and
\begin{equation}
\overline{\lim}_{t\longrightarrow L}\frac{\beta^{\prime}}{\beta^{2}}%
(t)\leqq1-\frac{p-1}{Q}. \label{cdtq}%
\end{equation}
Then (\ref{mbq}) holds.
\end{lemma}

\begin{proof}
The conditions imply $\Lambda=\infty$ and
\[
\overline{\lim}_{t\longrightarrow L}\frac{\beta^{\prime}}{\beta^{2}%
}(t)=\overline{\lim}_{t\longrightarrow L}\frac{gg^{\prime\prime}}{g^{\prime2}%
}(\Psi(t))=\overline{\lim}_{\tau\longrightarrow\infty}\frac{gg^{\prime\prime}%
}{g^{\prime2}}(\tau);
\]
then (\ref{cdtq}) implies that $g^{(p-1)/Q}$ is concave near $\infty,$ thus at
most linear.
\end{proof}

\begin{remark}
As observed in \cite{AAP}, many "elementary" nondecreasing functions $\beta$
on $\left[  0,\infty\right)  $ satisfy condition (\ref{mbq}) for \textbf{any}
$Q>p-1.$ In the examples of Section \ref{Cor}, we have seen that for
$\beta(u)=u^{m},$ $m>0,$ $g(v)=O(v(\ln v)^{m/(m+1)})$ near $\infty.$ For
$\beta(u)=e^{u},$ $g(v)=O(v\ln v)$ near $\infty.$ For $\beta(u)=e^{e^{u}%
+u}+e^{u}+1,$ $g(v)=$ $O(v\ln v\ln(\ln v)).$ In those cases, $\overline{\lim
}_{t\longrightarrow\infty}(\beta^{\prime}/\beta^{2})(t)=0.$
\end{remark}

An open question raised in \cite{AAP} and also \cite{DaAGiSe} was to know if
\textbf{any} nondecreasing $\beta$ defined on $\left[  0,\infty\right)  $
satisfies (\ref{mbq}) for some $Q>p-1.$ Here we show that condition
(\ref{mbq}) is \textbf{not always satisfied}, even with large $Q,$ even when
$\tau^{Q}$ is replaced by an exponential:

\begin{lemma}
Consider any function $F\in C^{0}(\left[  0,\infty\right)  )$ strictly convex,
with $\lim_{s\longrightarrow\infty}F(s)=\infty.$ Then there exists a function
$\beta\in C^{0}(\left[  0,\infty\right)  ,$ increasing with $\beta(0)\geqq0,$
$\lim_{t\longrightarrow\infty}\beta(t)=\infty$ such that the corresponding
function $g$ given by (\ref{gbt}) satisfies
\begin{equation}
\overline{\lim}_{\tau\longrightarrow\infty}\frac{g(\tau)}{F(\tau)}=\infty.
\label{lims}%
\end{equation}

\end{lemma}

\begin{proof}
From Remark \ref{tru} there is a one-to-one mapping between such a function
and a function $g\in C^{1}(\left[  0,\infty\right)  ),$ convex, such that
$\lim_{s\longrightarrow\infty}g(s)/s=\infty,$ and
\[
1/(1+g(s))\not \in L^{1}\left(  \left(  0,\infty\right)  \right)  .
\]
Thus it is sufficient to show the existence of such a function $g$ satisfying
(\ref{lims}). We first construct a function $g$ which is only continuous. Let
$\mathcal{F}$ be the curve defined by $F.$ Set $g(s)=0$ for $s\in\left[
0,1\right]  .$ There exists $m_{1}>1$ such that the line of slope $m_{1}$
issued from $(1,0)$ cuts $\mathcal{F}$ at two points $s_{1}^{\prime}%
<s_{1}^{\prime\prime}.$ Then we define $g(s)=m_{1}(s-1)$ for any $s\in\left[
1,s_{1}\right]  ,$ where $s_{1}>s_{1}^{\prime\prime}$ is chosen such that
$s_{1}-1\geqq(1+g(1))e^{m_{1}},$ that means $s_{1}\geqq1+e^{m_{1}}.$ Then%
\[
\int_{1}^{s_{1}}ds/(1+g(s)\geqq1,
\]
and the point ($s_{1},g(s_{1}))$ is under $\mathcal{F}.$ By induction for any
$n\geqq1$, we consider $m_{n}>2m_{n-1}$ such that the line of slope $m_{n}$
issued from $(s_{n-1},g(s_{n-1}))$ cuts the curve $\mathcal{F}_{n}$ defined by
$nF$ at two points $s_{n}^{\prime}<s_{n}^{\prime\prime}.$ We define
$g(s)=g(s_{n-1})+m_{n}(s-s_{n-1})$ for any $s\in\left[  s_{n-1},s_{n}\right]
,$ where $s_{n}>s_{n}^{\prime\prime}$ is chosed such that $s_{n}-s_{n-1}%
\geqq(1+g(s_{n-1}))e^{m_{n}}$ and $s_{n}\geqq2s_{n-1}.$Then
\[
\int_{s_{n-1}}^{s_{n}}ds/(1+g(s)\geqq1.
\]
The function $g$ satisfies $1/(1+g(s))\not \in L^{1}\left(  \left(
0,\infty\right)  \right)  ,$ and $g\geqq nF$ on $\left[  s_{n}^{\prime}%
,s_{n}^{\prime\prime}\right]  ,$ and $s_{n}^{\prime}>s_{n}$ $>1,$ thus
(\ref{lims}) holds; and $g(s_{n})\geqq m_{n}(s_{n}-s_{n-1})\geqq m_{n}%
s_{n}/2,$ then $\lim_{s\longrightarrow\infty}g(s)/s=\infty.$ Then we
regularize $g$ near the points $s_{n}$ in order to get a $C^{1}$ convex function.
\end{proof}

\subsection{Extensions}

1) In the correlation Theorem \ref{TP}, we can assume that $f$ depends also on
$u$ or $v.$ If $u$ is a solution of a problem of the form
\[
-\Delta_{p}u=\beta(u)\left\vert \nabla u\right\vert ^{p}+\lambda f(x,u),
\]
where $f(x,u)\in L^{1}(\Omega),$ $f(x,u)\geqq0,$ then formally $v$ is a
solution of
\[
-\Delta_{p}v=\lambda f(x,H(v))(1+g(v))^{p-1}.
\]
Conversely, if $v$ is a solution of a problem of the form
\[
-\Delta_{p}v=\lambda f(x,v)(1+g(v))^{p-1},
\]
then formally $u$ is a solution of
\[
-\Delta_{p}u=\beta(u)\left\vert \nabla u\right\vert ^{p}+\lambda
f(x,\Psi(u)).
\]
This extends \textbf{strongly} the domain of applications of our
result.\medskip

\begin{remark}
This argument was an essential point in the Proof of Theorem \ref{truc}: we
used the fact that, for any $g$ satisfying (\ref{hypo}) with $\Lambda=\infty,$
and any $v\in\mathcal{W}(\Omega),$ such that $-\Delta_{p}v=F\geqq0$, then
$u=H(v)\in\mathcal{W}$ and is a solution of equation $-\Delta_{p}%
u=\beta(u)\left\vert \nabla u\right\vert ^{p}+Fe^{-\gamma(u)}.$\medskip
\end{remark}

Let us give a simple example of application:

\begin{corollary}
Let $\omega\in C^{1}\left(  \left[  0,\infty\right)  \right)  $ be nonnegative
and nondecreasing, and $f\in L^{r}\left(  \Omega\right)  ,r>N/p$. Consider the
problem%
\[
-\Delta_{p}u=(p-1)\left\vert \nabla u\right\vert ^{p}+\lambda f(x)(1+\omega
(u))^{p-1},\qquad u=0\hspace{0.5cm}\text{on }\partial\Omega.
\]
(i) Then for small $\lambda>0,$ there exists a solution in $W_{0}^{1,p}%
(\Omega)\cap L^{\infty}\left(  \Omega\right)  .\medskip$

\noindent(ii) Assume that $\lim\sup_{t\longrightarrow\infty}\omega
(t)^{p-1}/e^{kt}<\infty$ for some $k>0$.

If $r^{\prime}(k+1)<N/(N-p)$ then for any small $\lambda>0,$ there exists an
infinity of solutions in $W_{0}^{1,p}(\Omega).$

If $r^{\prime}(k/p^{\prime}+1)<N/(N-p)$ and $\omega$ is convex, there exists
two solutions in $W_{0}^{1,p}(\Omega)\cap L^{\infty}\left(  \Omega\right)  .$
\end{corollary}

\begin{proof}
Setting $v=e^{u}-1,$ then $v$ satisfies the equation $-\Delta_{p}v=\lambda
f(x)(1+\tilde{g}(v))^{p-1}$ in $\Omega,$ where $1+\tilde{g}(v)=(1+v)(1+\omega
(\ln(1+v))).$ And $\tilde{g}$ satisfies (\ref{hmq}) with $Q=(p-1)(k+1),$ and
is convex when $\omega$ is convex. The results follows from Proposition
\ref{elin}, Theorems \ref{TP}, \ref{meas} and \ref{main}.$\medskip$
\end{proof}

\begin{remark}
In particular for any $b>0,$ for any $f\in L^{r}\left(  \Omega\right)
,r>N/p$, and small $\lambda>0,$ problem
\[
-\Delta_{p}u=\left\vert \nabla u\right\vert ^{p}+\lambda f(x)(1+u)^{b}%
\hspace{0.5cm}\text{in }\Omega,\qquad u=0\hspace{0.5cm}\text{on }%
\partial\Omega,
\]
has an infinity of solutions in $W_{0}^{1,p}(\Omega),$ one of them in
$L^{\infty}\left(  \Omega\right)  ,$ two of them if $b\geqq p-1.\medskip$
\end{remark}

2) Theorem \ref{TP} also covers and precises the recent multiplicity result of
\cite[Theorem 3.1]{Ab}, relative to radial solutions of problems with other
powers of the gradient:

\begin{corollary}
Let $\Omega=B(0,1).$ Consider the problem%
\begin{equation}
-\Delta_{m}w=c\left\vert \nabla w\right\vert ^{q}+\lambda f\hspace
{0.5cm}\text{in }\Omega,\qquad w=0\hspace{0.5cm}\text{on }\partial\Omega,
\label{Pr}%
\end{equation}
with $m>1$ and $q\geqq(m-1)N/(N-1),$ where $f$ is radial and $f\in
L^{r}\left(  \Omega\right)  ,r>N(q-m+1)/q,$ and $c>0.$ Then there exists
$\tilde{\lambda}>0$ such that for any $\lambda<\tilde{\lambda}$, problem
(\ref{Pr}) in $\mathcal{D}^{\prime}(\Omega)$ admits an infinity of radial
solutions, and one of them in $C^{1}\left(  \overline{\Omega}\right)  $.
\end{corollary}

\begin{proof}
In the radial case, problem (\ref{Pr}) only involves the derivative
$w^{\prime}:$%
\begin{equation}
-r^{1-N}(r^{N-1}\left\vert w^{\prime}\right\vert ^{m-2}w^{\prime})^{\prime
}=c\left\vert w^{\prime}\right\vert ^{q}+\lambda f \label{pwf}%
\end{equation}
hence the change of functions $w^{\prime}=A\left\vert u^{\prime}\right\vert
^{p/q-1}u^{\prime}$ with $p=q/(q-m+1)$ and $A=(c/(p-1)^{-p/q}.$ reduces
formally to
\begin{equation}
-r^{1-N}(r^{N-1}\left\vert u^{\prime}\right\vert ^{p-2}u^{\prime})^{\prime
}=((p-1)\left\vert u^{\prime}\right\vert ^{p}+\rho f, \label{puf}%
\end{equation}
where $\rho=(c/(p-1))^{p-1}\lambda.$ By hypothesis, $1<p\leqq N,$ and $f\in
L^{r}\left(  \Omega\right)  ,r>N/p.$ From Theorem \ref{T2}, for any
$\rho<\lambda_{1}(f)$ defined at (\ref{VP}), and for any measure $\mu_{s}$
$\in\mathcal{M}_{s}^{+}(B(0,1))$ there exists a renormalized nonnegative
solution $v_{s}$ of problem
\begin{equation}
-\Delta_{p}v_{s}=\rho f(1+v_{s})^{p-1}+\mu_{s}\hspace{0.5cm}\text{in }%
\Omega,\qquad v_{s}=0\quad\text{on }\partial\Omega; \label{Pvr}%
\end{equation}
thus there exists an infinity of nonnegative solutions $u_{s}=\ln(1+v_{s})\in
W_{0}^{1,p}(\Omega)$ of
\[
-\Delta_{p}u_{s}=(p-1)\left\vert \nabla u_{s}\right\vert ^{p}+\rho
f\hspace{0.5cm}\text{in }\Omega.
\]
Take $\mu_{s,a}=a\delta_{0},$ with $a>0$. Then (\ref{Pvr}) has at least a
radial solution $v_{s,a}$, obtained as in Theorem \ref{exa} by the Schauder
theorem for radial functions. Then $u=u_{s,a}$ is radial, and $r\mapsto u(r)$
satisfies (\ref{puf}) in $\mathcal{D}^{\prime}(\left(  0,1)\right)  $, hence
$u\in C^{1}\left(  \left(  0,1\right]  \right)  $ and $u^{\prime}(r)<0.$ Then
$w(r)=-A\int_{r}^{1}\left\vert u^{\prime}\right\vert ^{p/q-1}u^{\prime}ds\in
C^{1}\left(  \left(  0,1\right]  \right)  ,$ $w(r)\geqq0$ and $w$ satisfies
(\ref{Pvr}) in $\mathcal{D}^{\prime}(\left(  0,1)\right)  $ with
$\lambda=((p-1)/c)^{p-1}\rho.$ Moreover $x\mapsto w^{\prime}(\left\vert
x\right\vert )\in L^{q}\left(  \Omega\backslash\left\{  0\right\}  \right)  ,$
and $\left\{  0\right\}  $ has a $p$-capacity 0 since $p\leqq N,$ thus $w\in
W_{0}^{1,q}\left(  \Omega\right)  ,$ hence $\left\vert \nabla w\right\vert
^{m-1}\in L^{p^{\prime}}(\Omega).$ Let $\varphi\in\mathcal{D}(\Omega)$ and
$\varphi_{n}\in\mathcal{D}(\left(  \Omega\backslash\left\{  0\right\}
\right)  $ converging to $\varphi$ in $W_{0}^{1,p}(\Omega).$ Then
\begin{align*}
\int_{\Omega}\left\vert \nabla w\right\vert ^{m-2}\nabla w.\nabla\varphi
_{n}dx  &  =A^{m-1}\int_{\Omega}\left\vert \nabla u\right\vert ^{p-2}\nabla
u.\nabla\varphi_{n}dx\\
&  =A^{m-1}\int_{\Omega}((p-1)\left\vert \nabla u\right\vert ^{p}+\rho
f)\varphi_{n}dx=\int_{\Omega}(c\left\vert \nabla w\right\vert ^{q}+\rho
f)\varphi_{n}dx;
\end{align*}
going to the limit, we find that $w$ is a solution of (\ref{Pr}) in
$\mathcal{D}^{\prime}(\Omega).$ Then there exists an infinity of radial
solutions of (\ref{Pr}) for any $\lambda<\tilde{\lambda}=((p-1)/c)^{p-1}%
\lambda_{1}(f).$ And taking $\mu_{s,a}=0,$ the problem in $u$ admits a bounded
radial solution $u_{0}\in C^{1}\left(  \left[  0,1\right]  \right)  $, thus
(\ref{Pr}) admits a radial solution $w_{0}\in C^{1}\left(  \overline{\Omega
}\right)  .$
\end{proof}

\begin{remark}
Moreover, since $v_{s}$ is radial, from the assumptions on $f,$ we know the
precise behaviour near $0$ of the singular solutions:

If $q>(m-1)N/(N-1),$ in otherwords $p<N,$ then $v(r)=c_{N,p}ar^{(p-N)/(p-1)}%
(1+o(1))$ near $0,$ with $c_{N,p}=(p-1)(n-p)^{-1}\left\vert S_{N-1}\right\vert
^{-1/(p-1)};$ and $v^{\prime}(r)=c_{N}a(p-N)(p-1)^{-1}r^{(1-N)(p-1)}(1+o(1)).$
And $u^{\prime}=v^{\prime}/(1+v),$ thus $\left\vert u^{\prime}\right\vert
^{p/q-1}u^{\prime}=-((N-p)/(p-1)r)^{-p/q}(1+o(1)).$ If $q>m,$ that means if
$q>p,$ then $w$ is bounded, the singularity appears at the level of the
gradient. If $q<m,$ then $w(r)=Cr^{-(m-q)/(q-m+1)}(1+o(1)),$ with
$C=C(N,m,q,c).$ If $q=m-1,$ then $w(r)=C(-\ln r)^{-1}(1+o(1)).$

If $q=(m-1)N/(N-1),$ then $p=N,$ and $\lim_{r\rightarrow0}(-\ln r)^{-1}%
v(r)=c_{N}a$ with $c_{N}=\left\vert S_{N-1}\right\vert ^{-1/(N-1)},$
$\lim_{r\rightarrow0}rv^{\prime}(r)=-c_{N}a,$ $\left\vert u^{\prime
}\right\vert ^{p/q-1}u^{\prime}=-(r(-\ln r))^{-(N-1)/(m-1)}(1+o(1)).$ If
$N<m,$ then $w$ is bounded; if $N>m,$ then $w=C(-\ln r)^{-(N-1)/(m-1)}%
r^{-(N-m)/(m-1)}(1+o(1)),$ with $C=C(N,m,c).$ if $N=m,$ then $w=C(\ln(-\ln
r)(1+o(1)).$
\end{remark}

\section{Appendix}

\begin{proof}
[Proof of Lemma \ref{Pic}]The relation is known for $V\in W_{0}^{1,p}%
(\Omega),$ see for example \cite{AllHu}. Let $F=-\Delta_{p}V,$ and $F_{n}%
=\min(F,n)\in L^{\infty}(\Omega),$ and $V_{n}=\mathcal{G}(F_{n}).$ Then
$F_{n}\rightarrow F$ in $L^{1}(\Omega).$ And ($V_{n})$ is nondecreasing; from
Remark \ref{conv}, $\left(  V_{n}\right)  $ converges a.e. to a renormalized
solution $w$ of $-\Delta_{p}v=-\Delta_{p}V;$ from uniqueness, $w=V$; and
\[
\int_{\Omega}\left\vert \nabla U\right\vert ^{p}dx\geqq\int_{\Omega}U^{p}%
V_{n}^{1-p}(-\Delta_{p}V_{n})dx.
\]
From the Fatou Lemma $U^{p}V^{1-p}(-\Delta_{p}V)\in L^{1}(\Omega),$ and
(\ref{picone}) holds.$\medskip$
\end{proof}

\begin{proof}
[Proof of Lemma \ref{boot}]We have $\overline{m}\in\left(  1,N/p\right)  $ for
$p<N,$ and $\overline{m}=1$ for $p=N.\medskip$

\textbf{ }$\bullet$ First suppose $1<m<N/p,$ thus $p<N.$ Let $\varepsilon>0$
and $k>0.$ We use the test function $\phi_{\beta,\varepsilon}(T_{k}(U))$,
where $\phi_{\beta,\varepsilon}(w)=$ $\int_{0}^{w}(\varepsilon+\left\vert
t\right\vert )^{-\beta}dt$, for given real $\beta<1.$We get
\begin{equation}
\int_{\Omega}\frac{\left\vert \nabla T_{k}(U)\right\vert ^{p}}{(\varepsilon
+\left\vert T_{k}(U)\right\vert )^{\beta}}dx\leqq(1-\beta)^{-1}\int_{\Omega
}\left\vert F\right\vert (\varepsilon+\left\vert T_{k}(U)\right\vert
)^{1-\beta}dx \label{E29b}%
\end{equation}
Setting $\eta=(p-1)mN/(N-m)$ and then $\eta^{\ast}=(p-1)Nm/(N-pm),$ we take
\[
\beta=1-\frac{\eta^{\ast}}{m^{\prime}}=\frac{Np(\overline{m}-m)}{\overline
{m}(N-pm)},\qquad\alpha=1-\frac{\beta}{p}=\frac{\eta^{\ast}}{p^{\ast}};
\]
then $\beta,\alpha\in\left(  0,1\right)  $ for $m<\overline{m},$ and
$\beta\leqq0\leqq\alpha-1$ for $m\geqq\overline{m}$. The function
$U_{k,\varepsilon}=((\varepsilon+\left\vert T_{k}(U)\right\vert )^{\alpha
}-\varepsilon^{\alpha})$sign$(U)$ belongs to $W_{0}^{1,p}(\Omega)$, and from
(\ref{E29b}) we get
\begin{align}
\int_{\Omega}\left\vert \nabla U_{k,\varepsilon}\right\vert ^{p}dx  &
=\alpha^{p}\int_{\Omega}\frac{\left\vert \nabla T_{k}(U)\right\vert ^{p}%
}{(\varepsilon+\left\vert T_{k}(U)\right\vert )^{\beta}}dx\leqq C\int_{\Omega
}\left\vert F\right\vert (\varepsilon^{\alpha}+\left\vert U_{k,\varepsilon
}\right\vert )^{\eta^{\ast}/\alpha m^{\prime}}dx\nonumber\\
&  \leqq C\left\Vert F\right\Vert _{L^{m}(\Omega)}(\int_{\Omega}%
(\varepsilon^{\alpha}+\left\vert U_{k,\varepsilon}\right\vert )^{\eta^{\ast
}/\alpha}dx)^{\frac{1}{m^{\prime}}}, \label{E29C}%
\end{align}
where $C>0$. From the Sobolev injection of $W_{0}^{1,p}(\Omega)$ into
$L^{p^{\ast}}(\Omega)$ we find, with other constants $C>0,$ depending on
$\Omega$
\begin{align*}
(\int_{\Omega}(\varepsilon^{\alpha}+\left\vert U_{k,\varepsilon}\right\vert
)^{p^{\ast}}dx)^{p/p\ast}  &  \leqq C(\varepsilon^{\alpha p}+(\int_{\Omega
}\left\vert U_{k,\varepsilon}\right\vert ^{p^{\ast}}dx)^{p/p\ast}\\
&  \leqq C(\varepsilon^{\alpha p}+\left\Vert F\right\Vert _{L^{m}(\Omega
)}(\int_{\Omega}(\varepsilon^{\alpha}+\left\vert U_{k,\varepsilon}\right\vert
)^{p^{\ast}}dx)^{\frac{1}{m^{\prime}}},
\end{align*}
and $p^{\ast}<pm^{\prime}$, because $m<N/p$; thus from the Young inequality
\begin{equation}
\int_{\Omega}\left\vert T_{k}(U)\right\vert ^{\eta^{\ast}}dx\leqq\int_{\Omega
}(\varepsilon^{\alpha}+\left\vert U_{k,\varepsilon}\right\vert )^{p^{\ast}%
}dx\leqq C(\varepsilon^{\eta^{\ast}}+\left\Vert F\right\Vert _{L^{m}(\Omega
)}^{1/(p/p^{\ast}-1/m^{\prime})}). \label{ED}%
\end{equation}
And $1/(p/p^{\ast}-1/m^{\prime})=\eta^{\ast}/(p-1)$, thus
\[
\int_{\Omega}\left\vert T_{k}(U)\right\vert ^{\eta^{\ast}}dx\leqq C\left\Vert
F\right\Vert _{L^{m}(\Omega)}^{\eta^{\ast}/(p-1)},
\]
and from the Fatou Lemma, (iii) follows:
\begin{equation}
\left(  \int_{\Omega}\left\vert U\right\vert ^{(p-1)Nm/(N-pm)}dx\right)
^{(N-pm)/Nm}\leqq C\left\Vert F\right\Vert _{L^{m}(\Omega)}. \label{estu}%
\end{equation}
$\bullet$ Assume moreover that $m<\overline{m}.$ Using (\ref{E29C}),
(\ref{ED}) and going to the limit as $k\longrightarrow\infty$, we find
\begin{equation}
\int_{\Omega}\frac{\left\vert \nabla U\right\vert ^{p}}{(\varepsilon
+\left\vert U)\right\vert )^{\beta}}dx\leqq C(\left\Vert F\right\Vert
_{L^{m}(\Omega)}\varepsilon^{\eta^{\ast}/m^{\prime}}+\left\Vert F\right\Vert
_{L^{m}(\Omega)}^{\eta^{\ast}p/p^{\ast}(p-1)}). \label{Ee}%
\end{equation}
We have $\eta<p,$ and $\beta\eta/(p-\eta)=\eta^{\ast},$ thus from H\"{o}lder
inequality,
\[
\int_{\Omega}\left\vert \nabla U\right\vert ^{\eta}dx\leqq\left(  \int
_{\Omega}\frac{\left\vert \nabla U\right\vert ^{p}}{(\varepsilon+\left\vert
U)\right\vert )^{\beta}}dx\right)  ^{\eta/p}\left(  \int_{\Omega}%
(\varepsilon+\left\vert U)\right\vert )^{\eta^{\ast}}dx\right)  ^{1-\eta/p}.
\]
Using (\ref{Ee}) and (\ref{estu}), and going at the limit as $\varepsilon
\longrightarrow0,$ (iv) follows for $m<\overline{m}:$
\begin{equation}
\left(  \int_{\Omega}\left\vert \nabla U\right\vert ^{(p-1)Nm/(N-m)}dx\right)
^{(N-m)/Nm}\leqq C\left\Vert F\right\Vert _{L^{m}(\Omega)}\text{ }
\label{estg}%
\end{equation}
$\bullet$ Assume $m\geqq\overline{m},$ $p<N.$ Then $L^{m}(\Omega)\subset
W^{-1,p^{\prime}}(\Omega),$ thus, from uniqueness, $U\in W_{0}^{1,p}(\Omega)$
and it is a variational solution. More precisely, $L^{m}(\Omega)\subset
W^{-1,Nm/(N-m)}(\Omega).$ If $m=\overline{m},$ then $Nm/(N-m)=p^{\prime},$ and
(\ref{estg}) follows. If $m>\overline{m},$ then from \cite{KiZh},
\cite{KiZh2}, $U\in W^{1,\ell}(\Omega)$ with $\ell=(p-1)Nm/(N-m),$ and
(\ref{estg}) still holds, and (iii) follows for $m\geqq\overline{m}$; and (i)
and (ii) from the Sobolev injection$.$ Another proof in case $m>N/p$ is given
in \cite{P}.

\noindent$\bullet$ Assume $m>1$ and $p=N.$ then again $L^{m}(\Omega)\subset
W^{-1,Nm/(N-m)}(\Omega),$ hence (\ref{estg}) still holds, and then $U\in
L^{\infty}(\Omega).\medskip$
\end{proof}

\begin{proof}
[Proof of Proposition \ref{cig}]If $p=N,$ then $U\in L^{\sigma}(\Omega)$ for
any $\sigma\geqq1$ from Remark \ref{estlk}, then $f(x)\left\vert U\right\vert
^{Q}\in L^{m}(\Omega)$ for any $m\in\left(  1,r\right)  ,$ then $U\in
W_{0}^{1,N}(\Omega)\cap L^{\infty}(\Omega)$ from Lemma \ref{boot}, and
$\left\vert \nabla U\right\vert ^{N-1}\in L^{\tau}(\Omega)$ for $\tau
=Nm/(N-m)$.

Next suppose $p<N.$ First assume that $Qr^{\prime}<Q_{1}.$ Let $k\in\left(
0,1\right)  $ such that $1/r^{\prime}>k+Q(N-p)/N(p-1).$ Then $f(x)\left\vert
U\right\vert ^{Q}\in L^{m_{0}}(\Omega)$ with $m_{0}=1/(1-k)>1.$ Taking $k$
small enough, one finds $m_{0}<N/p,$ thus $h(x)\in L^{m_{0}}(\Omega),$ then
from Lemma \ref{boot}, $\left\vert U\right\vert ^{s_{1}}\in L^{1}(\Omega)$
with $s_{1}=(p-1)Nm_{0}/(N-pm_{0}).$ Then $f(x)\left\vert U\right\vert ^{Q}\in
L^{m_{1}}(\Omega),$ where
\[
\frac{1}{m_{1}}-\frac{1}{r}=\frac{Q}{p-1}\left(  \frac{1}{m_{0}}-\frac{p}%
{N}\right)  .
\]
And $1/m_{1}-1/m_{0}<(Q-p+1)(1-m_{0})/m_{0}(p-1)<0,$ hence $m_{1}>m_{0}.$ For
any $n\in\mathbb{N}$ such that $f(x)\left\vert U\right\vert ^{q}\in L^{m_{n}%
}(\Omega)$ and $m_{n}<N/p,$ we can define $m_{n+1}$ by
\[
\frac{1}{m_{n+1}}-\frac{1}{r}=\frac{Q}{p-1}\left(  \frac{1}{m_{n}}-\frac{p}%
{N}\right)  .
\]
and $m_{n}<m_{n+1}.$ If $m_{n}<N/p$ for any $n,$ it has a limit $m$, then
$m=(Q/(p-1)-1)/\left(  Qp^{\prime}/N-1/r\right)  .$ When $Q\geqq p-1,$ then
$m<1,$ which is impossible. Then after a finite number $\bar{n}$ of steps we
arrive to $m_{\bar{n}}>N/p$ , thus (i) follows from Lemma \ref{boot}.\medskip

Next assume $Q>p-1$ and $Qr^{\prime}=Q_{1},$ thus $p<N,$ and $\left\vert
U\right\vert ^{p-1}\in L^{\sigma}(\Omega)$ for some $\sigma>N/(N-p).$ Setting
$\sigma=(1+\theta)N/(N-p)$ with $\theta>0,$ and $m_{0}=(1+\theta
r)/(1+\theta)>1,$ there holds $Qr/(r-m_{0})=\sigma,$ and from H\"{o}lder
inequality:
\[
\int_{\Omega}(f\left\vert U\right\vert ^{q})^{m_{0}}dx\leqq\left(
\int_{\Omega}f^{r}dx\right)  ^{m_{0}/r}\left(  \int_{\Omega}\left\vert
U\right\vert ^{qr/(r-m_{0})}dx\right)  ^{1-m_{0}/r},
\]
thus we still have $f(x)\left\vert U\right\vert ^{Q}\in L^{m_{0}}(\Omega)$ and
$1/m_{1}-1/m_{0}=(Q-p+1)(1-m_{0})/m_{0}(p-1)<0$, thus (ii) follows as
above.\medskip

And (iii) follows from \cite[Propositions 1.2 and 1.3]{GuVe}. Indeed the
equation can be written under the form
\[
-\Delta_{p}U=K(x)(1+\left\vert U\right\vert ^{p-1}),
\]
where $\left\vert K(x)\right\vert \leqq f(x)(1+\left\vert U\right\vert
^{Q-p+1});$ if $(Q+1)r^{\prime}\leqq p^{\ast}$ then $K(x)\in L^{s}(\Omega)$
for some $s\geqq N/p,$ then $U\in L^{\infty}(\Omega)$ if the inequality is
strict, and $U\in L^{k}(\Omega)$ for any $k\geqq1$ in case of
equality.\medskip

Next assume $Q<p-1.$ Then $\left\vert h(x)\right\vert $ $\leqq f(x)(\left\vert
U\right\vert ^{p-1}+2),$ hence (iv) holds from above. If $r=N/p,$ then again
$Qr^{\prime}<Q_{1},$ and we find $m=N/p.$ Then $h(x)\in L^{s}(\Omega)$ for any
$s<N/p,$ hence $U\in W_{0}^{1,p}(\Omega)$ and (v) holds from Lemma \ref{boot}.
If $r<N/p$ then $m<N/p.$ Thus from Lemma \ref{boot} $U^{k}\in L^{1}(\Omega)$
for any $k<(p-1)Nm/(N-pm)=\theta.$ If $(Q+1)r^{\prime}<Q_{1}$, then
$m>\overline{m},$ thus $U\in W_{0}^{1,p}(\Omega)$. If ($Q+1)r^{\prime}\geqq
Q_{1},$ then $m\leqq\overline{m},$ thus $\left\vert \nabla U\right\vert
^{p-1}\in L^{\tau}(\Omega)$ for any $\tau<Nm/(N-m)=\theta.$Then (vi) follows.
\medskip
\end{proof}

\begin{proof}
[Proof of Lemma \ref{secm}]In \cite[Proposition 2.1]{B-VPo}, we have given the
estimates (\ref{B}) for the superharmonic continuous functions in
$\mathbb{R}^{N}$. In fact they adapt to any local renormalized solution of the
equation in $\Omega$. Indeed such a solution satisfies $U^{p-1}\in
L_{loc}^{\sigma}(\Omega)$ for any $\sigma\in\left(  0,N/(N-p)\right)  .$ Let
$x_{0}\in\Omega$ and $\rho>0$ such that $B(x_{0},4\rho)\subset\Omega.$ Let
$\varphi_{\rho}=\xi_{\rho}^{\lambda}$ with $\lambda>0$ large enough, and
$\xi_{\rho}(x)=\zeta(\left\vert x-x_{0}\right\vert /\rho),$ where $\zeta
_{\rho}\in$ $\mathcal{D}(\mathbb{R})$ with values in $\left[  0,1\right]  ,$
such that $\xi(t)=1$ for\ $\left\vert t\right\vert \leq1,0$ for\ $\left\vert
t\right\vert \geq2.$ Let $\sigma\in\left(  1,N/(N-p)\right)  $ and $\alpha
\in\left(  1-p,0\right)  .$ We set $U_{\varepsilon}=U+\varepsilon,$ for any
$\varepsilon>0.$ Let $k>\varepsilon.$ Then we can take
\[
\phi=T_{k}(U_{\varepsilon})^{\alpha}\xi_{\rho}^{\lambda}%
\]
as a test function, where $\lambda>0$ large enough will be fixed after. Hence
\begin{align*}
&  \int_{\Omega}FT_{k}(U_{\varepsilon})^{\alpha}\xi_{\rho}^{\lambda
}dx+\left\vert \alpha\right\vert \int_{\Omega}T_{k}(U_{\varepsilon}%
)^{\alpha-1}\xi_{\rho}^{\lambda}\left\vert \nabla(T_{k}(U)\right\vert ^{p}dx\\
&  \leq\lambda\int_{\Omega}T_{k}(U_{\varepsilon})^{\alpha}\xi_{\rho}%
^{\lambda-1}\left\vert \nabla(T_{k}(U)\right\vert ^{p-2}\nabla(T_{k}%
(U)\nabla\xi_{\rho}dx
\end{align*}%
\[
\leq\frac{\left\vert \alpha\right\vert }{2}\int_{\Omega}T_{k}(U_{\varepsilon
})^{\alpha-1}\xi_{\rho}^{\lambda}\left\vert \nabla(T_{k}(U)\right\vert
^{p}dx+C(\alpha)\int_{\Omega}T_{k}(U_{\varepsilon})^{\alpha+p-1}\xi_{\rho
}^{\lambda-p}\left\vert \nabla\xi_{\rho}\right\vert ^{p}dx.
\]
Hence
\[
\int_{\Omega}FT_{k}(U_{\varepsilon})^{\alpha}\xi_{\rho}^{\lambda}%
dx+\frac{\left\vert \alpha\right\vert }{2}\int_{\Omega}T_{k}(U_{\varepsilon
})^{\alpha-1}\xi_{\rho}^{\lambda}\left\vert \nabla(T_{k}(U)\right\vert
^{p}dx\leq C(\alpha)\int_{\Omega}T_{k}(U_{\varepsilon})^{\alpha+p-1}\xi_{\rho
}^{\lambda-p}\left\vert \nabla\xi_{\rho}\right\vert ^{p}dx.
\]
Then we make $\varepsilon$ tend to $0$ and $k$ to $\infty.$ Setting
$\theta=(p-1)\sigma/(p-1+\alpha)>1,$ we obtain%
\[
\int_{\Omega}FU^{\alpha}\xi_{\rho}^{\lambda}dx+\frac{\left\vert \alpha
\right\vert }{2}\int_{\Omega}U^{\alpha-1}\xi_{\rho}^{\lambda}\left\vert \nabla
U\right\vert ^{p}dx\leq C\left(  \int_{\text{supp }\nabla\zeta}U^{(p-1)\sigma
}\xi_{\rho}^{\lambda}dx\right)  ^{1/\theta}\left(  \int_{\Omega}\xi_{\rho
}^{\lambda-p\theta^{\prime}}\left\vert \nabla\xi_{\rho}\right\vert
^{p\theta^{\prime}}dx\right)  ^{1/\theta^{\prime}}%
\]
with a new constant $C$ depending of $\alpha$ from the H\"{o}lder inequality.
Taking $\lambda$ large enough,
\[
\int_{\Omega}FU^{\alpha}\xi_{\rho}^{\lambda}+\frac{\left\vert \alpha
\right\vert }{2}\int_{\Omega}U^{\alpha-1}\xi_{\rho}^{\lambda}\left\vert \nabla
U\right\vert ^{p}\leq C\rho^{N/\theta^{\prime}-p}\left(  \int_{\text{supp
}\nabla\xi_{\rho}}U^{(p-1)\sigma}\xi_{\rho}^{\lambda}dx\right)  ^{1/\theta}%
\]
Next we take $\phi=\xi_{\rho}^{\lambda}$ as a test function. We get
\begin{align*}
\int_{\Omega}F\xi_{\rho}^{\lambda}dx  &  \leq\lambda\int_{\Omega}\xi_{\rho
}^{\lambda-1}\left\vert \nabla U\right\vert ^{p-2}\nabla U.\nabla\xi_{\rho
}dx\\
&  \leqq\lambda\left(  \int_{\Omega}U_{\varepsilon}^{\alpha-1}\xi_{\rho
}^{\lambda}\left\vert \nabla U\right\vert ^{p}dx\right)  ^{1/p^{\prime}%
}\left(  \int_{\Omega}U_{\varepsilon}^{(1-\alpha)(p-1)}\xi_{\rho}^{\lambda
-p}\left\vert \nabla\xi_{\rho}\right\vert ^{p}dx\right)  ^{1/p}.
\end{align*}
Since $\ell>p-1,$ we can fix an $\alpha\in\left(  1-p,0\right)  $ such that
$\tau=\sigma/(1-\alpha)>1.$ Then as $\varepsilon\rightarrow0,$
\begin{align*}
\int_{\Omega}F\xi_{\rho}^{\lambda}dx  &  \leq C\left(  \int_{\text{supp
}\nabla\xi_{\rho}}U^{(p-1)\sigma}\xi_{\rho}^{\lambda}dx\right)  ^{1/\theta
p^{\prime}+1/\tau p}\\
&  \times\left(  \int_{\Omega}\xi_{\rho}^{\lambda-\theta^{\prime}p}\left\vert
\nabla\xi_{\rho}\right\vert ^{\theta^{\prime}p}dx\right)  ^{1/\theta^{\prime
}p^{\prime}}\left(  \int_{\Omega}\xi_{\rho}^{\lambda-\tau^{\prime}p}\left\vert
\nabla\xi_{\rho}\right\vert ^{\tau^{\prime}p}dx\right)  ^{1/\tau^{\prime}p}.
\end{align*}
But $1/\theta p^{\prime}+1/\tau p=\sigma=1-(1/\theta^{\prime}p^{\prime}%
+1/\tau^{\prime}p),$ hence
\[
\int_{\Omega}F\xi_{\rho}^{\lambda}dx\leq C\left(  \int_{\text{supp }\nabla
\xi_{\rho}}U^{(p-1)\sigma}\xi_{\rho}^{\lambda}dx\right)  ^{1/\sigma}%
\rho^{N(1-1/\sigma)-p}%
\]
and (\ref{B}) follows. Otherwise, if $U\in W_{loc}^{1,p}(\Omega),$ from the
weak Harnack inequality, there exists a constant $C^{\prime}=C^{\prime}%
(\sigma,N,p)$ such that
\[
\left(  \frac{1}{\rho^{N}}\int_{B(x_{0},2\rho)}U^{(p-1)\sigma}dx\right)
^{1/(p-1)\sigma}\leqq C^{\prime}\inf\text{ess}_{B(x_{0},\rho)}U,
\]
hence (\ref{Ha}) holds by fixing $\sigma.\medskip$
\end{proof}

\begin{proof}
[Proof of Lemma \ref{Ponce}]Let $\hat{h}(x,t)=h(x,\max(0,\min(t,u(x))).$ Then
$0\leqq\hat{h}(x,t)\leqq F(x)$ a.e. in $\Omega.$ From the Schauder theorem for
any $n\in\mathbb{N}$ there exists $V_{n}\in W_{0}^{1,p}(\Omega)\cap L^{\infty
}(\Omega)$ such that
\[
-\Delta_{p}V_{n}=T_{n}(\hat{h}(x,V_{n}))\qquad\text{in }\Omega.
\]
From Remark \ref{conv}, up to a subsequence, $V_{n}$ converges a.e. to a
renormalized solution $V$ of equation
\[
-\Delta_{p}V=\hat{h}(x,V)\qquad\text{in }\Omega,
\]
and $V$ $\geqq0$ from the Maximum Principle. It remains to show that $V\leqq
U.$ For fixed $m>0,$ and $n\in\mathbb{N}$ the function $\omega=T_{m}%
((V_{n}-U)^{+})=T_{m}((V_{n}-T_{\left\Vert V_{n}\right\Vert _{L^{\infty
}\left(  \Omega\right)  }}U)^{+})\in W_{0}^{1,p}(\Omega);$ and $\omega
^{+}=\omega^{-}=0,$ thus from \cite[Definition 2.13]{DMOP}
\[
\int_{\Omega}\left\vert \nabla U\right\vert ^{p-2}\nabla U.\nabla\omega
dx=\int_{\Omega}\omega h(x,U)dx+\int_{\Omega}\omega^{+}d\mu_{s}^{+}%
-\int_{\Omega}\omega^{-}d\mu_{s}^{-}=\int_{\Omega}\omega h(x,U)dx
\]
Otherwise, $\omega$ is also admissible in the equation relative to $V_{n}:$
\[
\int_{\Omega}\left\vert \nabla V_{n}\right\vert ^{p-2}\nabla V_{n}%
.\nabla\omega dx=\int_{\Omega}T_{n}(\hat{h}(x,V_{n}))\omega dx;
\]
then
\begin{align*}
&  \int_{\Omega}\left(  \left\vert \nabla V_{n}\right\vert ^{p-2}\nabla
V_{n}-\left\vert \nabla U\right\vert ^{p-2}\nabla U\right)  .\nabla\left(
T_{m}((V_{n}-U)^{+})\right)  dx\\
&  =\int_{\Omega}(T_{n}(\hat{h}(x,V_{n})-h(x,U))T_{m}((V_{n}-U)^{+}%
)dx\leqq\int_{\Omega}(h(x,U)-T_{n}(h(x,U))T_{m}((V_{n}-U)^{+})dx.
\end{align*}
From the Fatou Lemma and Lebesgue Theorem, going to the limit as
$n\longrightarrow\infty$ for fixed $m,$ since the truncations converge
strongly in $W_{0}^{1,p}(\Omega),$ we deduce
\[
\int_{\Omega}\left(  \left\vert \nabla V\right\vert ^{p-2}\nabla V-\left\vert
\nabla U\right\vert ^{p-2}\nabla U\right)  .\nabla\left(  T_{m}((V-U)^{+}%
)\right)  dx\leqq0.
\]
Then $T_{m}((V-U)^{+}=0$ for any $m>0,$ thus $V\leqq U$ a.e. in $\Omega.$
\end{proof}

\end{document}